\documentclass{amsart}

\usepackage{graphicx} 
\usepackage{hyperref}
\usepackage{amsthm}
\usepackage{amsmath}
\usepackage{amssymb}
\usepackage{thmtools}
\usepackage{thm-restate}
\usepackage[backend=biber,
style=alphabetic,
]{biblatex}
\DeclareFieldFormat{postnote}{#1}
\DeclareFieldFormat{multipostnote}{#1}
\usepackage{tikz}
\usepackage{circuitikz}
\usepackage{subfiles}
\usepackage{caption}
\DeclareCaptionFormat{custom}
{%
    \textbf{#1#2}\textit{\small #3}
}
\theoremstyle{definition}
\newtheorem{counter}{bad bad bad}[section]
\newtheorem{lemma}[counter]{Lemma}
\newtheorem{definition}[counter]{Definition}
\newtheorem{notation}[counter]{Notation}
\newtheorem{remark}[counter]{Remark}
\newtheorem{proposition}[counter]{Proposition}

\newtheorem{fact}[counter]{Fact}
\newtheorem{corollary}[counter]{Corollary}
\newtheorem{hypothesis}[counter]{Hypothesis}
\newtheorem{question}[counter]{Question}

\newtheorem{innercustomthm}{Theorem}
\newenvironment{customthm}[1]
  {\renewcommand\theinnercustomthm{#1}\innercustomthm}
  {\endinnercustomthm}

\newcommand{\K}{\mathbf{K}}
\newcommand{\gS}{\operatorname{gS}}
\newcommand{\LS}{\operatorname{LS}}
\newcommand{\gtp}{\mathbf{gtp}}

\newcommand{\lek}{\leq_{\K}}
\newcommand{\cof}{\text{cf}}
\newcommand{\calt}{\mathcal{T}}
\newcommand{\calc}{\mathcal{C}}
\newcommand{\lesst}{\vartriangleleft}
\newcommand{\mon}{\mathfrak{C}}
\newcommand{\splt}{\operatorname{split}}
\newcommand{\kkappalims}{K_{(\lambda, \geq \kappa)}}
\newcommand{\Kkappalims}{\K_{(\lambda, \geq \kappa)}}

\newbox\noforkbox \newdimen\forklinewidth
\setbox0\hbox{$\textstyle\smile$}
\setbox1\hbox to \wd0{\hfil\vrule width \forklinewidth depth-2pt
 height 10pt \hfil}
\setbox\noforkbox\hbox{\lower 2pt\box1\lower 2pt\box0\relax}
\def\unionstick{\mathop{\copy\noforkbox}\limits}

\def\nonfork_#1{\unionstick_{\textstyle #1}}

\setbox0\hbox{$\textstyle\smile$}
\setbox1\hbox to \wd0{\hfil{\sl /\/}\hfil}
\setbox2\hbox to \wd0{\hfil\vrule height 10pt depth -2pt width
              \forklinewidth\hfil}
\newbox\doesforkbox
\setbox\doesforkbox\hbox{\lower 2pt\box1 \lower 2pt\box2\lower2pt\box0\relax}
\def\nunionstick{\mathop{\copy\doesforkbox}\limits}

\def\fork_#1{\nunionstick_{\textstyle #1}}

\newcommand{\dnf}{\unionstick}

\newcommand{\nf}{\unionstick}
\newcommand{\dnfbold}[4]{#2 \overset{#4}{\underset{#1}{\overline{\nf}}} #3}
\newcommand{\dnfb}{\overline{\dnf}}

\title{Long limit models are isomorphic assuming a splitting-like relation}
\author{Jeremy Beard}

\begin{document}

\maketitle

\begin{abstract}
	We prove the uniqueness of high cofinality limit models in stable abstract elementary classes (AECs) with amalgamation, assuming the existence of a rather weak independence relation.
	
	\begin{customthm}{\ref{main-theorem}}
		Suppose $\K$ is a $\lambda$-stable AEC, where $\LS(\K) \leq \lambda$, $\kappa < \lambda^+$ is regular, and $\K_\lambda$ satisfies the amalgamation property. Let $\K'$ is the class of all $(\lambda, \delta)$-limit models where $\cof(\delta) \geq \kappa$ (or any AC where $\K' \subseteq \K_\lambda$ contains all such $(\lambda, \delta)$-limit models when $\cof(\delta) \geq \kappa$). Suppose also that $\dnf$ is an independence relation on $\K'$ satisfying weak uniqueness, weak existence, universal continuity* in $\K_\lambda$, $(\geq \kappa)$-local character, and $(\lambda, \theta)$-weak non-forking amalgamation in some regular $\theta \in [\kappa, \lambda^+)$.
		
		Let $\delta_1, \delta_2 < \lambda^+$ be limit with $\cof(\delta_l) \geq \kappa$ for $l = 1, 2$. Then for all $M, N_1, N_2 \in \K_\lambda$, if $N_l$ is $(\lambda, \delta_l)$-limit over $M$ for $l = 1, 2$, then $N_1 \underset{M}{\cong} N_2$. 
		
		Moreover, if $K_\lambda$ also satisfies the joint embedding property, then for all $N_1, N_2 \in \K_\lambda$, if $N_l$ is $(\lambda, \delta_l)$-limit for $l = 1, 2$, then $N_1 {\cong} N_2$.
	\end{customthm}
	
	This generalises both \cite[Theorem 3.1]{bema} and \cite[Theorem 1.2]{bovan} - the former to apply to independence relations that satisfy much weaker forms of uniqueness, extension, and non-forking amalagamation, and the latter to independence relations other than $\lambda$-non-splitting. As such, this generalises all other positive isomorphism results of limit models known to the author.

\end{abstract}

\section{Introduction}

Due to important connections to Shelah's categoricity conjecture (\cite[Question 6.14(3)]{sh702}), the question of when limit models are isomorphic has historically only be posed in settings surrounding categoricity and superstability in AECs. Limit models behave similarly to saturated models, and since $<\LS(K)$ saturated models are harder to work with in AECs, they become a useful surrogate. 

We say a model $N$ is \emph{universal} over a model $M$ of the same cardinality if for every $N'$ extending $M$, there is a $\K$-embedding $f : N' \rightarrow N$ fixing $M$ (Definition \ref{universal-def}). We say $N$ is a \emph{$(\lambda, \delta)$-limit model over $M$} if $N$ has cardinality $\lambda$ and there is a universal increasing sequence $\langle M_i : i < \delta \rangle$ where $M = M_0$ and $N = \bigcup_{i < \delta} M_i$ (Definition \ref{limit-model-def}), and we call $\delta$ the \emph{length} of the limit. Under categoricity conditions, or more generally in superstable AECs with symmetry, \emph{all} limit models are isomorphic, even over a base model (the original attempt at a proof in a categorical setting in \cite{shvi99} had issues that were corrected and extended by VanDieren over the course of several papers \cite{van06} \cite{gvv} \cite{van16a} \cite{van16b}).

Recently the question of what the spectrum of limit models might look like in a strictly stable setting has been posed and explored. In \cite[Theorem 1.2]{bovan}, Boney and VanDieren showed that, in a nice enough stable AEC, all limits above the local character cardinal for non-splitting are isomorphic. They conjectured that the set of regular cardinals $\mu$ such that the $(\lambda, \mu)$-limit model ($\mu$ is the `length' of the limit) is isomorphic to the $(\lambda, \lambda)$-limit model (the `longest' limit model, if $\lambda$ is regular) might be an interval of regular cardinals, with minimal element given by the least (universal) local character cardinal \cite[Conjecture 1.1]{bovan}. Assuming a relation with stronger properties than $\lambda$-non-splitting and further conditions including $\aleph_0$-tameness, in \cite{bema} the author and Mazari-Armida answered this with a complete description of the isomorphism spectrum: in a nice $\aleph_0$-tame AEC with a well behaved independence relation, two limit models of different regular lengths are isomorphic if and only if both lengths are at least the local character cardinal. In particular, \cite[Conjecture 1.1]{bovan} is true in such an AEC if and only if $\cof(\lambda)$ is above the local character cardinal.

The approaches these papers took to showing the `long limits are isomorphic' both involve developing a theory of towers. In an abstract sense, a tower is an increasing sequence of models with a sequence of named elements, where the elements have some independence over the models coded into the structure between towers. The approach of \cite{bovan} was to use non-splitting as the independence notion, and to code into the tower common models that all towers in the sequence were independent over (so their towers are comprised of \emph{two} sequences of models and a sequence of singletons). On the other hand, the approach of \cite{bema} used towers with a \emph{single} sequence of models and a sequence of singletons, and coded the non-forking properties into the tower ordering - types do not fork over the models from smaller towers (inspired by Vasey's presentation in \cite{vasnotes} and \cite{vas19}). More importantly, the towers from \cite{bema} were compatible with \emph{arbitrary} independence relations satisfying enough nice properties, and in fact only required that the relation be defined on `long' limit models. However, some of the properties the arbitrary independence relation must satisfy are stronger than the known properties of $\lambda$-non-splitting (in particular, only weaker forms of extension and uniqueness are known to hold for $\lambda$-non-splitting), so the result from \cite{bema} cannot be directly applied to $\mu$-non-splitting in their context.

The goal of this paper is to apply an approach to towers inspired by \cite{vas19} and \cite{bema} to arbitrary independence relations with the weaker properties of non-splitting, thus extended all previous `long limit models are isomorphic' results including \cite{bovan}. We prove the following:

\begin{customthm}{\ref{main-theorem}}
    Suppose $\K$ is a $\lambda$-stable AEC, where $\LS(\K) \leq \lambda$, $\kappa < \lambda^+$ is regular, and $\K_\lambda$ satisfies the amalgamation property. Let $\K'$ is the subclass of all $(\lambda, \delta)$-limit models where $\cof(\delta) \geq \kappa$ (or any AC where $\K' \subseteq \K_\lambda$ contains all such $(\lambda, \delta)$-limit models when $\cof(\delta) \geq \kappa$). Suppose also that $\dnf$ is an independence relation on $\K'$ satisfying weak uniqueness, weak existence, universal continuity, $(\geq \kappa)$-local character, and $(\lambda, \theta)$-weak non-forking amalgamation in some regular $\theta \in [\kappa, \lambda^+)$.
    
    Let $\delta_1, \delta_2 < \lambda^+$ be limit with $\cof(\delta_l) \geq \kappa$ for $l = 1, 2$. Then for all $M, N_1, N_2 \in \K_\lambda$, if $N_l$ is $(\lambda, \delta_l)$-limit over $M$ for $l = 1, 2$, then $N_1 \underset{M}{\cong} N_2$. 
    
    Moreover, if $\K_\lambda$ also satisfies the joint embedding property, then for all $N_1, N_2 \in \K_\lambda$, if $N_l$ is $(\lambda, \delta_l)$-limit for $l = 1, 2$, then $N_1 {\cong} N_2$.
\end{customthm}

For context, $(\lambda, \theta)$-weak non-forking amalgamation is effectively a weakening of $(\lambda, \delta)$-symmetry (originally defined as \cite[Definition 2.8]{bovan}) for a general independence relation (see Lemma \ref{symmetry_implies_nfap}).

In particular, since the assumptions of \cite[Theorem 1.2]{bovan} imply our hypotheses hold for $\lambda$-non-splitting, we immediately reobtain \cite[Theorem 1.2]{bovan} as Corollary \ref{bovan-main-corollary}. We also reobtain \cite[Theorem 3.1]{bema}, since it's hypotheses are identical to Theorem \ref{main-theorem} besides the weaker versions of uniqueness, extension, and non-forking amalgamation (see Corollary \ref{bema-weaker-corollary}).

The new setting presents several challenges, the principal of which is having access only to the weaker forms of uniqueness, extension, transitivity, and non-forking amalgamation. This creates roadblocks in many arguments throughout the proof, especially the tower extension lemmas, the preservation of fullness of towers under long unions, and perhaps most of all finding reduced extensions. These `weaker' properties can only be applied when there is sufficient `universal wiggle room' - that is, there often needs to be a universal extension separating the domain of a type and the model it is independent of. To accommodate for this, we introduce extra `universal wiggle room' between the singletons and the models they are independent of in our version of towers. More precisely, rather than having $a_i \in M_{i+1}$ and that $\gtp(a_i/M_i', M_{i+1}')$ $\dnf$-does not fork over $M_i$ in the definition of tower extension as in \cite[3.19]{bema}, we have $a_i \in M_{i+2}$ and that $\gtp(a_i/M_{i+1}', M_{i+2}')$ $\dnf$-does not fork over $M_i$. While this complicates the picture somewhat, with some heavy tweaks many arguments using the broad strokes of \cite{bema} go through.

There is another serious hurdle in this setting - we can only show that the tower ordering $\lesst$ is transitive on \emph{universal} towers (more precisely, the middle tower must be universal). This is problematic, as if $\langle \calt^j : j < \alpha + 1\rangle$ is a $\lesst$-increasing sequence of towers, $\calt^\alpha$ is not universal, and we find $\calt'$ where $\calt^\alpha \lesst \calt'$, it may be possible that $\calt^j \not\lesst \calt'$ for some $j < \alpha$. To get around this, instead of looking for extensions of single towers, we often have to ask when a \emph{$\lesst$-chain} of towers may be extended. We call the nicely behaved $\lesst$-chains that allow this \emph{brilliant} - these are chains of towers respecting the tower ordering where each tower is either universal or a union of the previous cofinality $\geq \kappa$-many towers in the sequence.

One consequence of this is that (because of technical reasons in the proof strategy to show reduced extensions exist), it becomes more useful to consider \emph{reduced chains} of towers, rather than reduced individual towers. These reduced extensions of brilliant chains of towers always exist, but we may have to add many towers to the chain rather than just one as in previous approaches. Again this complicates the picture, but allows something like the arguments from \cite{bema} to be used.

This paper has \ref{applications-section} Sections. Section \ref{preliminaries-section} goes through the necessary background. Section \ref{main-result-section} gives a proof of the main theorem; this is the brunt of the paper, so see the end of Subsection \ref{proof-outline-subsection} for an overview of it's structure (Subsection \ref{subsection-non-alg-jep-nmm} is not essential to understand our main arguments, but discusses the strength of our assumptions). Section \ref{applications-section} shows how the result generalises all results on limit models currently known to the author.

This paper was written while the author was working on a Ph.D. thesis under the direction of Rami Grossberg at Carnegie Mellon University, and the author would like to thank Professor Grossberg for his guidance and assistance in his research in general and in this work specifically. We would also like to thank Jin Li for helpful comments that improved the presentation of this paper.

\section{Preliminaries}\label{preliminaries-section}

We assume some knowledge of the basics of abstract elementary classes (AECs), such as found in \cite{baldwinbook}.

\subsection{Notation}

We use $\K = (K, \lek)$ to denote an abstract elementary class (AEC), or occasionally an abstract class (AC); that is, a class of models closed under isomorphisms, where $\lek$ is a partial ordering respecting isomorphisms and refining the submodel relation. A sub-AC $\K' = (K', \leq_{\K'})$ of $\K$ is an AC where $K' \subseteq K$ and $\leq_{\K'} = \lek \upharpoonright (K')^2$; we write $\K' \subseteq \K$. $\LS(\K)$ denotes the L\"{o}wenheim Skolem cardinal of $\K$. We use $M, N, L$ for models, $|M|$ for the universe of $M$, and $\|M\|$ for it's cardinality. $N_1 \cong N_2$ denotes that an isomorphism from $N_1$ to $N_2$ exists, and $N_1 \underset{M}{\cong} N_2$ that one exists fixing $M$. $a, b, c, d$ will be used for singletons in models. $\gtp(a/M, N)$ is the Galois (orbital) type of $a$ over $M$ in $N$. When $p \in S(M)$ and $q \in S(N)$ with $M \lek N$, we say $p$ \emph{extends} $q$, notated by $p \subseteq q$, if there exist $N' \in \K$ and $a \in N'$ where $q = \gtp(a/N, N')$ and $p = \gtp(a/M, N')$. We use $\gS(M)$ to denote the set of all Galois types over $M$. $\alpha, \beta, \gamma, \delta, \theta, \xi, \sigma, \kappa$ will be ordinals, $\gamma, \delta$ typically for limit ordinals, $\theta, \kappa$ typically regular cardinals. $\kappa$ specifically will normally be a local character cardinal (this mirrors Shelah's $\kappa_r(T)$, as well as \cite[2.1]{bovan} and \cite[2.4]{vas18}).

If $\K$ is an AC and $\lambda$ a cardinal, we use the notation $\K_\lambda$ to denote the models in $\K$ of size $\lambda$. We say $\K$ is \emph{$\lambda$-stable} if for every $M \in \K_\lambda$, $|\gS(M)| \leq \lambda$. We say $\K$ has the \emph{amalgamation property} (AP) if for every $M, M_1, M_2 \in \K$ where $M \lek M_l$ for $l = 1, 2$, there exist $N \in \K$ and $\K$-embeddings $f_l : M_l \rightarrow N$ fixing $M$. $\K$ has the \emph{joint embedding property} (JEP) if for every $N_1, N_2 \in \K$, there is some $M$ which $\K$-embeds into both. $\K$ has \emph{no maximal models} (NMM) if for every $M \in \K$ there is $N \in \K$ where $M \lek N$ and $M \neq N$. We say $\K$ has the \emph{$\lambda$-amalgamation property} ($\lambda$-AP), the \emph{$\lambda$-joint embedding property} ($\lambda$-JEP), and \emph{$\lambda$-no maximal models} ($\lambda$-NMM) if $\K_\lambda$ satisfies the relevant property.

\subsection{Limit models}

Limit models are a kind of replacement for saturated models, and can be useful in cases where saturation makes less sense (the notion of $\leq \LS(\K)$-saturated models can be harder to work with as there are no models of size $<\LS(\K)$, but this is not a problem for limit models). 

\begin{definition}\label{universal-def}
    Let $\K$ be an AEC, $\lambda \geq \LS(\K)$, and $M, N \in \K_\lambda$. We say $N$ is \emph{universal over $M$} if for every $N' \in \K_\lambda$ where $M \lek N'$, there exists a $\K$-embedding $g : N' \rightarrow N$ fixing $M$. We use the shorthand $M \lek^u N$.
\end{definition}

\begin{definition}\label{limit-model-def}
    \begin{enumerate}
        \item Let $\K$ be an AEC, $M, N \in \K_\lambda$, and $\delta < \lambda^+$ a limit ordinal. We say $N$ is a \emph{$(\lambda, \delta)$-limit model over $M$} if there is a $\lek^u$-increasing sequence $\langle M_i : i \leq \delta\rangle$ such that $M_0 = M$ and $N = M_\delta = \bigcup_{i < \delta} M_i$. We say $\langle M_i : i < \delta \rangle$ \emph{witnesses} the limit.

        \item $N$ is a \emph{$(\lambda, \delta)$-limit model} if $N$ is a $(\lambda, \delta)$-limit model over some $M \in \K_\lambda$. Similarly define \emph{$\lambda$-limit model}. We may omit $\lambda$ when clear from context.

        \item Given $\kappa < \lambda^+$ regular (and $M \in \K_\lambda$), $N$ is a $(\lambda, \geq\kappa)$-limit model (over $M$) if it is a $(\lambda, \delta)$-limit model (over $M$) for some $\delta < \lambda^+$ where $\cof(\delta) \geq \kappa$.
    
        \item Denote the class of all $(\lambda, \geq \kappa)$-limit models by $\kkappalims$, and the AC $\Kkappalims = (\kkappalims, \lek \upharpoonright (\kkappalims)^2)$.
    \end{enumerate}
    
\end{definition}

We will make use of the following fact frequently without explicit mention.

\begin{fact}[Existence of limit models, {\cite[Claim 2.9]{grva06a}}]\label{limit-models-exist}
    Let $\K$ be an AEC stable in $\lambda \geq \LS(\K)$, where $\K_\lambda$ as AP. Then for all $M \in \K_\lambda$, there exists $N \in \K_\lambda$ universal over $M$. Therefore, for all $M \in \K_\lambda$ and limit $\delta < \lambda^+$, there exists $N \in \K_\lambda$ a $(\lambda, \delta)$-limit model over $M$.
\end{fact}

The following holds by a straightforward back and forth argument. A proof when the limits do not share a base model can be found in \cite[2.7]{beard25}\footnote{Though the statement of \cite[2.7]{beard25} assumes $\K$ is an elementary class, the proof works also in our setting.} - when you want to fix the base model, take $M_0 = N_0 = M^*$ to be the base model and $f^* = \operatorname{id}_{M_0}$ in the construction.

\begin{fact}[{\cite[Fact 1.3.6]{shvi99}}]\label{cofinality-isomorphisms}
    Let $\K$ be a an AEC, stable in $\lambda \geq \LS(\K)$ where $\K_\lambda$ has AP, and $\delta_1, \delta_2 < \lambda^+$ limit ordinals where $\cof(\delta_1) = \cof(\delta_2)$. Let $M, N_1, N_2 \in \K_\lambda$ where $N_l$ is a $(\lambda, \delta_l)$-limit model over $M$ for $l = 1, 2$. Then $N_1 \underset{M}{\cong} N_2$.

    Moreover, if $\K_\lambda$ has JEP, $N_1, N_2 \in \K_\lambda$, and $N_l$ is a $(\lambda, \delta_l)$-limit model for $l = 1, 2$ (possibly over different models), then $N_1 \cong N_2$.
\end{fact}

\subsection{Independence relations}

Now we introduce some background of independence relations. These are abstract relations that satisfy some of the useful properties of first order non-forking. Since such relations don't necessarily exist in stable AECs, we assume the existence of an independence relation for our main theorem. These relations exist in many contexts beyond first order, such as $\lambda$-non-splitting (Definition \ref{non-splitting-def}) in the setting of \cite{bovan}, $\lambda$-non-forking in categorical (or more generally $\lambda$-superstable and $\lambda$-symmetric) AECs \cite{vas17unique}, and a host of algebraic examples \cite[6.2, 6.13]{bema}, \cite[3.2]{maro25}.

There are many nice properties we may ask of independence notions, normally some analogue of known properties of forking in (super)stable theories. Many were isolated by Shelah, who first observed that these properties were perhaps more important than the actual definition of forking in the first order setting, and first formalised many of them, in \cite[III]{sh:c}. Other characterisations of non-forking properties with subtle variations, often improving, expanding, or translating to different contexts, appear in \cite{shma285} \cite{sh300}, \cite{shya164}, \cite{bogr}, \cite{van06}, and \cite{lrv}, to name a few. The most intensive collection of such properties might be Shelah's notion of \emph{good frames} \cite{sh:h}, though in more recent years interest has been shown in many weaker frameworks \cite{maz20}, \cite{yang24}, \cite{mvy24}. We will use the same presentation as \cite{bema}, which is largely based on these previous works (see Definition \ref{def-dnf-properties}).

\begin{definition}
    Let $\K$ be an AC. Let $\dnf$ be a relation on tuples $(M_0, a, M, N)$ where $M_0 \lek M \lek N$ and $a \in N$. We write $a \dnf_{M_0}^N M$ as a shorthand for $\dnf(M_0, a, M, N)$. 
    
    We say that $\dnf$ is an \emph{independence relation on $\K$} if $\dnf$ satisfies
    \begin{enumerate}
        \item \emph{Invariance:} if $a \dnf_{M_0}^N M$, and $f : N \rightarrow N'$ is a $\K$-embedding, then $f(a) \dnf_{f[M_0]}^{f[N]} f[M]$ (that is, $\dnf$ is well defined up to isomorphism).
        \item \emph{Monotonicity} if $a \dnf_{M_0}^N M$, $M_0 \lek M_1 \lek M \lek N_0 \lek N \lek N'$, and $a \in N_0$ , then $a \dnf_{M_0}^N M_1$, $a \dnf_{M_0}^{N_0} M$, and $a \dnf_{M_0}^{N'} M$ (we can shrink $M$ and either shrink or grow $N$, so long as we maintain $M_0 \lek M \lek N$ and $a \in N$).
        \item \emph{Base monotonicity} if $a \dnf_{M_0}^{N_0} M$ and $M_0 \lek M_1 \lek M$, then $a \dnf_{M_1}^{N_0} M$ (that is, we can grow $M_0$ so long as is stays below $M$).
    \end{enumerate}
\end{definition}

\begin{definition}
    Let $\dnf$ be an independence relation on an AC $\K$. Suppose $M_0 \lek M$ and $p \in \gS(M)$. We say \emph{$p$ $\dnf$-does not fork over $M_0$} if and only if there exists $N \in \K$ and $a \in N$ such that $p = \gtp(a/M, N)$ and $a \dnf_{M_0}^N M$.
\end{definition}

\begin{remark}
    If $\dnf$ is an independence relation on an AC $\K$, then by invariance and monotonicity, $p$ $\dnf$-does not fork over $M_0$ is equivalent to saying: \emph{for all} $N \in \K$ and $a \in N$ such that $p = \gtp(a/M, N)$, $a \dnf_{M_0}^N M$.
\end{remark}

Independence relations become very useful when we start to add some of the following criteria. We choose to specify explicitly which of these our relations satisfy rather than defining a specific `frame' notion to use throughout the paper, as is often done, since this allows us greater flexibility in how we compare independence relations and how the criteria interact. Some we will not use for a while or are mainly interesting as comparisons to the assumptions in Theorem \ref{main-theorem}, but we collect them all here for convenience.

\begin{definition}\label{def-dnf-properties}
    Let $\dnf$ be an independence relation. Let $\kappa$ be a regular cardinal. Then we say that $\dnf$ satisfies
    \begin{enumerate}
        \item \emph{Weak uniqueness} if for all $M_0 \lek^u M \lek N$ and all $p, q \in \gS(N)$, if $p, q$ both $\dnf$-do not fork over $M_0$ and $p \upharpoonright M = q \upharpoonright M$, then $p = q$.
        \item \emph{Weak extension} if for all $M_0 \lek^u M \lek N$ and all $p \in \gS(M)$, if $p$ $\dnf$-does not fork over $M_0$, then there exists $q \in \gS(N)$ extending $p$ such that $q$ $\dnf$-does not fork over $M_0$.
        \item \emph{Weak transitivity} if for all $M_0 \lek M_1 \lek^u M_2 \lek M_3$ and all $p \in \gS(M_3)$, if $p$ $\dnf$-does not fork over $M_1$ and $p \upharpoonright M_2$ $\dnf$-does not fork over $M_0$, then $p$ $\dnf$-does not fork over $M_0$.
        \item \emph{Uniqueness} if whenever $M \lek N$, $q_1, q_2 \in \gS(N)$ both $\dnf$-do not fork over $M$, and $q_1 \upharpoonright M = q_2 \upharpoonright M$, then $q_1 = q_2$.
        \item \emph{Extension} if whenever $M_0 \lek M \lek N$ and $p \in \gS(M)$ $\dnf$-does not fork over $M_0$, then there is $q \in \gS(N)$ extending $p$ where $q$ $\dnf$-does not fork over $M_0$.
        \item \emph{Transitivity} if whenever $M_0 \lek M \lek N$ and $p \in \gS(N)$ where $p \upharpoonright M$ $\dnf$-does not fork over $M_0$ and $p$ $\dnf$-does not fork over $M$, then $p$ $\dnf$-does not fork over $M_0$.
        \item \emph{$(\geq \kappa)$-universal continuity} if for all $\delta$ limit with $\cof(\delta) \geq \kappa$, all $\lek^u$-increasing sequences of models $\langle M_i : i < \delta \rangle$ where $\bigcup_{i < \delta} M_i \in \K$, for all $p \in \gS(\bigcup_{i<\delta} M_i)$, if $p \upharpoonright M_i$ $\dnf$-does not fork over $M_0$ for all $i \in [1, \delta)$, then $p$ $\dnf$-does not fork over $M_0$.
        \item \emph{Universal continuity} if $\dnf$ satisfies $(\geq \aleph_0)$-universal continuity.
        \item \emph{$(\geq \kappa)$-local character} if for all $\delta$ limit, all $\lek^u$-increasing sequences of models $\langle M_i : i < \delta \rangle$ where $\bigcup_{i < \delta} M_i \in \K$ and $\cof(\delta) \geq \kappa$, then for all $p \in \gS(\bigcup_{i < \delta} M_i)$ there exists $i < \delta$ such that $p$ $\dnf$-does not fork over $M_i$.
        \item \emph{Weak disjointness} if whenever $M_0 \lek^u M \lek N$, $p \in \gS(N)$, $p$ $\dnf$-does not fork over $M_0$, and $p \upharpoonright M$ is non-algebraic, then $p$ is non-algebraic.
        \item \emph{Disjointness} if whenever $M \lek N$, $p \in \gS(N)$, $p$ $\dnf$-does not fork over $M$, and $p \upharpoonright M$ is non-algebraic, then $p$ is non-algebraic.
        \item\label{def-nfap} \emph{Non-forking amalgamation} if whenever $M, M_1, M_2 \in \K$ where $M \lek M_l$ and $a_l \in M_l$ such that $\gtp(a_l/M, M_l)$ $\dnf$-does not fork over $M$ for $l = 1, 2$, then there exist $N \in \K$ and $\K$-embeddings $f_l : M_l \rightarrow N$ fixing $M$ for $l = 1, 2$ such that $\gtp(f_l(a_l)/f_{3-l}[M_{3-l}], N)$ $\dnf$-does not fork over $M$ for $l = 1, 2$.
    \end{enumerate}
\end{definition}

For most of this paper, we will be using the `weak' forms of the above properties. We include the others here so we can compare our main hypotheses with other results (especially in Section \ref{applications-section}). The classic example of an independence relation only known to satisfy these weak properties is $\lambda$-non-splitting.

\begin{definition}[{\cite[3.2]{sh394}}]\label{non-splitting-def}
	Suppose $\K$ is an AEC, and $\lambda \geq \LS(\K)$.
	
	Given $M\in \K_\lambda, N \in \K$ where $M \lek N$ and $p \in \gS(N)$, we say $p$ $\lambda$-splits over $M$ if there exist $N_1, N_2 \in \K_\lambda$ where $M \lek N_1, N_2$ and an isomorphism $f:N_1 \rightarrow N_2$ fixing $M$ such that $f(p \upharpoonright N_1) \neq p \upharpoonright N_2$. Otherwise, we say \emph{$p$ $\lambda$-does not split over $M$}.
	
	The relation $\dnf_{\lambda-\splt}$ is defined by $a \dnf_{\underset{{M_0}}{\lambda-\splt}}^N M$ if and only if $\gtp(a/M, N)$ $\lambda$-does not split over $M_0$.
\end{definition}

Note the independence relation corresponds to $\lambda$-\emph{non}-splitting, rather than $\lambda$-splitting. Like with first order forking, the negation is what we normally interpret as independence.

The following is stated in \cite{van02}, with proof explained in \cite[Fact 3.5]{vas17satfork}.

\begin{fact}\label{non-splitting-properties}
	Suppose $\K$ is an AEC stable in some $\lambda \geq \LS(\K)$. Then $\dnf_{\lambda-\splt}$ is an independence relation satisfying weak extension and weak uniqueness.
\end{fact}

We will sometimes be interested in relations possibly only defined on models in $\kkappalims$ for stable $\lambda$ and regular $\kappa<\lambda^+$. In this context, continuity loses meaning for short limits, so we use an appropriate replacement, \emph{universal continuity*}. 

Before this, we define the circumstance that allows us to go back and forth between types in different ACs.

\begin{definition}
	Let $\K$ be an AC with sub AC $\K'$, such that $\K'$ has AP. We say that $\K'$ \emph{respects types from $\K$} if
	\begin{enumerate}
		\item for all $M \in \K'$, $\Phi_M : \gS_{\K'}(M) \rightarrow \gS_{\K}(M)$ given by $\Phi_M(\gtp_{\K'}(a/M, N)) = \gtp_{\K}(a/M, N)$ for $M \leq_{\K'} N$ and $a \in N$ is a well defined bijection
		\item when $M \leq_{\K'} N$, $p \in \gS_{\K'}(M)$, and $q \in \gS_{\K'}(N)$, then $p \subseteq q$ if and only if $\Phi_M(p) \subseteq \Phi_N(q)$.
	\end{enumerate}
\end{definition}

Essentially this says that we do not mind which AC our types are computed in. This is necessary in Definition \ref{continuity*-def}, since we want to say that the non-forking types in $\K'$ have an extension in $\K$. Because of this we will abuse notation and not distinguish between $p$ and $\Phi_M(p)$, including in the Definition \ref{continuity*-def}.

The following shows that when $\K'$ is between $\K_\lambda$ and $\Kkappalims$, $\K'$ respects types from $\K$. We will work in this context for our main result (see Hypothesis \ref{main_hypothesis}).

\begin{lemma}\label{respects-types-and-embeddings-if-contains-all-high-limits}
	Suppose $\K$ is an AEC and $\lambda \geq \LS(\K)$, that $\K$ is $\lambda$-stable and $\K_\lambda$ has AP, and that $\K'$ is an AC where $\Kkappalims \subseteq \K' \subseteq \K_\lambda$. Then 
	\begin{enumerate}
		\item Embeddings in $\K'$ are just $\K$-embeddings between elements of $\K'$
		\item $\K'$ has AP and respects types from $\K$.
	\end{enumerate} 
	 
\end{lemma}

\begin{proof}
	(1) follows from an application of coherence and the fact that for all $M \in \K_\lambda$ there is $N \in \Kkappalims$, so also in $\K'$, where $M \lek N$. 
	
	To see (2), AP follows from AP in $\K_\lambda$ and then possibly increasing the largest model to one in $\Kkappalims$. 
	
	Now we show show $\K'$ respects types form $\K$. By (1) of the Lemma statement, if $M, N_1, N_2 \in \K'$ and $M \lek N_l, a_l \in N_l$ for $l = 1, 2$, then $\gtp_\K(a_1/M, N_1) = \gtp_\K(a_2/M, N_2)$ if and only if $\gtp_{\K'}(a_1/M, N_1) = \gtp_{\K'}(a_2/M, N_2)$ - we can use the same embeddings (into an element of $\Kkappalims$) to witness both. 
	
	Finally suppose $p \in \gS_{\K'}(M)$ and $q \in \gS_{\K'}(N)$. We must show $p \subseteq q$ if and only if $\Phi_M(p) \subseteq \Phi_N(q)$. First assume $p \subseteq q$. Then we can take $N' \in \K'$ and $a \in N'$ with $\gtp_{\K'}(a/N, N') = q$. Then $p = \gtp_{\K'}(a/M, N')$, and we have $\Phi_M(p) = \gtp_{\K}(a/M, N'))$ and $\Phi_N(q) = \gtp_{\K}(a/M, N'))$, meaning $\Phi_M(p)\subseteq \Phi_N(q)$.
	
	Conversely, assume $\Phi_M(p) \subseteq \Phi_N(q)$. Take $N' \in \K_\lambda$ and $a \in N'$ with $\Phi_N(q) = \gtp_{\K'}(a/N, N')$. So $\Phi_M(p) = \gtp_{\K'}(a/M, N')$. Without loss of generality, $N' \in \Kkappalims$, so $N' \in \K'$. Then since $\Phi_M, \Phi_N$ are bijections, we have $p = \gtp_{\K'}(a/M, N')$ and $q = \gtp_{\K'}(a/M, N')$, meaning $p \subseteq q$ as desired.
\end{proof}

\begin{remark}
	The proof that $p \subseteq q$ implies $\Phi_M(p) \subseteq \Phi_N(q)$ works in general when $K'$ respects types from $\K$. We only used that $\Kkappalims \subseteq \K' \subseteq \K_\lambda$ in the other direction to find $N'\in\Kkappalims$ above $N$.
\end{remark}

Before we define universal continuity*, we will justify why it is a natural replacement for universal continuity.

\begin{remark}\label{cty-star-justification-remark}
	Suppose $\dnf$ is an independence relation on an AEC $\K$ that satisfies weak extension and weak uniqueness. Note then if we have a $\subseteq$-increasing sequence $\langle p_i : i < \delta \rangle$ of types where $p_i \in \gS(M_i)$ $\dnf$-does not fork over $M_0$ and $M_i \lek^u M_{i+1}$ for all $i < \delta$, then by taking an extension $q \in \gS(\bigcup_{i<\delta} M_i)$ of $p_1$ where $q$ $\dnf$-does not fork over $M_0$, weak uniqueness gives that $q$ extends every $p_i$. So there is at least one extension of each $p_i$, and it $\dnf$-does not fork over $M_0$. Universal continuity says that \emph{any} $q' \in \gS(\bigcup_{i<\delta} M_i)$ extending all $p_i$ is a $\dnf$-non-forking (over $M_0$) extension. But then weak uniqueness (using that $q\upharpoonright M_1 = p_1 = q' \upharpoonright M_1$ and $M_0 \lek^u M_1$) guarantees $q = q'$. From this we see that continuity in this context is equivalent to saying that there exists \emph{exactly one} type over $\bigcup_{i<\delta} M_i$ extending every $p_i$.
\end{remark}

Armed with this, we define the following:

\begin{definition}[{\cite[Definition 3.2]{bema}}]\label{continuity*-def}
    Let $\K'$ be a sub-AC of AEC $\K$, and $\dnf$ an independence relation on $\K'$. Assume that $\K'$ has AP and respects types from $\K$. We say $\dnf$ satisfies \emph{universal continuity* in $\K$} if the following holds: 
    
    Suppose $\delta$ is a limit ordinal, $\langle M_i : i < \delta \rangle$ is a $\leq$-increasing sequence in $\K'$, and $N \in \K'$ such that $M_i \lek N$ for all $i < \delta$. Suppose further that $\langle p_i : i < \delta \rangle$ is a $\subseteq$-increasing sequence of types where $p_i \in \gS(M_i)$ such that $p_i$ $\dnf$-does not fork over $M_0$ for all $i < \delta$. Then there exists a unique $p_\delta \in \gS(\bigcup_{i < \delta}M_i)$ such that $p_\delta \upharpoonright M_i = p_i$ for all $i < \delta$. 
\end{definition}

Intuitively, universal continuity* says that the union of non-forking types satisfies the reformulation of universal continuity from Remark \ref{cty-star-justification-remark}, even if it escapes the AC $\K'$ (strictly speaking it is the images of the types $p_i$ under $\Phi_{M_i}$ that are subsets of the type $q$). We impose the condition that $M_i \lek N$ for some $N \in \K'$ so that there is still a non-forking extension of all the types, just higher up in $\K'$ (when $\Kkappalims \subseteq \K' \subseteq \K_\lambda$, this ensures that we do not escape $\K_\lambda$). The upshot is that for any such $N \in \K$ and for any $q \in \gS(N)$ extending $p_i$ for all $i < \delta$ which $\dnf$-does not fork over $M_0$ (exactly one of which exists by weak extension and uniqueness), we get that $p_\delta = q \upharpoonright \bigcup_{i < \delta} M_i$.

\begin{remark}
	In \cite[Definition 3.2]{bema}, the first author neglected to include the `respecting types' assumption. Without this, the definition is flawed, since we cannot necessarily compare types in $\K$ and $\K'$. That said, all instances where universal continuity* was applied in that paper did respect types (in particular it is implied by the other assumptions of \cite[Definition 3.4]{bema} by Lemma \ref{respects-types-and-embeddings-if-contains-all-high-limits}) which was the version of universal continuity* assumed throughout the paper besides \cite[\textsection 3.4]{bema}, where it also holds, so this is a very minor omission and does not affect any of the results.
\end{remark}

\begin{definition}
    Given an independence relation $\dnf$ on an AC $\K$ and a sub-AC $\K' \lek \K$, the \emph{restriction of $\dnf$ to $\K'$} (denoted $\dnf \upharpoonright \K'$) is the restriction of $\dnf$ to tuples $(M_0, a, M, N)$ where all models are in $\K'$.
\end{definition}

\begin{definition}
    Let $\K$ be an AEC stable in $\lambda \geq \LS(\K)$ where $\K_\lambda$ has AP, and $\K'$ an AC where $\Kkappalims \subseteq \K' \subseteq \K$. We say $\dnf$ satisfies \emph{$\Kkappalims$-universal continuity* in $\K$} if the restriction of $\dnf$ to $\Kkappalims$ satisfies universal continuity* in $\K$.
\end{definition}

\begin{remark}
	Note that the relation is still strictly speaking only defined on types in $\K'$, not necessarily on types in $\K$ or $\Kkappalims$. That said, since $\K'$ and $\Kkappalims$ respect types from $\K$, we can think of types over models from $\Kkappalims$ as computed in $\K$, $\Kkappalims$, or $\K'$ interchangeably by Remark \ref{respects-types-and-embeddings-if-contains-all-high-limits}, so it makes sense to think about the restriction $\dnf \upharpoonright \Kkappalims$ as an independence relation on $\Kkappalims$.
\end{remark}

By the same approach as \cite[Lemma 3.11]{bema}, $\Kkappalims$-universal continuity* in $\K$ implies $(\geq \kappa)$-universal continuity, even with our weaker versions of extension and uniqueness.

\begin{lemma}\label{uniform-cty-star-gives-high-univ-cty}
	Let $\K$ be an AEC, and $\K'$ an AC where $\Kkappalims \subseteq \K' \subseteq \K$, and let $\dnf$ be an independence relation on $\K'$ satisfying weak uniqueness, weak extension, and $\Kkappalims$-universal continuity*. Then $\dnf$ satisfies $(\geq \kappa)$-universal continuity.
\end{lemma}

\begin{proof}
	Suppose $\cof(\delta) \geq \kappa$, and $\langle M_i : i < \delta \rangle$ is a $\lek^u$-increasing sequence of models where $M_i \in \K'$ for all $i < \delta$ such that $\bigcup_{i < \delta} M_i \in \K'$ (note this is always the case if $\delta < \lambda^+$ as $\Kkappalims \subseteq \K'$). Suppose further that there exists $p \in \gS(\bigcup_{i<\delta} M_i)$ such that $p \upharpoonright M_i$ $\dnf$-does not fork over $M_0$ for all $i < \delta$. By weak extension, take $q \in \gS(\bigcup_{i < \delta} M_i)$ extending $p\upharpoonright M_1$ such that $q$ $\dnf$-does not fork over $M_0$. It is enough to show $p = q$.
	
	Take a sequence $\langle N_i : i < \delta \rangle$ such that $M_i \lek^u N_i \lek^u M_{i+1}$ and $N_i \in \Kkappalims$. For all $i \in [1, \delta)$, since $p \upharpoonright N_i$ and $q \upharpoonright N_i$ both $\dnf$-do not fork over $M_0$ by monotonicity, $(p \upharpoonright N_i) \upharpoonright M_1 = (q\upharpoonright N_i) \upharpoonright M_1$,  and $M_0 \lek^u M_1 \lek N_i$, we have $p \upharpoonright N_i = q \upharpoonright N_i$ by weak uniqueness. So $p \upharpoonright N_i \subseteq p, q$ for all $i < \delta$. Note that $p \upharpoonright N_i$ $\dnf$-does not fork over $N_0$ by base monotonicity for each $i \in [1, \delta)$. Then $\Kkappalims$-universal continuity* in $\K$ implies there is exactly one type in $\gS(\bigcup_{i < \delta}M_i) = \gS(\bigcup_{i < \delta}N_i)$ extending $p \upharpoonright N_i$ for all $i < \delta$, so we must have that $p=q$ as desired.
\end{proof}

\begin{fact}[{\cite[3.7]{vas17satfork}}]\label{weak-uniq-and-ext-implies-trans}
    Let $\dnf$ be an independence relation on an AC $\K$. If $\dnf$ has weak uniqueness and weak extension, then $\dnf$ satisfies weak transitivity.
\end{fact}

\begin{proof}
    Suppose $M_0 \lek M_1 \lek^u M_2 \lek M_3$ and $p \in \gS(M_3)$ where $p$ $\dnf$-does not fork over $M_1$ and $p \upharpoonright M_2$ $\dnf$-does not fork over $M_0$. Since $M_0 \lek^u M_2 \lek M_3$ and $p \upharpoonright M_2$ $\dnf$-does not fork over $M_0$, there is $q \in \gS(M_3)$ extending $p \upharpoonright M_2$ which $\dnf$-does not fork over $M_0$. By monotonicity, $q$ $\dnf$-does not fork over $M_1$. Using $M_1 \lek^u M_2 \lek M_3$, $p \upharpoonright M_2 = q \upharpoonright M_2$, and that $p, q$ both $\dnf$-do not fork over $M_1$, by weak uniqueness, $p = q$. Therefore $p$ $\dnf$-does not fork over $M_0$.
\end{proof}

We won't need the following lemma for our main argument, but it will be useful in the discussion of Subsection \ref{subsection-non-alg-jep-nmm}. The proof is very similar to \cite[Fact 2.4(8)]{vasnotes}.

\begin{lemma}\label{lemma-weak-uniq-and-ext-gives-disjoint}
	Suppose $\K$ is stable in $\lambda \geq \LS(\K)$, $\kappa < \lambda^+$ regular, and $\dnf$ is an independence relation on $\Kkappalims$ satisfying weak uniqueness, weak extension, and $(\lambda, \geq \kappa)$-local character. Then $\dnf$ satisfies weak disjointness.
\end{lemma}

\begin{proof}
	Let $M_0 \lek^u M \lek N$ be $(\lambda, \geq \kappa)$-limit models, and $p \in \gS(N)$ where $p$ $\dnf$-does not fork over $M$ and $p \upharpoonright M$ is non-algebraic. By taking a $(\lambda, \kappa)$-limit model $M'$ over $M_0$, which is without loss of generality a $\K$-substructure of $M$ since $M_0 \lek^u M$, by replacing $M$ by $M'$ we may assume that $M$ is a $(\lambda, \kappa)$-limit over $M_0$. Similarly by taking $N'$ a $(\lambda, \kappa)$-limit model over $N$ and using weak extension on $p$ to get $p' \in \gS(N')$ extending $p$ and $\dnf$-non-forking over $M_0$, by replacing $N$ and $p$ by $N'$ and $p'$ respectively, we may assume that $N$ is a $(\lambda, \kappa)$-limit model over $M$.
	
	Since $M$ is a $(\lambda, \kappa)$-limit model over $M_0$, there is some $M_1$ such that $M_1$ is a $(\lambda, \kappa)$-limit model over $M_0$ and $M$ is a $(\lambda, \kappa)$-limit model over $M_1$ (take some sequence $\langle M_i : i < \kappa \rangle$ where $M_{i+1}$ is a $(\lambda, \kappa)$-limit model over $M_i$ for all $i$ witnessing that $M$ is a $(\lambda, \kappa)$-limit model over $M_0$, and use $M_1$ from that sequence). By uniqueness of limit models of the same cofinality (Fact \ref{cofinality-isomorphisms}), there exists an isomorphism $f : M \cong N$ fixing $M_1$. Note that $f(p \upharpoonright M)$ and $p$ are both extensions of $p \upharpoonright M_1$ which $\dnf$-do not fork over $M_0$. By weak uniqueness, $f(p \upharpoonright M) = p$. Since $p \upharpoonright M$ is non-algebraic, $p$ is non-algebraic also.
\end{proof}

\subsection{On Symmetry}

Symmetry is another nice property of independence relations, and appears in some form as an assumption in all positive `limit models are isomorphic' results known to the author. Shelah proved the original form in the context of first order stable theories, specifically, that $a \dnf_M Mb$ implies $b \dnf_M Ma$.\footnote{Here $a \dnf_M b$ denotes that $\operatorname{tp}(a/M \cup \{b\}, \mon)$ does not fork over $M$, where $\mon$ is the monster model.} \cite[Theorem III.4.13]{sh:c}. This does not adapt immediately to our notion of independence relation, since our relations only allow \emph{models} on the right. To accommodate for this when defining \emph{good frames}, Shelah used a notion of symmetry that essentially says if $a$ is independent of some model $M^b$ containing $b$ (over $M$), then $b$ is independent of some model $M^a$ containing $a$ (over $M$) \cite[2.1(E)(f)]{sh:h}. This notion was adapted to $\lambda$-non-splitting and called $\lambda$-symmetry, and a hierarchy of similar formulations of symmetry were identified by VanDieren and Vasey, which are often equivalent \cite[\textsection 4]{vanvas}. These additionally require that the models are limit models, and some require that certain models are universal extensions of others, partly because limit models are so well behaved in this setting.

Symmetry often follows from some level of categoricity; in particular, assuming $\K$ has enough AP and NMM, $\lambda$-superstability and $\lambda$-symmetry follow from categoricity in some $\mu > \lambda$ \cite[4.8]{vas17satsol}. $\lambda$-symmetry also follows from $\lambda$-superstability and $\lambda$-tameness \cite[6.9]{vanvas}.

In a $\lambda$-superstable AEC with $\lambda$-symmetry, a very well behaved independence relation on limit models called $\lambda$-non-forking can be defined from $\lambda$-non-splitting. With some additional assumptions including tameness and that $\lambda$ is a successor, we in fact have a good $\lambda$-frame \cite[6.6]{vanvas}, \cite[17.9]{vasnotes}.

Symmetry can also be used to prove non-forking amalgamation \cite[5.2]{vas19}, \cite[2.21]{bema}. Non-forking amalgamation essentially says that we can amalgamate two models with named singletons in such a way that the images of the singletons and the `opposite' models are independent. In fact, non-forking amalgamation can be used instead of symmetry to prove limit models are isomorphic, as in \cite[2.7]{vas19}, \cite[3.1]{bema}, and (as we will later see) Theorem \ref{main-theorem}.

While $\lambda$-symmetry is the natural analogue in $\lambda$-superstable AECs, in the strictly stable setting only high cofinality limit models are well behaved, so the notion must be adapted. Boney and VanDieren defined a further weakening, $(\lambda, \theta)$-symmetry, in \cite[2.8]{bovan} which is more appropriate in this strictly stable context. In this subsection, our goal is to define a version of non-forking amalgamation that follows from their version of symmetry, that will be sufficiently strong to prove the main theorem.

Since we want to avoid assuming the existence of a monster model (in particular that we may not have JEP or NMM), we will have to keep track of the models our types are computed in. This muddies the picture somewhat, particularly the proof of Lemma \ref{symmetry_implies_nfap}, so we provide diagrams where we think they will be useful.

\begin{definition}
    Given $\K$ and AEC stable in $\lambda \geq \LS(\K)$, $\dnf$ an independence relation on some AC $\K'$ where $\Kkappalims \subseteq \K' \subseteq \K_\lambda$, and $\theta < \lambda^+$ limit, we say $\dnf$ satisfies 
    \begin{enumerate}
    	
        \item (\cite[2.8]{bovan}) \emph{$(\lambda, \theta)$-symmetry} if whenever $M_0, M, N, M^a, N^a \in \K'$ where $M$ is a $(\lambda, \theta)$-limit model over $M_0$, $M \lek^u M^a \lek N^a$, $M \lek N \lek N^a$, $a, b \in N$, $a \in M^a$, and $\gtp(a/M, N)$ $\dnf$-does not fork over $M_0$ and $\gtp(b/M^a, N^a)$ $\dnf$-does not fork over $M$, then there exist $M^b, N^b \in \K'$ where $M \lek M^b \lek N^b$ and $N \lek N^b$ such that $\gtp(a/M^b, N^b)$ $\dnf$-does not fork over $M_0$ (see Figure \ref{symmetry_diagram}).
        
        \begin{figure}[!ht]
\centering

\begin{circuitikz}
\tikzstyle{every node}=[font=\normalsize]
\node [font=\normalsize] at (7.75,4.75) {$M_0$};
\node [font=\normalsize] at (8.75,5.75) {$(\lambda, \theta)$-limit};
\node [font=\normalsize] at (6.75,6.9) {$u$};
\node [font=\normalsize] at (7.75,6.75) {$M$};
\node [font=\normalsize] at (5,7.25) {$a$};
\node [font=\normalsize, rotate=20] at (5.3,7.35) {$\in$};
\node [font=\normalsize] at (5.75,7.5) {$M^a$};
\node [font=\normalsize] at (3.45,8.7) {$a \dnf_{M_0}^N M$};
\node [font=\normalsize] at (3.5,7.5) {$b \dnf_{M}^{N^a} M^a$};
\node [font=\normalsize, color=blue] at (10.5,7.25) {$b$};
\node [font=\normalsize, rotate=160, color=blue] at (10.2,7.35) {$\in$};
\node [font=\normalsize, color=blue] at (9.75,7.5) {$M^b$};
\node [font=\normalsize, color=blue] at (12,8.25) {$a \dnf_{M_0}^{N^b} M^b$};
\node [font=\normalsize] at (5.75,9.5) {$N^a$};
\node [font=\normalsize, color=blue] at (9.75,9.5) {$N^b$};
\node [font=\normalsize] at (7.75,9.5) {$a, b$};
\node [font=\normalsize, rotate=90] at (7.75,9.15) {$\in$};
\node [font=\normalsize] at (7.75,8.75) {$N$};
\draw [->, >=Stealth] (7.75,5) -- (7.75,6.5);
\draw [->, >=Stealth, color=blue, dashed] (8,6.75) -- (9.5,7.5);
\draw [->, >=Stealth] (7.75,7) -- (7.75,8.5);
\draw [->, >=Stealth, color=blue, dashed] (9.75,7.75) -- (9.75,9.25);
\draw [->, >=Stealth, color=blue, dashed] (8,8.75) -- (9.5,9.5);
\draw [->, >=Stealth] (7.5,6.75) -- (6,7.5);
\draw [->, >=Stealth] (5.75,7.75) -- (5.75,9.25);
\draw [->, >=Stealth] (7.5,8.75) -- (6,9.5);
\end{circuitikz}

\caption{$(\lambda, \theta)$-symmetry says that, given the solid black diagram, the dotted blue diagram exists.}
\label{symmetry_diagram}
\end{figure}
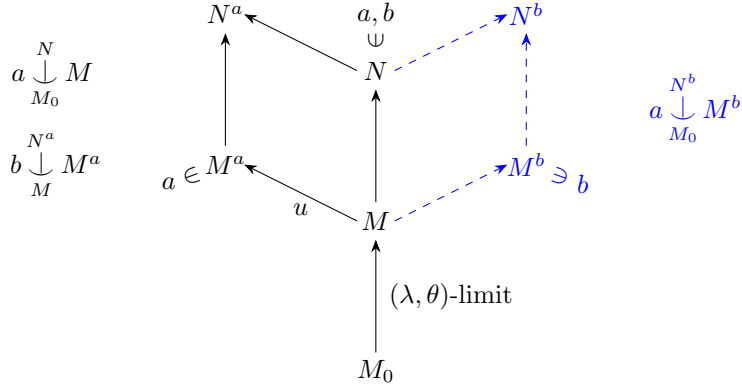
        
        \item \emph{$(\lambda, \theta)$-weak non-forking amalgamation} if for every $M_0, M, M_1, M_2 \in \K'$ where $M$ is a $(\lambda, \theta)$-limit model over $M_0$, and $M \lek M_l$ and $a_l \in M_l$ for $l = 1, 2$ such that $\gtp(a_l /M, M_l)$ $\dnf$-does not fork over $M_0$, there exist $N \in \K'$ and $f_l :M_l \underset{M_0}{\rightarrow} N$ $\K$-embeddings such that $\gtp(f_l(a_l) / f_{3-l}[M_{3-l}], N)$ $\dnf$-does not fork over $M_0$, for $l = 1, 2$ (see Figure \ref{weak_non_forking_amalgamation}).
        
        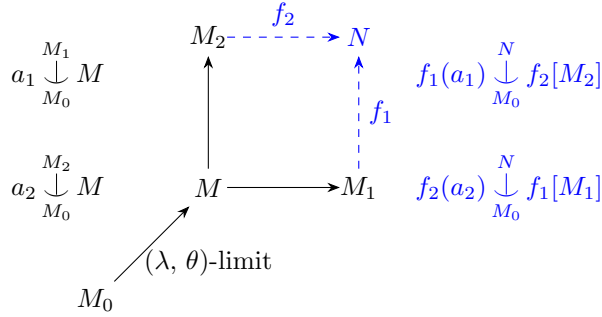
\begin{figure}[!ht]
\centering
\begin{circuitikz}
\tikzstyle{every node}=[font=\normalsize]
\node [font=\normalsize] at (4,11.5) {$a_1 \dnf_{M_0}^{M_1} M$};
\node [font=\normalsize] at (4,10) {$a_2 \dnf_{M_0}^{M_2} M$};
\node [font=\normalsize] at (4.5,8.5) {$M_0$};
\node [font=\normalsize] at (6,10) {$M$};
\draw [->, >=Stealth] (4.75,8.75) -- (5.75,9.75);
\draw [->, >=Stealth] (6,10.25) -- (6,11.75);
\draw [->, >=Stealth] (6.25,10) -- (7.75,10);
\node [font=\normalsize] at (8,10) {$M_1$};
\node [font=\normalsize] at (6,12) {$M_2$};
\node [font=\normalsize, color=blue] at (8.3,11) {$f_1$};
\node [font=\normalsize, color=blue] at (7,12.3) {$f_2$};
\draw [->, >=Stealth, dashed, color=blue] (6.25,12) -- (7.75,12);
\draw [->, >=Stealth, dashed, color=blue] (8,10.25) -- (8,11.75);
\node [font=\normalsize, color=blue] at (8,12) {$N$};
\node [font=\normalsize, color=blue] at (10,11.5) {$f_1(a_1) \dnf_{M_0}^N f_2[M_2]$};
\node [font=\normalsize, color=blue] at (10,10) {$f_2(a_2) \dnf_{M_0}^N f_1[M_1]$};
\node [font=\normalsize] at (6,9) {($\lambda$, $\theta$)-limit};
\end{circuitikz}

\caption{$(\lambda, \delta)$-weak non-forking amalgamation says that, given the solid black diagram, the dotted blue diagram exists.}
\label{weak_non_forking_amalgamation}
\end{figure}
    \end{enumerate}
\end{definition}

\begin{remark}
	The definition of $(\lambda, \theta)$-symmetry is equivalent if we assume $N = N^a$, which in some ways simplifies the picture, but we feel our presentation makes clearer how the situation is `symmetric' for $a$ and $b$.
\end{remark}

\begin{remark}
    In \cite{bovan}, a seemingly stronger form of symmetry than this is used. Their version says $M^b$ is limit over $M$ \cite[Definition 2.8]{bovan}. However, if $\K$ has AP and $\dnf$ satisfies weak uniqueness and weak extension, then with little work these are equivalent to an even nicer statement: we may assume the model $M^b$ is a $(\lambda, \kappa)$-limit model over $M$. That said, we won't actually need either of these `stronger' forms; we only need that $M \lek M^b$.
\end{remark}

The following argument is similar to \cite[Lemma 16.2]{vasnotes}, but we avoid using a monster model.

\begin{lemma}\label{symmetry_implies_nfap}
    Let $\K$ be an AEC where $\K_\lambda$ has AP, and let $\K'$ be an AC where $\Kkappalims \subseteq \K' \subseteq \K_\lambda$. If $\dnf$ is an independence relation on $\K$ satisfying weak extension, weak uniqueness, and $(\lambda, \theta)$-symmetry, then $\dnf$ satisfies $(\lambda, \theta)$-weak non-forking amalgamation.
\end{lemma}

\begin{proof}
    Note we can be ambiguous with where types are computed and in which AC our maps are considered embeddings by Lemma \ref{respects-types-and-embeddings-if-contains-all-high-limits}.

    Suppose $M_0, M, M_1, M_2 \in \K'$ where $M$ is a $(\lambda, \theta)$-limit model over $M_0$, and $M \lek M_l$ and $a_l \in M_l$ such that $\gtp(a_l /M, M_l)$ $\dnf$-does not fork over $M_0$ for $l = 1, 2$. Without loss of generality, $M_0 \lek^u M_1$ (replace $M_1$ by any $(\lambda, \kappa)$-limit over the original $M_1$). By weak extension, there are some $\bar{M}_1 \in \K'$ and some $a_2^* \in \bar{M}_1$ such that $M_1 \lek \bar{M}_1$, $\gtp(a_2^*/M_1, \bar{M}_1)$ extends $\gtp(a_2/M, M_2)$, and $\dnf$-does not fork over $M_0$. As $\gtp(a_2^*/M, \bar{M}_1) = \gtp(a_2/M, M_2)$, there exist $M_2^* \in \K'$ and a $\K$-embedding $g_1 : M_2 \rightarrow M_2^*$ fixing $M$ such that $\bar{M}_1 \lek M_2^*$ and $g_1(a_2) = a_2^*$.

    We have that $M$ is a $(\lambda, \theta)$-limit model over $M_0$, $M \lek^u M_1 \lek M_2^*$, $M \lek M_2^* \lek M_2^*$, $a_1, a_2^* \in M_2^*$, $a_1 \in M_1$, $\gtp(a_1/M, M_1)$ $\dnf$-does not fork over $M$ (by base monotonicity) and $\gtp(a_2^*/M, M_0)$ $\dnf$-does not fork over $M_0$. Therefore by $(\lambda, \theta)$-symmetry, there exist $L_2, L_2^* \in \K'$ where $M \lek L_2 \lek L_2^*$, and $M^*_2 \lek L_2^*$ such that $\gtp(a_1/L_2, L_2^*)$ $\dnf$-does not fork over $M_0$ (see Figure \ref{symm_app_in_symm_implies_nfap}). 

    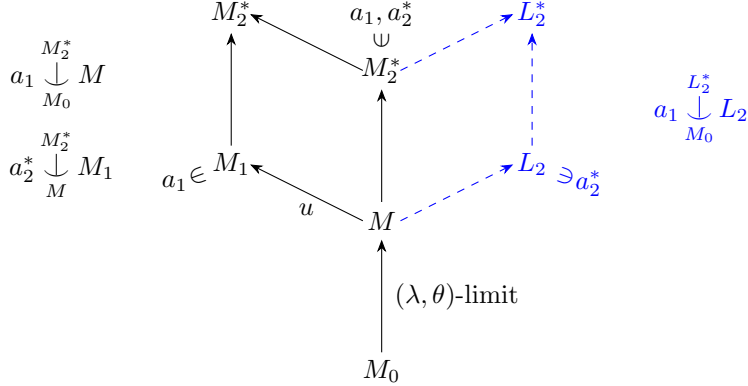
\begin{figure}[!ht]
\centering
\begin{circuitikz}
\tikzstyle{every node}=[font=\normalsize]
\node [font=\normalsize] at (7.75,4.75) {$M_0$};
\node [font=\normalsize] at (8.75,5.75) {$(\lambda, \theta)$-limit};
\node [font=\normalsize] at (6.75,6.9) {$u$};
\node [font=\normalsize] at (7.75,6.75) {$M$};
\node [font=\normalsize] at (5.75,7.5) {$M_1$};
\node [font=\normalsize] at (5,7.25) {$a_1$};
\node [font=\normalsize, rotate=20] at (5.3,7.35) {$\in$};
\node [font=\normalsize, color=blue] at (9.75,7.5) {$L_2$};
\node [font=\normalsize, color=blue] at (10.5,7.25) {$a_2^*$};
\node [font=\normalsize, rotate=160, color=blue] at (10.2,7.35) {$\in$};
\node [font=\normalsize] at (5.75,9.5) {$M_2^*$};
\node [font=\normalsize, color=blue] at (9.75,9.5) {$L_2^*$};
\node [font=\normalsize] at (7.75,8.75) {$M_2^*$};
\draw [->, >=Stealth] (7.75,5) -- (7.75,6.5);
\draw [->, >=Stealth, color=blue, dashed] (8,6.75) -- (9.5,7.5);
\draw [->, >=Stealth] (7.75,7) -- (7.75,8.5);
\draw [->, >=Stealth, color=blue, dashed] (9.75,7.75) -- (9.75,9.25);
\draw [->, >=Stealth, color=blue, dashed] (8,8.75) -- (9.5,9.5);
\draw [->, >=Stealth] (7.5,6.75) -- (6,7.5);
\draw [->, >=Stealth] (5.75,7.75) -- (5.75,9.25);
\draw [->, >=Stealth] (7.5,8.75) -- (6,9.5);

\node [font=\normalsize, color=blue] at (12,8.25) {$a_1 \dnf_{M_0}^{L_2^*} L_2$};
\node [font=\normalsize] at (3.45,8.7) {$a_1 \dnf_{M_0}^{M_2^*} M$};
\node [font=\normalsize] at (3.5,7.5) {$a_2^* \dnf_{M}^{M_2^*} M_1$};
\node [font=\normalsize] at (7.75,9.5) {$a_1, a_2^*$};
\node [font=\normalsize, rotate=90] at (7.75,9.15) {$\in$};

\end{circuitikz}

\caption{The application of symmetry in Lemma 
\ref{symmetry_implies_nfap}}
\label{symm_app_in_symm_implies_nfap}
\end{figure}

    As $\gtp(a_1/L_2, L_2^*)$ $\dnf$-does not fork over $M_0$, by weak extension there exist $L_3^* \in \K'$ and $a_1^* \in L_3^*$ such that $\gtp(a_1^*/L_2^*, L_3^*)$ $\dnf$-does not fork over $M_0$ and extends $\gtp(a_1/L_2, L_2^*)$. By type equality, there is some $N \in \K'$ and some $g_2 : L_3^* \rightarrow N$ fixing $L_2$ such that $L_2^* \lek N$ and $g_2(a_1^*) = a_1$. Figure \ref{symmetry_implies_nfap_diagram} shows the system we've constructed.
    
    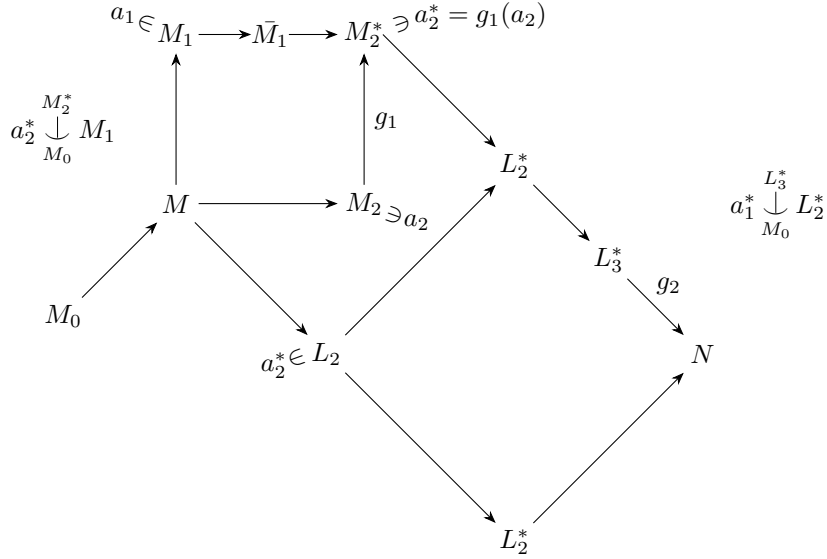
\begin{figure}[!ht]
\centering
\begin{circuitikz}
\tikzstyle{every node}=[font=\normalsize]
\node [font=\normalsize] at (2.5,8.5) {$M_0$};
\node [font=\normalsize] at (4,10) {$M$};
\node [font=\normalsize] at (4,12.25) {$M_1$};
\node [font=\normalsize] at (3.3,12.5) {$a_1$};
\node [font=\normalsize, rotate=-20] at (3.6,12.4) {$\in$};
\node [font=\normalsize] at (5.25,12.25) {$\bar{M_1}$};
\node [font=\normalsize] at (6.5,12.25) {$M_2^*$};
\node [font=\normalsize] at (8.05,12.5) {$a_2^* = g_1(a_2)$};
\node [font=\normalsize, rotate=200] at (7,12.4) {$\in$};
\node [font=\normalsize] at (6.5,10) {$M_2$};
\node [font=\normalsize] at (7.2,9.75) {$a_2$};
\node [font=\normalsize, rotate=160] at (6.9,9.85) {$\in$};
\node [font=\normalsize] at (8.5,10.5) {$L_2^*$};
\node [font=\normalsize] at (6,8) {$L_2$};
\node [font=\normalsize] at (5.3,7.85) {$a_2^*$};
\node [font=\normalsize, rotate=20] at (5.6,7.95) {$\in$};
\draw [->, >=Stealth] (2.75,8.75) -- (3.75,9.75);
\draw [->, >=Stealth] (4.3,10) -- (6.15,10);
\draw [->, >=Stealth] (4,10.25) -- (4,12);
\draw [->, >=Stealth] (6.5,10.25) -- (6.5,12);
\draw [->, >=Stealth] (4.3,12.25) -- (5,12.25);
\draw [->, >=Stealth] (5.5,12.25) -- (6.15,12.25);
\draw [->, >=Stealth] (6.75,12.25) -- (8.25,10.75);
\draw [->, >=Stealth] (4.25,9.75) -- (5.75,8.25);
\draw [->, >=Stealth] (6.25,8.25) -- (8.25,10.25);
\node [font=\normalsize] at (6.8,11.1) {$g_1$};
\node [font=\normalsize] at (9.75,9.25) {$L_3^*$};
\node [font=\normalsize] at (8.5,5.5) {$L_2^*$};
\node [font=\normalsize] at (10.55,8.9) {$g_2$};
\node [font=\normalsize] at (11,8) {$N$};
\node [font=\normalsize] at (2.5,11) {$a_2^* \dnf_{M_0}^{M_2^*} M_1$};
\node [font=\normalsize] at (12,10) {$a_1^* \dnf_{M_0}^{L_3^*} L_2^*$};

\draw [->, >=Stealth] (8.75,10.25) -- (9.5,9.5);
\draw [->, >=Stealth] (10,9) -- (10.75,8.25);
\draw [->, >=Stealth] (8.75,5.75) -- (10.75,7.75);
\draw [->, >=Stealth] (6.25,7.75) -- (8.25,5.75);

\end{circuitikz}

\caption{The construction in Lemma \ref{symmetry_implies_nfap}}
\label{symmetry_implies_nfap_diagram}
\end{figure}

    Now take $f_1:M_1 \rightarrow N$ to be the identity on $M_1$, and $f_2 : M_2 \rightarrow N$ to be $f_2 = (g_2 \upharpoonright M_2^*) \circ g_1$. Note that $f_1, f_2$ both fix $M$. Also, \begin{align*}
    	\gtp(f_1(a_1)/f_2[M_2], N) &= \gtp(a_1/g_2 \circ g_1[M_2], N)\\
    	&\subseteq \gtp(a_1/g_2[L_2^*], N) \\
    	&= g_2(\gtp(a_1^*/L_2^*, L_3^*))
    \end{align*} 
    the latter of which $\dnf$-does not fork over $M_0$ by invariance, so $\gtp(f_1(a_1)/f_2[M_2], N)$ $\dnf$-does not fork over $M_0$ by monotonicity. And \begin{align*}
        \gtp(f_2(a_2)/f_1[M_1], N) &= \gtp(g_2(g_1(a_2))/M_1, N)\\
        &= \gtp(g_2(a_2^*)/M_1, L_2^*)\\
        &= \gtp(a_2^*/M_1, M_2^*)
    \end{align*}
    The final equality holds since $a_2^* \in L_2$ and $g_2$ fixes $L_2$. By our earlier work, this type $\dnf$-does not fork over $M_0$. Thus $N, f_1, f_2$ satisfy the criteria of $(\lambda, \theta)$-weak non-forking amalgamation, as desired.
\end{proof}

\begin{question}
	Does the converse hold? More precisely, if $\K$ is an AEC where $\K_\lambda$ has AP, $\K'$ is an AC where $\Kkappalims \subseteq \K' \subseteq \K_\lambda$, and $\dnf$ is an independence relation on $\K$ satisfying weak extension and weak uniqueness, does $(\lambda, \theta)$-weak non-forking amalgamation imply $(\lambda, \theta)$-symmetry?
\end{question}

\section{Main result}\label{main-result-section}

We now state the main theorem, which is our goal for this section.

\begin{restatable}{theorem}{maintheorem}\label{main-theorem}
    Suppose $\K$ is a $\lambda$-stable AEC, where $\LS(\K) \leq \lambda$, $\K_\lambda$ has AP, and $\kappa < \lambda^+$ is regular. Let $\K'$ be an AC where $\Kkappalims \subseteq \K' \subseteq \K_\lambda$. Suppose $\dnf$ is an independence relation on $\K'$ satisfying weak uniqueness, weak existence, $\Kkappalims$-universal continuity* in $\K$, $(\geq \kappa)$-local character, and $(\lambda, \theta)$-weak non-forking amalgamation in some regular $\theta \in [\kappa, \lambda^+)$. 
    
    Let $\delta_1, \delta_2 < \lambda^+$ be limit with $\cof(\delta_l) \geq \kappa$ for $l = 1, 2$. Then for all $M, N_1, N_2 \in \K_\lambda$, if $N_l$ is $(\lambda, \delta_l)$-limit over $M$ for $l = 1, 2$, then $N_1 \underset{M}{\cong} N_2$. Moreover, if $\K_\lambda$ has JEP, then for all $N_1, N_2 \in \K_\lambda$, if $N_l$ is $(\lambda, \delta_l)$-limit for $l = 1, 2$, then $N_1 {\cong} N_2$.
\end{restatable}

\begin{remark}\label{assume-K-prime-is-Kkappalims}
	Note if we have $\K', \dnf$ as above, then replacing them with $\Kkappalims, \dnf \upharpoonright \Kkappalims$ respectively, the hypotheses still hold.
\end{remark}

With this in mind, for this section, we assume the slightly simpler assumptions of Hypothesis \ref{main_hypothesis} throughout.

\begin{hypothesis}\label{main_hypothesis}
    $\K$ is a $\lambda$-stable AEC, where $\LS(\K) \leq \lambda$ and $\K_\lambda$ has AP, and $\kappa < \lambda^+$ is regular. $\dnf$ is an independence relation on $\Kkappalims$ satisfying weak uniqueness, weak existence, universal continuity* in $\K$, $(\geq \kappa)$-local character, and $(\lambda, \theta)$-weak non-forking amalgamation in some regular $\theta \in [\kappa, \lambda^+)$.
\end{hypothesis}

\subsection{An outline of the proof}\label{proof-outline-subsection}

Before we get into the weeds, let's outline the broad strokes of our strategy. Note by Fact \ref{cofinality-isomorphisms} it is enough to find for any $M$ a model $N$ that is both a $(\lambda, \delta_1)$-limit model and a $(\lambda, \delta_2)$-limit model over $M$.

We will define a notion of \emph{tower} $\calt$, essentially a sequence of models with nice non-forking properties, and an ordering on towers $\lesst$. We will construct a $\lesst$-chain of these towers $\langle \calt^j : j \leq \alpha \rangle$ where $\cof(\alpha) = \cof(\delta_1)$ such that 
\begin{enumerate}
    \item $\calt^\alpha = \bigcup_{j < \alpha} \calt^j$ (see Definition \ref{tower-union-def})
    \item the ordering of $\calt^j$ contains a copy of $\delta_2$
    \item $\calt^\alpha$ is universal and continuous when we restrict to the copy of $\delta_2$
\end{enumerate}

This effectively says that the chain of towers (when we throw away the models not in the copy of $\delta_2$) is a matrix of models where the final column is the union of the previous columns, is a universal sequence, and is continuous at $\delta_2$ (see Figure \ref{construction_overview}).

\begin{figure}[!ht]
\centering
\begin{circuitikz}

\def\uppersemi at (#1,#2){\def\Radius{0.75}
    
  \draw
    (\Radius + #1, #2) arc(0:180:\Radius);}

\def\sidesemi at (#1,#2){\def\Radius{0.75}
    
  \draw
    (#1, #2 - \Radius) arc(-90:90:\Radius);}

\def\fullcircle at (#1,#2){\def\Radius{0.75}
    
  \draw
    (\Radius + #1, #2) arc(0:360:\Radius);}

\def\upperright at (#1,#2){\def\Radius{0.75}
    
  \draw
    (\Radius + #1, #2) arc(0:90:\Radius);}
    
\def\upperleft at (#1,#2){\def\Radius{0.75}
    
  \draw
    (#1, \Radius + #2) arc(90:180:\Radius);}

\def\threequarter at (#1,#2){\def\Radius{0.75}
    
  \draw
    (#1, #2 - \Radius) arc(-90:180:\Radius);}

\tikzstyle{every node}=[font=\normalsize]

\fullcircle at (3.75,6.5);
\uppersemi at (3.75,7.25);
\uppersemi at (3.75,8);
\uppersemi at (3.75,8.75);
\uppersemi at (3.75,9.5);
\uppersemi at (3.75,10.25);

\uppersemi at (3.75,13.5);

\draw [short] (3,6.5) -- (3,13.5);
\draw [short] (4.5,6.5) -- (4.5,13.5);
\node [font=\normalsize] at (3.75,6.5) {$M^0_0$};
\node [font=\normalsize] at (3.75,7.5) {$M^0_1$};
\node [font=\normalsize] at (3.75,8.25) {$M^0_2$};
\node [font=\normalsize] at (3.75,9) {$M^0_3$};
\node [font=\normalsize] at (3.75,9.75) {$M^0_4$};
\node [font=\normalsize] at (3.75,10.5) {$M^0_5$};
\node [font=\normalsize] at (3.75,12) {$\vdots$};
\node [font=\normalsize] at (3.75,13.65) {$M^0_{\delta_2}$};
\node [font=\normalsize] at (3.75,14.75)    {$\calt_0$};

\sidesemi at (5 ,6.5);
\upperright at  (5 ,7.25);
\upperright at  (5 ,8);
\upperright at  (5 ,8.75);
\upperright at  (5 ,9.5);
\upperright at  (5 ,10.25);

\upperright at  (5 ,13.5);

\draw [short] (5.75,6.5) -- (5.75,13.5);
\node [font=\normalsize] at (5.15 ,6.5)  {$M^1_0$};
\node [font=\normalsize] at (5.15 ,7.5)  {$M^1_1$};
\node [font=\normalsize] at (5.15 ,8.25) {$M^1_2$};
\node [font=\normalsize] at (5.15 ,9)    {$M^1_3$};
\node [font=\normalsize] at (5.15 ,9.75) {$M^1_4$};
\node [font=\normalsize] at (5.15 ,10.5) {$M^1_5$};
\node [font=\normalsize] at (5.15,12) {$\vdots$};
\node [font=\normalsize] at (5.15 ,13.65) {$M^1_{\delta_2} $};
\node [font=\normalsize] at (5.15,14.75)    {$\calt_1$};

\sidesemi at    (6.25 ,6.5);
\upperright at  (6.25 ,7.25);
\upperright at  (6.25 ,8);
\upperright at  (6.25 ,8.75);
\upperright at  (6.25 ,9.5);
\upperright at  (6.25 ,10.25);

\upperright at  (6.25 ,13.5);

\draw [short] (7  ,6.5) -- (7  ,13.5);
\node [font=\normalsize] at (6.4 ,6.5)  {$M^2_0$};
\node [font=\normalsize] at (6.4 ,7.5)  {$M^2_1$};
\node [font=\normalsize] at (6.4 ,8.25) {$M^2_2$};
\node [font=\normalsize] at (6.4 ,9)    {$M^2_3$};
\node [font=\normalsize] at (6.4 ,9.75) {$M^2_4$};
\node [font=\normalsize] at (6.4 ,10.5) {$M^2_5$};
\node [font=\normalsize] at (6.4,12) {$\vdots$};
\node [font=\normalsize] at (6.4 ,13.65) {$M^2_{\delta_2} $};
\node [font=\normalsize] at (6.4,14.75)    {$\calt_2$};

\sidesemi at    (7.5  ,6.5);
\upperright at  (7.5  ,7.25);
\upperright at  (7.5  ,8);
\upperright at  (7.5  ,8.75);
\upperright at  (7.5  ,9.5);
\upperright at  (7.5  ,10.25);

\upperright at  (7.5  ,13.5);

\draw [short] (8.25,6.5) -- (8.25,13.5);
\node [font=\normalsize] at (7.65 ,6.5)  {$M^3_0$};
\node [font=\normalsize] at (7.65 ,7.5)  {$M^3_1$};
\node [font=\normalsize] at (7.65 ,8.25) {$M^3_2$};
\node [font=\normalsize] at (7.65 ,9)    {$M^3_3$};
\node [font=\normalsize] at (7.65 ,9.75) {$M^3_4$};
\node [font=\normalsize] at (7.65 ,10.5) {$M^3_5$};
\node [font=\normalsize] at (7.65 ,12) {$\vdots$};

\node [font=\normalsize] at (7.65 ,13.65) {$M^3_{\delta_2} $};
\node [font=\normalsize] at (7.65,14.75)    {$\calt_3$};

\sidesemi at    (8.75 ,6.5);
\upperright at  (8.75 ,7.25);
\upperright at  (8.75 ,8);
\upperright at  (8.75 ,8.75);
\upperright at  (8.75 ,9.5);
\upperright at  (8.75 ,10.25);

\upperright at  (8.75 ,13.5);

\draw [short] (9.5,6.5) -- (9.5,13.5);
\node [font=\normalsize] at (8.9 ,6.5)  {$M^4_0$};
\node [font=\normalsize] at (8.9 ,7.5)  {$M^4_1$};
\node [font=\normalsize] at (8.9 ,8.25) {$M^4_2$};
\node [font=\normalsize] at (8.9 ,9)    {$M^4_3$};
\node [font=\normalsize] at (8.9 ,9.75) {$M^4_4$};
\node [font=\normalsize] at (8.9 ,10.5) {$M^4_5$};
\node [font=\normalsize] at (8.9,12) {$\vdots$};

\node [font=\normalsize] at (8.9 ,13.65) {$M^4_{\delta_2} $};
\node [font=\normalsize] at (8.9,14.75)    {$\calt_4$};

\sidesemi at    (11.25 ,6.5);
\upperright at  (11.25 ,7.25);
\upperright at  (11.25 ,8);
\upperright at  (11.25 ,8.75);
\upperright at  (11.25 ,9.5);
\upperright at  (11.25 ,10.25);

\upperright at  (11.25 ,13.5);

\draw [short] (12,6.5) -- (12,13.5);
\node [font=\normalsize] at (10.8 ,6.5)  {$\dots$};
\node [font=\normalsize] at (10.8 ,7.6)  {$\dots$};
\node [font=\normalsize] at (10.8 ,8.35) {$\dots$};
\node [font=\normalsize] at (10.8 ,9.1)    {$\dots$};
\node [font=\normalsize] at (10.8 ,9.85) {$\dots$};
\node [font=\normalsize] at (10.8 ,10.6) {$\dots$};
\node [font=\normalsize, rotate=90] at (10.65 ,11.9) {$\ddots$};

\node [font=\normalsize] at (10.8 ,13.65) {$\dots$};
\node [font=\normalsize] at (10.8,14.75)    {$\dots$};

\sidesemi at    (12.5 ,6.5);
\upperright at  (12.5 ,7.25);
\upperright at  (12.5 ,8);
\upperright at  (12.5 ,8.75);
\upperright at  (12.5 ,9.5);
\upperright at  (12.5 ,10.25);

\upperright at  (12.5 ,13.5   );

\draw [short] (13.25,6.5) -- (13.25,13.5);
\node [font=\normalsize] at (12.65 ,6.5)  {$M^\alpha_0$};
\node [font=\normalsize] at (12.65 ,7.5)  {$M^\alpha_1$};
\node [font=\normalsize] at (12.65 ,8.25) {$M^\alpha_2$};
\node [font=\normalsize] at (12.65 ,9)    {$M^\alpha_3$};
\node [font=\normalsize] at (12.65 ,9.75) {$M^\alpha_4$};
\node [font=\normalsize] at (12.65 ,10.5) {$M^\alpha_5$};
\node [font=\normalsize] at (12.65 ,12) {$\vdots$};

\node [font=\normalsize] at (12.65 ,13.65) {$M^\alpha_{\delta_2} $};
\node [font=\normalsize] at (12.65,14.75)    {$\calt^\alpha$};


\draw [short] (3.75,5.75) -- (12.5,5.75);
\draw [short] (3.75,7.25) -- (12.5,7.25);
\draw [short] (3.75,8   ) -- (12.5,8   );
\draw [short] (3.75,8.75) -- (12.5,8.75);
\draw [short] (3.75,9.5 ) -- (12.5,9.5 );
\draw [short] (3.75,10.25)-- (12.5,10.25);
\draw [short] (3.75,11)   -- (12.5,11);

\draw [short] (3.75,14.25)-- (12.5,14.25);

\end{circuitikz}

\caption{The chain we will construct, looking only at models for rows in the `copy of $\delta_2$'}
\label{construction_overview}
\end{figure}
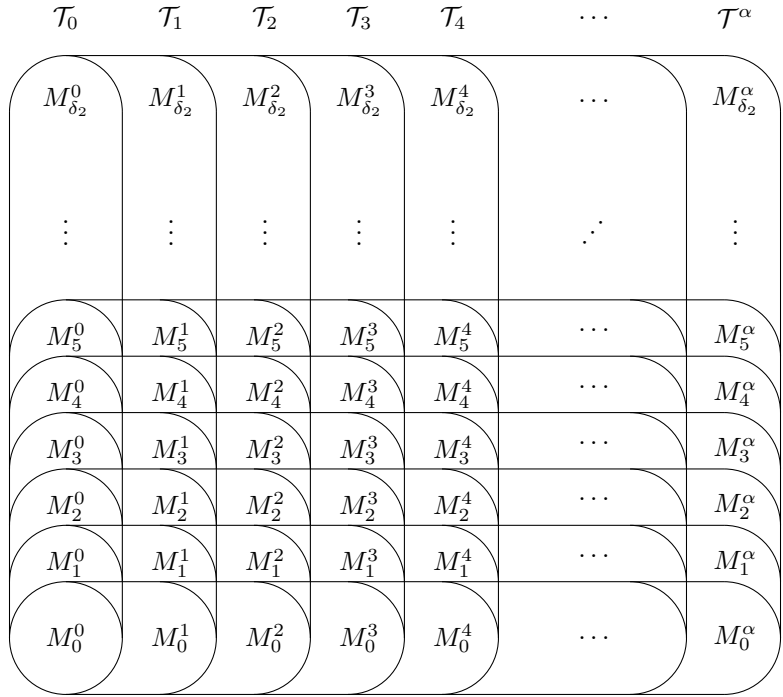

Using the notations in Figure \ref{construction_overview}, since the `top right' model $M^\alpha_{\delta_2}$ is the union of the top row, which the tower ordering $\lesst$ will ensure is a universal sequence, and $\cof(\alpha) = \cof(\delta_1)$, it is a $(\lambda, \delta_1)$-limit model over $M^0_{\delta_2}$, and hence over $M_0^0$. Similarly as the rightmost column is universal and continuous at $\delta_2$, $M^\alpha_{\delta_2}$ is also a $(\lambda, \delta_2)$-limit model over $M_0^0$. We will have sufficient control to ensure $M \leq M_0^0$, so this is enough.

The difficulty will come in showing this construction is possible. We summarise the machinery we will develop to do this below.

\begin{enumerate}
    \item We must show that sequences of towers can be extended, so we can `add columns' to the chain - we will show this is true of what we call \emph{brilliant chains of towers} in Subsection \ref{extending-towers-and-brilliant-chains-subsection}, and ensure the chains we build are always brilliant.
    
    \item We will show that we can take (high cofinality) unions of towers to get the final tower $\calt^\alpha$, and that these unions respect the tower ordering (see Lemma \ref{tower-union-lemma}).
    
    \item We need a way to ensure the final tower is universal. Since a union of universal chains may not be universal, we use the stronger notion of a \emph{full tower}, which is preserved by (high cofinality) unions.
    
    In fact, fullness will imply the copy of $\delta_2$ within the final tower is universal, by effectively inserting enough realisations of types in the models between those in the copy of $\delta_2$.
    
    \item We need a way to guarantee the final tower is continuous at (the copy of) $\delta_2$. We develop a stronger notion that continuity, called a \emph{reduced chain of towers}, which are preserved by high cofinality unions. We address these in Subsection \ref{reduced-subsection}. 
    
    Note we only define reduced \emph{chains} rather than towers, unlike the approaches in \cite{bema} and \cite{bovan}; this is related to us only knowing that $\lesst$ is transitive on universal towers, and how this complication affects the method of finding reduced extensions (Lemma \ref{reduced-extensions-exist}).
\end{enumerate}

\subsection{Comments on non-algebraic types, JEP, and NMM}\label{subsection-non-alg-jep-nmm}

In this subsection, we will explore the strength of the assumptions in Hypothesis \ref{main_hypothesis}. None of the results here are needed to prove Theorem \ref{main-theorem}, so it is safe for a reader in a hurry to jump to Subsection \ref{subsection-towers-and-brilliant-chains}, but we hope this sets the scene a little.

Limit models are most intuitive when we assume JEP and NMM in $\K_\lambda$ (in addition to AP). Under those assumptions, if $N$ is a limit model over $M$, then $N$ is in most senses we can imagine as much bigger than $N$ as it could possibly be - for example, $|N| \setminus |M|$ has cardinality $\lambda$, $N$ realises $\lambda$ distinct non-algebraic types over $M$, and so on. That said, neither JEP or NMM is used for the main arguments of this paper (though JEP is needed for the `moreover' part of Theorem \ref{main-theorem}, which follows quickly from the first part of the statement). This is partly because we assume our independence relation is defined on all types, rather than just non-algebraic ones - otherwise in some constructions we might run into trouble with types over maximal models, as there are no non-algebraic types to find non-forking extensions of. But without JEP and NMM, some of our intuition goes away. That said, we will show that we can assume JEP and NMM without loss of generality, if we want to, after we explore how our relation compares to a non-algebraic one.

\subsubsection{Non-algebraic independence relations, and why we won't use them}

Hypothesis \ref{main_hypothesis} rules out the possibility that our relation is defined only on non-algebraic types (in particular, local character implies that algebraic types over limit models are non-forking over a smaller model, as shown below). While one might think this could restrict the cases when our hypotheses apply, we will show in fact that if we had a relation only defined for non-algebraic types with the usual reformulation of local character, then we can extend it to a relation on all types, so in fact we have lost no generality with this assumption. We formalise this below.

\begin{definition}
	Let $\dnf$ be an independence relation on an AC $\K$. $\dnf$ is \emph{non-algebraic} if for all $M \lek N$ and $p \in \gS(N)$, if $p$ $\dnf$-does not fork over $M$, then $p$ is non-algebraic.
\end{definition}

That is, $\dnf$ is effectively a relation on non-algebraic types. Put another way, $\dnf$ is non-algebraic if and only if whenever $a \dnf_{M_0}^N M$, then $a \notin M$.

When working with a non-algebraic independence relation, most of our notions of nice properties of independence relations (uniqueness, existence, transitivity, non-forking amalgamation, and their `weak' forms, universal continuity, universal continuity*) behave exactly as we should expect with unaltered definitions. The exception is local character, which we must restrict to only hold for non-algebraic types over a union of models. This is because if we assume $\delta$-local character with the normal definition, then taking $p \in \gS(\bigcup_{i<\delta} M_i)$ to be an algebraic type, we get that $p$ $\dnf$-does not fork over $M_i$ for some $i<\delta$, which implies that $\dnf$ is \emph{not} non-algebraic. The following common reformulation of local character avoids this issue - essentially the only difference is that the local character condition need only hold for the non-algebraic types.

\begin{definition}
	Let $\K$ be an AC with non-algebraic independence relation $\dnf$. Given a limit ordinal $\delta$, we say that $\dnf$ satisfies \emph{non-algebraic $\delta$-local character} if for all $\lek^u$-increasing sequences of models $\langle M_i : i < \delta \rangle$ where $\bigcup_{i < \delta} M_i \in \K$ and $\cof(\delta) \geq \kappa$, for all non-algebraic $p \in \gS(\bigcup_{i < \delta} M_i)$, there exists $i < \delta$ such that $p$ $\dnf$-does not fork over $M_i$.
	
	Given $\kappa$ regular, we say $\dnf$ satisfies \emph{non-alebraic $(\geq \kappa)$-local character} if $\dnf$ satisfies non-algebraic $\delta$-local character for all $\delta$ with $\cof(\delta) \geq \kappa$.
\end{definition}

Given a non-algebraic independence relation, we can `fill in' the algebraic part in a simple way. We will see this retains the important properties, and transforms non-algebraic local character into our normal local character notion, under our other assumptions (see Proposition \ref{non-alg-indep-rel-properties-iff-alg}).

\begin{definition}
	Suppose $\dnf$ is a non-algebraic independence relation on an AC $\K$. Define the \emph{algebraic closure} of $\dnf$ to be the relation $\dnfb{}{}{}{}$, given by:
	
	$\dnfbold{M_0}{a}{M}{N}$ if and only if either $a \dnf_{M_0}^N M$, or $a \in M_0$.
\end{definition}

\begin{remark}
	Given a non-algebraic independence relation $\dnf$ on an AC $\K$, $\dnfb$ will be satisfy disjointness - if $p \upharpoonright M$ is non-algebraic and $p \dnfb$-does not fork over $M$, then we must have $p$ $\dnf$-does not fork over $M$, which implies that $p$ is non-algebraic also.
\end{remark}

We can also `toss away' the algebraic part of a relation, and again, normally this preserves our properties (see Proposition \ref{non-alg-indep-rel-properties-iff-alg}).

\begin{definition}
	Suppose $\dnf$ is an independence relation on an AC $\K$. Define the \emph{non-algebraic restriction} of $\dnf$ to be the relation $\dnf^{\circ}$ given by:
	
	$a \dnf^{\underset{\circ}{N}}_{M_0} M$ if and only if $a \dnf_{M_0}^N M$ and $a \notin M$.
\end{definition}

\begin{remark}
	For any independence relation $\dnf$, $\dnf^\circ$ satisfies disjointness (since there are no algebraic $\dnf^\circ$-non-forking types), and $\overset{\circ}{\overline{\left(\dnf^\circ\right)}} = \dnf^\circ$ (since we are removing all algebraic types, adding some, and then removing all of them again - the non-algebraic types are never changed). In particular, if $\dnf$ is non-algebraic, then $\dnfbold{}{}{}{\circ} = \dnf$.
\end{remark}

\begin{proposition}\label{non-alg-indep-rel-properties-iff-alg}
	Suppose $\dnf$ is an independence relation on an AC $\K'$. Then
	\begin{enumerate}
		\item \begin{enumerate}
			\item If $\dnf$ is non-algebraic, then $\dnfb{}{}{}{}$ is an independence relation (that is, it satisfies invariance, monotonicity, and base monotonicity)
			\item $\dnf^\circ$ is an independence relation
		\end{enumerate}
		
		\item \begin{enumerate}
			\item If $\dnf$ is non-algebraic and satisfies weak uniqueness, then $\dnfb$ satisfies weak uniqueness.
			\item If $\dnf$ satisfies weak uniqueness, then $\dnf^\circ$ satisfies weak uniqueness.
		\end{enumerate}  
		
		\item \begin{enumerate}
			\item If $\dnf$ is non-algebraic and satisfies weak extension, then $\dnfb{}{}{}{}$ satisfies weak extension
			\item If $\dnf$ satisfies weak extension and weak disjointness, then $\dnf^\circ$ satisfies weak extension
		\end{enumerate}
		
		\item \begin{enumerate}
			\item If $\dnf$ is non-algebraic and satisfies uniqueness, then $\dnfb{}{}{}{}$ satisfies uniqueness
			\item If $\dnf$ satisfies uniqueness, then $\dnf^\circ$ satisfies uniqueness
		\end{enumerate}
		
		\item \begin{enumerate}
			\item If $\dnf$ is non-algebraic and satisfies extension, then $\dnfb{}{}{}{}$ satisfies extension
			\item If $\dnf$ satisfies extension and disjointness, then $\dnf$ satisfies extension
		\end{enumerate}
		
		\item \begin{enumerate}
			\item If $\delta$ is a limit ordinal and $\dnf$ is non-algebraic and satisfies $\delta$-universal continuity, weak existence, and weak uniqueness, then $\dnfb{}{}{}{}$ satisfies $\delta$-universal continuity
			\item If $\delta$ is a limit ordinal and $\dnf$ satisfies $\delta$-universal continuity, then $\dnf^\circ$ satisfies $\delta$-universal continuity
		\end{enumerate}
		
		\item \begin{enumerate}
			\item If $\K$ is an AEC with sub-AC $\K'$ which respects types in $\K$, and $\dnf$ satisfies universal continuity* in $\K$, weak existence, and weak uniqueness, then $\dnfb{}{}{}{}$ satisfies universal continuity* in $\K$
			\item If $\K$ is an AEC with sub-AC $\K'$ which respects types in $\K$, and $\dnf$ satisfies universal continuity* in $\K$, weak existence, and weak uniqueness, then $\dnf^\circ$ satisfies universal continuity* in $\K$
		\end{enumerate}
		
		\item \begin{enumerate}
			\item If $\K$ is an AEC stable in $\lambda \geq \LS(\K)$, $\K_\lambda$ has AP, $\kappa \leq \theta < \lambda^+$ are regular, $\Kkappalims \subseteq \K' \subseteq \K_\lambda$, and $\dnf$ is non-algebraic and satisfies $(\lambda, \theta)$-weak non-forking amalgamation, then $\dnfb{}{}{}{}$ satisfies $(\lambda, \theta)$-weak non-forking amalgamation.
			\item If $\K$ is an AEC stable in $\lambda \geq \LS(\K)$, $\K_\lambda$ has AP, $\kappa \leq \theta < \lambda^+$ are regular, $\Kkappalims \subseteq \K' \subseteq \K_\lambda$, and $\dnf$ satisfies $(\lambda, \theta)$-weak non-forking amalgamation, then $\dnf^\circ$ satisfies $(\lambda, \theta)$-weak non-forking amalgamation.
		\end{enumerate}
		
		\item \begin{enumerate}
			\item If $\K$ is an AEC stable in $\lambda \geq \LS(\K)$, $\K_\lambda$ has AP, $\kappa < \lambda^+$ is regular, with $\Kkappalims \subseteq \K' \subseteq \K_\lambda$, and $\dnf$ is algebraic and satisfies non-forking amalgamation, then $\dnfb{}{}{}{}$ satisfies non-forking amalgamation.
			\item If $\K$ is an AEC stable in $\lambda \geq \LS(\K)$, $\K_\lambda$ has AP, $\kappa < \lambda^+$ is regular, with $\Kkappalims \subseteq \K' \subseteq \K_\lambda$, and $\dnf$ is algebraic and satisfies non-forking amalgamation, then $\dnf^\circ$ satisfies non-forking amalgamation.
		\end{enumerate}
		
		\item \label{non-alg-indep-rel-properties-iff-alg-local-char}\begin{enumerate}
			\item If $\kappa$ is regular and $\dnf$ is non-algebraic and satisfies non-algebraic $(\geq \kappa)$-local character, then $\dnfb$ satisfies $(\geq \kappa)$-local character.
			\item If $\kappa$ is regular and $\dnf$ satisfies $(\geq \kappa)$-local character, then $\dnf^\circ$ satisfies non-algebraic $(\geq \kappa)$-local character.
		\end{enumerate}
		
	\end{enumerate}
\end{proposition}

\begin{proof}
	All of the (b) cases are simpler, so let's deal with those first. (1), (2), (4), (6), (7), (8), (9), (10) are all immediate for (b) (all types involved will be non-algebraic, and using the properties from $\dnf$ gets the job done). For (3) and (5), just note that the extension you get from $\dnf$ will be non-algebraic by weak disjointness and disjointness respectively.
	
	Now we deal with the (a) cases. Note these all assume that $\dnf$ is non-algebraic so that $\dnfb$ is well defined. We work through the clauses (1) - (10).
	\begin{enumerate}
		\item Invariance follows from the definition of $\dnfb{}{}{}{}$ and invariance of $\dnf$. 
		
		For monotonicity, suppose $M_0 \lek M \lek M' \lek N_0 \lek N \lek N'$ and $\dnfbold{M_0}{a}{M}{N}$. If $a \in M_0$ we immediately have $\dnfbold{M_0}{a}{M'}{N}$, $\dnfbold{M_0}{a}{M}{N_0}$, $\dnfbold{M_0}{a}{M}{N'}$ by definition. If $a \notin M_0$, we have $a\dnf_{M_0}^{N} M$, so $a\dnf_{M_0}^{N} M'$, $a\dnf_{M_0}^{N_0} M$, $a\dnf_{M_0}^{N'} M$ by monotonicity of $\dnf$, which is enough. 
		
		For base monotonicity, suppose $M_0 \lek M_0' \lek M \lek N$ and $\dnfbold{M_0}{a}{M}{N}$. If $a \in M_0$, then $a \in M_0'$, so  $\dnfbold{M_0'}{a}{M}{N}$. Otherwise $a\dnf_{M_0}^{N} M$, so base monotonicity of $\dnf$ gives $a\dnf_{M_0'}^{N} M$, which is enough.
		\item Suppose $M_0 \lek^u M \lek N$ and $p, q \in S(N)$ $\dnfb{}{}{}{}$-do not fork over $M_0$ with $p \upharpoonright M = q \upharpoonright M$. There are 2 cases: if $p \upharpoonright M_0$ is algebraic, then $p$ and $q$ are both extensions of the same algebraic type, so $p = q$. If $p \upharpoonright M_0$ is not algebraic, then $p$, $q$ $\dnf$-do not fork over $M_0$. By weak uniqueness of $\dnf$, $p = q$.
		\item Suppose $M_0 \lek^u M \lek N$, $p \in \gS(M)$, and $p$ $\dnf$-does not fork over $M_0$. There are two cases: if $p \upharpoonright M_0$ is algebraic, then take $q\in \gS(N)$ to be the unique (algebraic) extension of $p$ in $\gS(N)$. By definition, $q$ $\dnfb{}{}{}{}$-does not fork over $M_0$ as $q$ is realised in $M_0$. If $p \upharpoonright M_0$ is not algebraic, then $p$ $\dnf$-does not fork over $M_0$. By weak extension of $\dnf$, there exists $q \in \gS(N)$ where $q$ $\dnf$-does not fork over $M_0$. So $q$ $\dnfb{}{}{}{}$-does not fork over $M_0$.
		\item Similar to (2), but take $M_0 = M$.
		\item Similar to (3), but take $M_0 = M$.
		\item First note that by (2), (3), and Fact \ref{weak-uniq-and-ext-implies-trans}, $\dnfb{}{}{}{}$ satisfies weak transitivity. Suppose $\langle M_i : i \leq \delta \rangle$ is a $\lek^u$-increasing sequence in $\K$, continuous at $\delta$, and that $p \in \gS(M_\delta)$ where $p \upharpoonright M_i$ $\dnf$-does not fork over $M_0$ for all $i < \delta$. There are two cases. If $p$ is algebraic, then there is some $i < \delta$ such that $p$ is realised in $M_i$. This means $p$ $\dnfb$-does not fork over $M_i$. Hence $p \upharpoonright M_{i+1}$ $\dnfb$-does not fork over $M_0$ by monotonicity. Since $M_0 \lek M_i \lek^u M_{i+1} \lek M_\delta$, by weak transitivity of $\dnfb{}{}{}{}$, we have that $p$ $\dnfb{}{}{}{}$-does not fork over $M_0$.
		
		If $p$ is non-algebraic, then $p \upharpoonright M_i$ is non-algebraic for all $i$, so $p \upharpoonright M_i$ $\dnf$-does not fork over $M_i$ for all $i<\delta$. By $\delta$-universal continuity of $\dnf$, we have that $p$ $\dnf$-does not fork over $M_0$, so $p$ $\dnfb{}{}{}{}$-does not fork over $M_0$.
		
		\item As in (6), $\dnfb{}{}{}{}$ satisfies weak transitivity. Suppose that $\delta$ is a limit ordinal, $\langle M_i : i < \delta \rangle$ is a $\lek^u$-increasing sequence of models in $\K'$, $p_i \in \gS(M_i)$ and $p_i$ $\dnfb{}{}{}{}$-does not fork over $M_0$ for all $i < \delta$ where $\langle p_i : i < \delta \rangle$ are $\subseteq$-increasing. There are 2 cases. If $p_0$ is algebraic, then let $p_\delta \in \gS(\bigcup_{i < \delta} M_i)$ be the unique (algebraic) extension of $p_0$. It is clear that this is the unique extension of $p_i$ for all $i < \delta$.
		
		If $p_0$ is not algebraic, then $p_i$ $\dnf$-does not fork over $M_0$ for all $i < \delta$. By universal continuity* for $\dnf$, there is a unique $p_\delta \in \gS(\bigcup_{i<\delta} M_i)$ extending all $p_i$ for $i < \delta$.
		
		\item Suppose $M_0 \lek M \lek M_l \lek N$ and $a_l \in M_l$ where $\gtp(a_l/M, M_l)$ $\dnfb{}{}{}{}$-does not fork over $M_0$ for $l = 1, 2$ where $M$ is a $(\lambda, \theta)$-limit model over $M_0$. There are 4 cases.
		
		If $a_1, a_2 \in M$, since $\dnf$ is non-algebraic and $\gtp(a_l/M, N_l))$ $\dnfb$-does not fork over $M_0$, $a_1, a_2 \in M_0$. Take any $N \in \K$ and $f:M_l \rightarrow N$ fixing $M$ by AP. Then $\gtp(f_l(a_l)/f_{3-l}[M_{3-l}], N)) = \gtp(a_l/f_{3-l}[M_{3-l}], N))$ $\dnfb$-does not fork over $M_0$ since $a_l \in M_0$ for $l = 1, 2$.
		
		If $a_1 \in M, a_2 \notin M$, then as above, $a_1 \in M_0$. $a_2 \notin M$ implies that $\gtp(a_2/M, M_2)$ $\dnf$-does not fork over $M_0$. Take some $M_1' \geq_{\K} M_1$ and $d \in M_1'$ where $\gtp(d/M_1, M_1')$ is the $\dnf$-non-forking extension of $\gtp(a_2/M, M_2)$. By type equality, there exist $N \geq_{\K} M_1'$ and $f : M_2 \rightarrow N$ such that $f(a_2) = d$. Then taking $f = \operatorname{id} : M_1 \rightarrow N$ and $f_2 = f : M_2 \rightarrow N$, we have that $\gtp(f_1(a_1)/f_2[M_2], N)$ $\dnfb{}{}{}{}$-does not fork over $M_0$ since $f_1(a_1) = a_1 \in M_0$, and $\gtp(f_2(a_2)/f_1[M_1], N) = \gtp(d/M_1, N)$ $\dnf$-does not fork over $M_0$, and therefore $\dnfb{}{}{}{}$-does not fork over $M_0$, as desired.
		
		If $a_1 \notin M, a_2 \in M$, do the symmetric argument to the previous one.
		
		If $a_1, a_2 \notin M$, then $\gtp(a_l/M, M_l)$ $\dnf$-does not fork over $M_0$ for $l = 1, 2$. By $(\lambda, \theta)$-weak non-forking amalgamation of $\dnf$ there exist $N \in \K$ and $f_l : M_l \rightarrow N$ fixing $M$ for $l = 1, 2$ such that $\gtp(f_l(a_l)/f_{3-l}[M_{3-l}], N))$ $\dnf$-does not fork over $M$. So for $l = 1, 2$, $\gtp(f_l(a_l)/f_{3-l}[M_{3-l}], N))$ $\dnfb{}{}{}{}$-does not fork over $M$ as required.
		
		\item The proof is as in (8), but set $M_0 = M$.
		
		\item Let $\langle M_i : i < \delta \rangle$ be $\lek^u$-increasing where $\cof(\delta) \geq \kappa$ and $p \in gS(\bigcup_{i<\delta} M_i)$. There are two cases. If $p$ is algebraic, then $p$ is realised in some $M_i$ for $i < \delta$. Therefore $p$ $\dnfb$-does not fork over $M_i$.
		
		If $p$ is non-algebraic, then by $(\geq \kappa)$-local character of $\dnf$ there exists $i < \delta$ with $p$ $\dnf$-does not fork over $M_i$. Therefore $p$ $\dnfb$-does not fork over $M_i$, as desired.
	\end{enumerate}
\end{proof}

\begin{corollary}
	Suppose $\K$ is an AEC stable in $\lambda \geq \LS(\K)$, and $\dnf$ is an independence relation on $\Kkappalims$. If Hypothesis \ref{main_hypothesis} holds for $\K, \dnf, \lambda, \kappa, \theta$, then Hypothesis \ref{main_hypothesis} holds for $\K, \dnf^\circ, \lambda, \kappa, \theta$, but with non-algebraic $(\geq \kappa)$-local character replacing $(\geq \kappa)$-local character.
	
	Conversely if $\dnf$ is non-algebraic and Hypothesis \ref{main_hypothesis} holds for $\K, \dnf, \lambda, \kappa, \theta$, but with non-algebraic $(\geq \kappa)$-local character replacing $(\geq \kappa)$-local character, then Hypothesis \ref{main_hypothesis} holds for $\dnfb$.
\end{corollary}

\begin{proof}
	The first part follows from Lemma \ref{lemma-weak-uniq-and-ext-gives-disjoint} and the (b) parts of Proposition \ref{non-alg-indep-rel-properties-iff-alg}. The second part follows from the (a) parts of Proposition \ref{non-alg-indep-rel-properties-iff-alg}
\end{proof}

So, it is just as strong of an assumption to assume a relation defined on algebraic types with full local character, as to assume one on non-algebraic types with only non-algebraic local character. We may as well assume the former then, since it allows us to work with all types, rather than focussing on non-algebraic types, which may not exist under NMM.

\subsubsection{JEP and NMM, and why we could assume them if we wanted to}

Although we never need to use JEP and NMM (in $\lambda$) in our main argument, for those that would prefer to assume them for intuition, we explain why we can add the assumptions JEP and NMM without loss of generality.

Essentially, given $M \in \K$, $\K_M$ is the biggest sub-AC of $\K$ that contains $M$ and satisfies JEP. This technique can be found in \cite[6.11]{vas17univp2}, and is based off of \cite[II.3]{sh:87b}.

\begin{definition}
	Suppose $\K$ is and AC with AP. For $M_1, M_2$, write $M_1 \sim M_2$ if and only if there exist $N \in \K$ and $\K$-embeddings $f_l : M_l \rightarrow N$. This is an equivalence relation by AP.
	
	Given $M \in \K$, let $K_M$ be the equivalence class of $M$ under $\sim$. Let $\K_M$ be the AC with underlying class $K_M$ and with substructure inherited from $\K$.
\end{definition}

\begin{proposition}\label{wlog-jep}
	Suppose $\K, \dnf, \lambda, \kappa, \theta$ satisfy Hypothesis \ref{main_hypothesis}. Then for all $M \in \K_\lambda$, $\K_M$ is an AEC with JEP and $\LS(\K_M) = \LS(\K)$, and $\K_M, \dnf \upharpoonright \K_M, \lambda, \kappa, \theta$ satisfy Hypothesis \ref{main_hypothesis}.
	
	Further, for $M, N \in M$ and $\delta < \lambda^+$ limit, $N$ is a $(\lambda, \delta)$-limit model over $M$ in $\K$ if and only if $N \in \K_M$ and $N$ is a $(\lambda, \delta)$-limit model over $M$ in $\K_M$.
\end{proposition}

\begin{proof}
	$\K_M$ satisfies the isomorphism axioms and coherence axiom from the same properties of $\K$. Since all elements of a chain and above a chain are related by $\sim$ to the bottom model (since if $M' \lek N'$, the identity embedding of both into $N'$ witnesses $M' \sim N'$), the chain axioms also transfer to $\K_M$ from $\K$. Similarly a model obtained from the Löwenheim–Skolem axiom will also be related by $\sim$ to the original, so this holds in $\K_M$ too with $\LS(\K_M) = \LS(\K)$. So $\K_M$ is an AEC with $\LS(\K_M) = \LS(\K)$. $\K_M$ satisfies JEP by the definition of $\sim$ and the fact $\sim$ is an equivalence relation. 
	
	The properties in Hypothesis \ref{main_hypothesis} follow immediately from the corresponding properties holding for $\K, \dnf, \kappa, \lambda, \theta$. In particular, the properties of $\dnf$ transfer because all models mentioned in the property will lie in $\K_M$ also (for example, the $N$ in the definition of non-forking amalgamation will be in $\K_M$ since the other models in the amalgam are in $\K_M$ also).
	
	The final statement is immediate (as before, the identity embedding of $M$ and $N$ into $N$ witnesses that $N \in \K_M$).
\end{proof}

Therefore, if we want to prove the base-fixing version of Theorem \ref{main-theorem} for base $M \in \K_\lambda$, it is enough to prove the theorem in the case that $\K = \K_M$. This means we can assume JEP.

Now we deal with NMM in $\K_\lambda$. By `$N$ is a maximal model', we will mean that there is no $N'$ with $N \lek N'$ and $N \neq N'$.

When NMM in $\K_\lambda$ fails (and JEP holds), the picture is simplified, since all universal and limit models over $M$ are essentially the same - all must be a maximal model containing $M$.

\begin{proposition}\label{wlog-nmm}
	Suppose $\K$ is stable in $\lambda \geq \LS(\K)$ and $\K_\lambda$ has AP and JEP. If $\K_\lambda$ does not have NMM, and $M \lek^u N_1, N_2$, then $N_1 \cong_M N_2$.
	
	Moreover, if $N_1, N_2$ are $\lambda$-limit models over $M$, then $N_1 \cong_M N_2$.
\end{proposition}

\begin{proof}
	Since $\K_\lambda$ has a maximal model $N^*$, by applying JEP there exist $N \in \K_\lambda$ with $M \lek N$ and a $\K$-embedding $f:N^* \rightarrow N$. Since $N^*$ is maximal, $N = f[N^*]$, so $N$ is maximal also.
	
	 If $M \lek^u N_l$ for $l = 1, 2$, then since $M \lek N$ there is an embedding $f_l : N \rightarrow N_l$ fixing $M$. But since $N$ is maximal, $N_l = f_l[N]$. Therefore $f_2 \circ f_1^{-1} : N_1 \rightarrow N_2$ is an isomorphism fixing $M$.
	
	For the moreover part, note that if $N$ is a limit model over $M$, then $N$ is universal over $M$.
\end{proof}

So given JEP, the case where NMM in $\K_\lambda$ fails can be excluded. Together, Proposition \ref{wlog-jep} and Proposition \ref{wlog-nmm} say that without loss of generality we may add JEP and NMM in $\K_\lambda$ to Hypothesis \ref{main_hypothesis}.

\subsection{Towers and brilliant chains}\label{subsection-towers-and-brilliant-chains}

\begin{notation}
    Suppose $I$ is a well ordering. Define
    
    \begin{align*} 
    I^{-} &= \begin{cases}
        I \setminus \{i_0\} & \text{ if $\mathrm{otp}(I) = \alpha + 1$, with final element $i_0$}\\
        I & \text{ if $I$ has 0 or limit order type.}
    \end{cases}\\
    I^{-2} &= \begin{cases}
        I \setminus \{i_0\} & \text{ if $\mathrm{otp}(I) = \alpha + 1$ where $\alpha$ is 0 or limit, with final element $i_0$}\\
        I \setminus \{i_0, i_1\} & \text{ if $\mathrm{otp}(I) = \alpha + 2$ for some $\alpha$, with final elements $i_0 < i_1$}\\
        I & \text{ if $I$ has 0 or limit order type.}
    \end{cases}
    \end{align*}
\end{notation}

That is, $I^-$ removes the final element, and $I^{-2}$ removes the final two elements of $I$, if they exist. Put another way, $I^-$ is the largest set such that for all $i \in I^-$, $i + 1 \in I$, and $I^{-2}$ is the largest set such that for all $i \in I^{-2}$, $i + 2 \in I$.

\begin{definition}
    Let $I$ be a well ordering of order type $< \mu^+$. A \emph{tower} is a sequence $\langle M_i : i \in I \rangle ^\wedge \langle a_i : i \in I^{-2} \rangle$ such that $\langle M_i : i \in I \rangle$ is a $\lek$-increasing sequence in $\Kkappalims$, $a_i \in M_{i+2}$ for each $i \in I^{-2}$, and $\gtp(a_i / M_{i+1}, M_{i+2})$ $\dnf$-does not fork over $M_i$.
\end{definition}

\begin{definition}\label{universal-strong-limit-def}
	Let $\calt = \langle M_i : i \in I \rangle ^\wedge \langle a_i : i \in I^{-2} \rangle$ be a tower.
	\begin{enumerate}
		\item $\calt$ is \emph{universal} if for all $i \in I^-$, $M_i \lek^u M_{i+1}$
		\item Given $\delta < \lambda^+$ limit, $\calt$ is a \emph{strongly $(\lambda, \delta)$-limit} if for every $i \in I$ which is not minimal in $I$, $M_i$ is $(\lambda, \delta)$-limit over $\bigcup_{r<i}M_r$.
	\end{enumerate}
\end{definition}

Note that a strongly $(\lambda, \delta)$-limit tower is also universal. Next we show that towers exist.

\begin{lemma}\label{towers-exist}
    If $M$ is a $(\lambda, \geq\kappa)$-limit model and $\delta < \lambda^+$ with $\cof(\delta) \geq \kappa$, and $I$ is a well order with $|I| < \lambda^+$ and minimal element $i_0$, then there exists a strongly $(\lambda, \delta)$-limit tower $\calt = \langle M_i : i \in I \rangle ^\wedge \langle a_i : i \in I^{-2} \rangle$ where $M_{i_0} = M$.
\end{lemma}

\begin{proof}
	 Take $M_{i_0} = M$. For $i > i_0$, let $M_i$ be any $(\lambda, \delta)$-limit model over $\bigcup_{r<i} M_r$. For $i \in I^{-2}$, take any $p \in \gS(M_i)$. By $(\lambda, \geq\kappa)$-local character, since $M_i$ is a $(\lambda, \geq \kappa)$ limit model, $p$ $\dnf$-does not fork over some $M' \lek^u M_i$. By extension, there is $q \in \gS(M_{i+1})$ which $\dnf$-does not fork over $M'$. $q$ does not fork over $M_i$ by base monotonicity. As $M_{i+2}$ is universal over $M_{i+1}$, there exists $a_i \in M_{i+2}$ where $\gtp(a_i/M_{i+1}, M_{i+2}) = q$. This completes the construction.
\end{proof}

\begin{remark}
	Towers may be highly discontinuous. For example,  suppose $\kappa > \aleph_0$, and that the $(\lambda, \geq\kappa)$-limit models (which we will later prove are all isomorphic) are \emph{not} isomorphic to the $(\lambda, \aleph_0)$-limit model (this is possible by \cite[5.1]{bema}, or for a more concrete example, when $\K$ is the elementary class of a strictly stable theory by \cite[3.4]{beard25}). Given a universal tower $\calt = \langle M_i : i \in I \rangle ^\wedge \langle a_i : i \in I^{-2} \rangle$, at any $i \in I$ with $\cof(i) = \aleph_0$, we cannot possibly have $M_i = \bigcup_{r < i} M_r$, since the former is a $(\lambda, \geq \kappa)$-limit model, and the latter is a $(\lambda, \aleph_0)$-limit model. We can however always build towers that are universal and continuous only at $i \in I$ where $\cof_I(i) \geq \kappa$ - just take unions at such $i \in I$ in the construction from Lemma \ref{towers-exist}.
\end{remark}

\begin{definition}
    Suppose $\calt = \langle M_i : i \in I \rangle ^\wedge \langle a_i : i \in I^{-2} \rangle$ is a tower, and $I_0 \subseteq I$ such that for all $i \in I_0^{-}$, $i +_I 1 = i +_{I_0} 1$. Then define $\calt \upharpoonright I_0 = \langle M_i : i \in I_0 \rangle ^\wedge \langle a_i : i \in I_0^{-2} \rangle$.
\end{definition}

\begin{remark}
    Since $+_I$ and $+_{I_0}$ agree enough on $I_0$, $\calt \upharpoonright I_0$ is a tower.
\end{remark}

\begin{definition}\label{tower_ordering_def}
    We define an ordering $\lesst$ on towers as follows: given towers $\calt = \langle M_i : i \in I \rangle ^\wedge \langle a_i : i \in I^{-2} \rangle$ and $\calt' = \langle M'_i : i \in I' \rangle ^\wedge \langle a_i' : i \in (I')^{-2} \rangle$, $\calt \lesst \calt'$ if and only if 
    \begin{enumerate}
        \item $I \subseteq I'$ (as a subordering)
        \item $M_i \lek^u M_i'$ for all $i \in I$
        \item $a_i = a_i'$ for all $i \in I^{-2}$
        \item $\gtp(a_i/M_{i+_I 1}', M_{i+_I 2}')$ $\dnf$-does not fork over $M_i$ for all $i \in I^{-2}$.
        \item For all $i \in I^-$, $i +_{I} 1 = i +_{I'} 1$.
    \end{enumerate}
\end{definition}

\begin{remark}
    Since all the towers we discuss will be comparable to each other, we can use $a_i$ to denote the singletons in every tower unambiguously (rather than introducing them as $a_i'$ each time).
\end{remark}

\begin{lemma}\label{transitivity-on-universal}
    If $\calt^1 \lesst \calt^2$ and $\calt^2 \lesst \calt^3$ where $\calt^2$ is universal, then $\calt^1 \lesst \calt^3$.
\end{lemma}

\begin{proof}
    Say $\calt^j = \langle M_i^j : i \in I^j \rangle ^\wedge \langle a_i : i \in (I^j)^{-2} \rangle$ for $j = 1, 2, 3$. We check the conditions (1) to (5) from Definition \ref{tower_ordering_def}. Conditions (1), (2), (3), (5) follow immediately from transitivity of $\subseteq, \lek^u, =$, and $=$ respectively.
    
    For condition (4) of Definition \ref{tower_ordering_def}, let $i \in (I^1)^{-2}$. By (5), $i +_I 1$ and $i +_I 2$ are the same in $I = I^1, I^2, I^3$, so we may unambiguously use $+$ for all three. Take $p = \gtp(a_i / M^3_{i + 1}, M^3_{i + 2})$. By $\calt^1 \lesst \calt^2$, we know $p \upharpoonright M^2_{i+ 1}$ $\dnf$-does not fork over $M^1_i$. By $\calt^2 \lesst \calt^3$, we know that $p$ $\dnf$-does not fork over $M^2_i$. Note $M_i^1 \lek M_i^2 \lek^u M_{i+ 1}^2 \lek M_{i+1}^3$ (critically, the universal extension follows from $\calt^2$ being universal). So by weak transitivity, $p$ $\dnf$-does not fork over $M_i^1$ as desired.
\end{proof}

So the tower ordering is a partial ordering on universal towers. However, it is not clear if the ordering is transitive when the middle tower $\calt^2$ is not universal. So we emphasise the cases where this happens:

\begin{definition}
    A \emph{$\lesst$-chain of towers} is a sequence of towers $\langle \calt^j : j < \alpha \rangle$ where $\alpha < \lambda^+$ and for all $j < j' < \alpha$, $\calt^j \lesst \calt^{j'}$.
\end{definition}

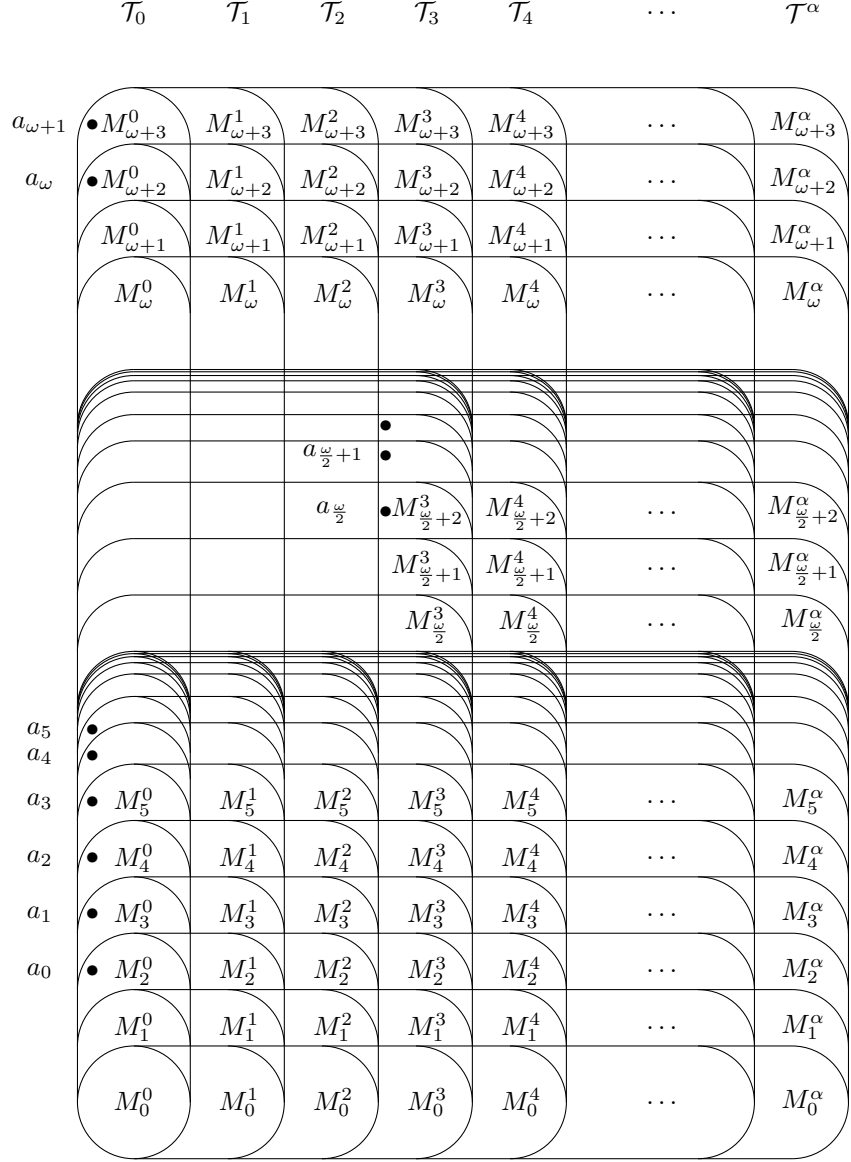
\begin{figure}[!ht]
\centering
\begin{circuitikz}

\def\uppersemi at (#1,#2){\def\Radius{0.75}
    
  \draw
    (\Radius + #1, #2) arc(0:180:\Radius);}

\def\sidesemi at (#1,#2){\def\Radius{0.75}
    
  \draw
    (#1, #2 - \Radius) arc(-90:90:\Radius);}

\def\fullcircle at (#1,#2){\def\Radius{0.75}
    
  \draw
    (\Radius + #1, #2) arc(0:360:\Radius);}

\def\upperright at (#1,#2){\def\Radius{0.75}
    
  \draw
    (\Radius + #1, #2) arc(0:90:\Radius);}
    
\def\upperleft at (#1,#2){\def\Radius{0.75}
    
  \draw
    (#1, \Radius + #2) arc(90:180:\Radius);}

\def\threequarter at (#1,#2){\def\Radius{0.75}
    
  \draw
    (#1, #2 - \Radius) arc(-90:180:\Radius);}

\tikzstyle{every node}=[font=\normalsize]

\fullcircle at (3.75,6.5);
\uppersemi at (3.75,7.25);
\uppersemi at (3.75,8);
\uppersemi at (3.75,8.75);
\uppersemi at (3.75,9.5);
\uppersemi at (3.75,10.25);
\uppersemi at (3.75,10.8);
\uppersemi at (3.75,11.15);
\uppersemi at (3.75,11.45);
\uppersemi at (3.75,11.6);
\uppersemi at (3.75,11.68);
\uppersemi at (3.75,11.72);
\uppersemi at (3.75,11.75);

\upperleft at  (3.75 ,12.5);
\upperleft at  (3.75 ,13.25);
\upperleft at  (3.75 ,14);
\upperleft at  (3.75 ,14.55);
\upperleft at  (3.75 ,14.9);
\upperleft at  (3.75 ,15.20);
\upperleft at  (3.75 ,15.35);
\upperleft at  (3.75 ,15.42);
\upperleft at  (3.75 ,15.47);
\upperleft at  (3.75 ,15.5);

\uppersemi at (3.75,17);
\uppersemi at (3.75,17.75);
\uppersemi at (3.75,18.5);
\uppersemi at (3.75,19.25);

\draw [short] (3,6.5) -- (3,19.25);
\draw [short] (4.5,6.5) -- (4.5,19.25);
\node [font=\normalsize] at (3.75,6.5) {$M^0_0$};
\node [font=\normalsize] at (3.75,7.5) {$M^0_1$};
\node [font=\normalsize] at (3.75,8.25) {$M^0_2$};
\node [font=\normalsize] at (3.75,9) {$M^0_3$};
\node [font=\normalsize] at (3.75,9.75) {$M^0_4$};
\node [font=\normalsize] at (3.75,10.5) {$M^0_5$};
\node [font=\normalsize] at (3.75,17.25) {$M^0_\omega$};
\node [font=\normalsize] at (3.75,18) {$M^0_{\omega + 1}$};
\node [font=\normalsize] at (3.75,18.75) {$M^0_{\omega + 2}$};
\node [font=\normalsize] at (3.75,19.5) {$M^0_{\omega + 3}$};
\node [font=\normalsize] at (3.75,21) {$\calt_0$};

\node [font=\normalsize] at (2.5,8.25) {$a_0$};
\node [font=\normalsize] at (2.5,9) {$a_1$};
\node [font=\normalsize] at (2.5,9.75) {$a_2$};
\node [font=\normalsize] at (2.5,10.5) {$a_3$};
\node [font=\normalsize] at (2.5,11.1) {$a_4$};
\node [font=\normalsize] at (2.5,11.45) {$a_5$};
\node [font=\normalsize] at (2.5,18.75) {$a_{\omega}$};
\node [font=\normalsize] at (2.5,19.5) {$a_{\omega+1}$};

\node [font=\normalsize] at (3.2,8.25) {$\bullet$};
\node [font=\normalsize] at (3.2,9)    {$\bullet$};
\node [font=\normalsize] at (3.2,9.75) {$\bullet$};
\node [font=\normalsize] at (3.2,10.5) {$\bullet$};
\node [font=\normalsize] at (3.2,11.1) {$\bullet$};
\node [font=\normalsize] at (3.2,11.45) {$\bullet$};
\node [font=\normalsize] at (3.2,18.75){$\bullet$};
\node [font=\normalsize] at (3.2,19.5) {$\bullet$};

\sidesemi at (5 ,6.5);
\upperright at  (5 ,7.25);
\upperright at  (5 ,8);
\upperright at  (5 ,8.75);
\upperright at  (5 ,9.5);
\upperright at  (5 ,10.25);
\upperright at  (5 ,10.8);
\upperright at  (5 ,11.15);
\upperright at  (5 ,11.45);
\upperright at  (5 ,11.6);
\upperright at  (5 ,11.68);
\upperright at  (5 ,11.72);
\upperright at  (5 ,11.75);

\upperright at  (5 ,17);
\upperright at  (5 ,17.75);
\upperright at  (5 ,18.5);
\upperright at  (5 ,19.25);

\draw [short] (5.75,6.5) -- (5.75,19.25);
\node [font=\normalsize] at (5.15 ,6.5)  {$M^1_0$};
\node [font=\normalsize] at (5.15 ,7.5)  {$M^1_1$};
\node [font=\normalsize] at (5.15 ,8.25) {$M^1_2$};
\node [font=\normalsize] at (5.15 ,9)    {$M^1_3$};
\node [font=\normalsize] at (5.15 ,9.75) {$M^1_4$};
\node [font=\normalsize] at (5.15 ,10.5) {$M^1_5$};
\node [font=\normalsize] at (5.15 ,17.25) {$M^1_\omega $};
\node [font=\normalsize] at (5.15 ,18)    {$M^1_{\omega + 1}$};
\node [font=\normalsize] at (5.15 ,18.75) {$M^1_{\omega + 2}$};
\node [font=\normalsize] at (5.15 ,19.5)  {$M^1_{\omega + 3}$};
\node [font=\normalsize] at (5.15,21)    {$\calt_1$};

\sidesemi at    (6.25 ,6.5);
\upperright at  (6.25 ,7.25);
\upperright at  (6.25 ,8);
\upperright at  (6.25 ,8.75);
\upperright at  (6.25 ,9.5);
\upperright at  (6.25 ,10.25);
\upperright at  (6.25 ,10.8);
\upperright at  (6.25 ,11.15);
\upperright at  (6.25 ,11.45);
\upperright at  (6.25 ,11.6);
\upperright at  (6.25 ,11.68);
\upperright at  (6.25 ,11.72);
\upperright at  (6.25 ,11.75);

\upperright at  (6.25 ,17);
\upperright at  (6.25 ,17.75);
\upperright at  (6.25 ,18.5);
\upperright at  (6.25 ,19.25);

\draw [short] (7  ,6.5) -- (7  ,19.25);
\node [font=\normalsize] at (6.4 ,6.5)  {$M^2_0$};
\node [font=\normalsize] at (6.4 ,7.5)  {$M^2_1$};
\node [font=\normalsize] at (6.4 ,8.25) {$M^2_2$};
\node [font=\normalsize] at (6.4 ,9)    {$M^2_3$};
\node [font=\normalsize] at (6.4 ,9.75) {$M^2_4$};
\node [font=\normalsize] at (6.4 ,10.5) {$M^2_5$};
\node [font=\normalsize] at (6.4 ,17.25) {$M^2_\omega $};
\node [font=\normalsize] at (6.4 ,18)    {$M^2_{\omega + 1}$};
\node [font=\normalsize] at (6.4 ,18.75) {$M^2_{\omega + 2}$};
\node [font=\normalsize] at (6.4 ,19.5)  {$M^2_{\omega + 3}$};
\node [font=\normalsize] at (6.4,21)    {$\calt_2$};

\sidesemi at    (7.5  ,6.5);
\upperright at  (7.5  ,7.25);
\upperright at  (7.5  ,8);
\upperright at  (7.5  ,8.75);
\upperright at  (7.5  ,9.5);
\upperright at  (7.5  ,10.25);
\upperright at  (7.5  ,10.8);
\upperright at  (7.5  ,11.15);
\upperright at  (7.5  ,11.45);
\upperright at  (7.5  ,11.6);
\upperright at  (7.5  ,11.68);
\upperright at  (7.5  ,11.72);
\upperright at  (7.5  ,11.75);

\upperright at  (7.5  ,12.5);
\upperright at  (7.5  ,13.25);
\upperright at  (7.5  ,14);
\upperright at  (7.5  ,14.55);
\upperright at  (7.5  ,14.9);
\upperright at  (7.5  ,15.20);
\upperright at  (7.5  ,15.35);
\upperright at  (7.5  ,15.42);
\upperright at  (7.5  ,15.47);
\upperright at  (7.5  ,15.5);

\upperright at  (7.5  ,17);
\upperright at  (7.5  ,17.75);
\upperright at  (7.5  ,18.5);
\upperright at  (7.5  ,19.25);

\draw [short] (8.25,6.5) -- (8.25,19.25);
\node [font=\normalsize] at (7.65 ,6.5)  {$M^3_0$};
\node [font=\normalsize] at (7.65 ,7.5)  {$M^3_1$};
\node [font=\normalsize] at (7.65 ,8.25) {$M^3_2$};
\node [font=\normalsize] at (7.65 ,9)    {$M^3_3$};
\node [font=\normalsize] at (7.65 ,9.75) {$M^3_4$};
\node [font=\normalsize] at (7.65 ,10.5) {$M^3_5$};

\node [font=\normalsize] at (7.65 ,12.85) {$M^3_{\frac{\omega}{2}} $};
\node [font=\normalsize] at (7.65 ,13.6)    {$M^3_{\frac{\omega}{2} + 1}$};
\node [font=\normalsize] at (7.65 ,14.35) {$M^3_{\frac{\omega}{2} + 2}$};

\node [font=\normalsize] at (7.65 ,17.25) {$M^3_\omega $};
\node [font=\normalsize] at (7.65 ,18)    {$M^3_{\omega + 1}$};
\node [font=\normalsize] at (7.65 ,18.75) {$M^3_{\omega + 2}$};
\node [font=\normalsize] at (7.65 ,19.5)  {$M^3_{\omega + 3}$};
\node [font=\normalsize] at (7.65,21)    {$\calt_3$};

\node [font=\normalsize] at (6.4 ,14.35) {$a_{\frac{\omega}{2}} $};
\node [font=\normalsize] at (6.4 ,15.1)    {$a_{\frac{\omega}{2} + 1}$};

\node [font=\normalsize] at (7.1 ,14.35) {$\bullet$};
\node [font=\normalsize] at (7.1 ,15.1)  {$\bullet$};
\node [font=\normalsize] at (7.1 ,15.5)  {$\bullet$};

\sidesemi at    (8.75 ,6.5);
\upperright at  (8.75 ,7.25);
\upperright at  (8.75 ,8);
\upperright at  (8.75 ,8.75);
\upperright at  (8.75 ,9.5);
\upperright at  (8.75 ,10.25);
\upperright at  (8.75 ,10.8);
\upperright at  (8.75 ,11.15);
\upperright at  (8.75 ,11.45);
\upperright at  (8.75 ,11.6);
\upperright at  (8.75 ,11.68);
\upperright at  (8.75 ,11.72);
\upperright at  (8.75 ,11.75);

\upperright at  (8.75 ,12.5);
\upperright at  (8.75 ,13.25);
\upperright at  (8.75 ,14);
\upperright at  (8.75 ,14.55);
\upperright at  (8.75 ,14.9);
\upperright at  (8.75 ,15.20);
\upperright at  (8.75 ,15.35);
\upperright at  (8.75 ,15.42);
\upperright at  (8.75 ,15.47);
\upperright at  (8.75 ,15.5);

\upperright at  (8.75 ,17);
\upperright at  (8.75 ,17.75);
\upperright at  (8.75 ,18.5);
\upperright at  (8.75 ,19.25);

\draw [short] (9.5,6.5) -- (9.5,19.25);
\node [font=\normalsize] at (8.9 ,6.5)  {$M^4_0$};
\node [font=\normalsize] at (8.9 ,7.5)  {$M^4_1$};
\node [font=\normalsize] at (8.9 ,8.25) {$M^4_2$};
\node [font=\normalsize] at (8.9 ,9)    {$M^4_3$};
\node [font=\normalsize] at (8.9 ,9.75) {$M^4_4$};
\node [font=\normalsize] at (8.9 ,10.5) {$M^4_5$};

\node [font=\normalsize] at (8.9 ,12.85) {$M^4_{\frac{\omega}{2}} $};
\node [font=\normalsize] at (8.9 ,13.6)    {$M^4_{\frac{\omega}{2} + 1}$};
\node [font=\normalsize] at (8.9 ,14.35) {$M^4_{\frac{\omega}{2} + 2}$};

\node [font=\normalsize] at (8.9 ,17.25) {$M^4_\omega $};
\node [font=\normalsize] at (8.9 ,18)    {$M^4_{\omega + 1}$};
\node [font=\normalsize] at (8.9 ,18.75) {$M^4_{\omega + 2}$};
\node [font=\normalsize] at (8.9 ,19.5)  {$M^4_{\omega + 3}$};
\node [font=\normalsize] at (8.9,21)    {$\calt_4$};

\sidesemi at    (11.25 ,6.5);
\upperright at  (11.25 ,7.25);
\upperright at  (11.25 ,8);
\upperright at  (11.25 ,8.75);
\upperright at  (11.25 ,9.5);
\upperright at  (11.25 ,10.25);
\upperright at  (11.25 ,10.8);
\upperright at  (11.25 ,11.15);
\upperright at  (11.25 ,11.45);
\upperright at  (11.25 ,11.6);
\upperright at  (11.25 ,11.68);
\upperright at  (11.25 ,11.72);
\upperright at  (11.25 ,11.75);

\upperright at  (11.25 ,12.5);
\upperright at  (11.25 ,13.25);
\upperright at  (11.25 ,14);
\upperright at  (11.25 ,14.55);
\upperright at  (11.25 ,14.9);
\upperright at  (11.25 ,15.20);
\upperright at  (11.25 ,15.35);
\upperright at  (11.25 ,15.42);
\upperright at  (11.25 ,15.47);
\upperright at  (11.25 ,15.5);

\upperright at  (11.25 ,17);
\upperright at  (11.25 ,17.75);
\upperright at  (11.25 ,18.5);
\upperright at  (11.25 ,19.25);

\draw [short] (12,6.5) -- (12,19.25);
\node [font=\normalsize] at (10.8 ,6.5)  {$\dots$};
\node [font=\normalsize] at (10.8 ,7.5)  {$\dots$};
\node [font=\normalsize] at (10.8 ,8.25) {$\dots$};
\node [font=\normalsize] at (10.8 ,9)    {$\dots$};
\node [font=\normalsize] at (10.8 ,9.75) {$\dots$};
\node [font=\normalsize] at (10.8 ,10.5) {$\dots$};

\node [font=\normalsize] at (10.8 ,12.85) {$\dots$};
\node [font=\normalsize] at (10.8 ,13.6)  {$\dots$};
\node [font=\normalsize] at (10.8 ,14.35) {$\dots$};

\node [font=\normalsize] at (10.8 ,17.25) {$\dots$};
\node [font=\normalsize] at (10.8 ,18)    {$\dots$};
\node [font=\normalsize] at (10.8 ,18.75) {$\dots$};
\node [font=\normalsize] at (10.8 ,19.5)  {$\dots$};
\node [font=\normalsize] at (10.8,21)    {$\dots$};

\sidesemi at    (12.5 ,6.5);
\upperright at  (12.5 ,7.25);
\upperright at  (12.5 ,8);
\upperright at  (12.5 ,8.75);
\upperright at  (12.5 ,9.5);
\upperright at  (12.5 ,10.25);
\upperright at  (12.5 ,10.8 );
\upperright at  (12.5 ,11.15);
\upperright at  (12.5 ,11.45);
\upperright at  (12.5 ,11.6 );
\upperright at  (12.5 ,11.68);
\upperright at  (12.5 ,11.72);
\upperright at  (12.5 ,11.75);

\upperright at  (12.5 ,12.5 );
\upperright at  (12.5 ,13.25);
\upperright at  (12.5 ,14   );
\upperright at  (12.5 ,14.55);
\upperright at  (12.5 ,14.9 );
\upperright at  (12.5 ,15.20);
\upperright at  (12.5 ,15.35);
\upperright at  (12.5 ,15.42);
\upperright at  (12.5 ,15.47);
\upperright at  (12.5 ,15.5 );

\upperright at  (12.5 ,17   );
\upperright at  (12.5 ,17.75);
\upperright at  (12.5 ,18.5 );
\upperright at  (12.5 ,19.25);

\draw [short] (13.25,6.5) -- (13.25,19.25);
\node [font=\normalsize] at (12.65 ,6.5)  {$M^\alpha_0$};
\node [font=\normalsize] at (12.65 ,7.5)  {$M^\alpha_1$};
\node [font=\normalsize] at (12.65 ,8.25) {$M^\alpha_2$};
\node [font=\normalsize] at (12.65 ,9)    {$M^\alpha_3$};
\node [font=\normalsize] at (12.65 ,9.75) {$M^\alpha_4$};
\node [font=\normalsize] at (12.65 ,10.5) {$M^\alpha_5$};

\node [font=\normalsize] at (12.65 ,12.85) {$M^\alpha_{\frac{\omega}{2}} $};
\node [font=\normalsize] at (12.65 ,13.6)    {$M^\alpha_{\frac{\omega}{2} + 1}$};
\node [font=\normalsize] at (12.65 ,14.35) {$M^\alpha_{\frac{\omega}{2} + 2}$};

\node [font=\normalsize] at (12.65 ,17.25) {$M^\alpha_\omega $};
\node [font=\normalsize] at (12.65 ,18)    {$M^\alpha_{\omega + 1}$};
\node [font=\normalsize] at (12.65 ,18.75) {$M^\alpha_{\omega + 2}$};
\node [font=\normalsize] at (12.65 ,19.5)  {$M^\alpha_{\omega + 3}$};
\node [font=\normalsize] at (12.65,21)    {$\calt^\alpha$};


\draw [short] (3.75,5.75) -- (12.5,5.75);
\draw [short] (3.75,7.25) -- (12.5,7.25);
\draw [short] (3.75,8   ) -- (12.5,8   );
\draw [short] (3.75,8.75) -- (12.5,8.75);
\draw [short] (3.75,9.5 ) -- (12.5,9.5 );
\draw [short] (3.75,10.25)-- (12.5,10.25);
\draw [short] (3.75,11)   -- (12.5,11);
\draw [short] (3.75,11.55)   -- (12.5,11.55);
\draw [short] (3.75,11.9 )   -- (12.5,11.9 );
\draw [short] (3.75,12.2 )   -- (12.5,12.2 );
\draw [short] (3.75,12.35)   -- (12.5,12.35);
\draw [short] (3.75,12.43)   -- (12.5,12.43);
\draw [short] (3.75,12.47)   -- (12.5,12.47);
\draw [short] (3.75,12.5 )   -- (12.5,12.5 );

\draw [short] (3.75,13.25)-- (12.5,13.25);
\draw [short] (3.75,14   )-- (12.5,14   );
\draw [short] (3.75,14.75)-- (12.5,14.75);
\draw [short] (3.75,15.3 )-- (12.5,15.3 );
\draw [short] (3.75,15.65)-- (12.5,15.65);
\draw [short] (3.75,15.95)-- (12.5,15.95);
\draw [short] (3.75,16.10)-- (12.5,16.10);
\draw [short] (3.75,16.17)-- (12.5,16.17);
\draw [short] (3.75,16.22)-- (12.5,16.22);
\draw [short] (3.75,16.25)-- (12.5,16.25);

\draw [short] (3.75,17.75)-- (12.5,17.75);
\draw [short] (3.75,18.5 )-- (12.5,18.5 );
\draw [short] (3.75,19.25)-- (12.5,19.25);
\draw [short] (3.75,20   )-- (12.5,20   );

\end{circuitikz}

\caption{A $\lesst$-chain of length $\alpha+1$. Note the singletons appear two levels up in the tower, and we may only insert new levels after an infinite chain of models because of condition (5) of \ref{tower_ordering_def}. $\calt^0, \calt^1$, and $\calt^2$ are ordered by $\omega + 4$, whereas $\calt^3, \calt^4$, and so on are ordered by $\{0, 1, 2, 3, \dots, \frac{\omega}{2}, \frac{\omega}{2} + 1, \frac{\omega}{2} + 2, \dots, \omega, \omega+1, \omega+2, \omega+3\}$. That is, we inserted $\omega$-many new levels just under level $\omega$.}
\label{chain_intuition_diagram}
\end{figure}

The following notations will make it easier to discuss appending towers and $\lesst$-chains of towers to the end of a $\lesst$-chain of towers.

\begin{notation}
	If $\calc = \langle \calt^j : j < \alpha \rangle$ is a $\lesst$-chain of towers, $\calt'$ a tower, and $\calc' = \langle \bar{\calt}^j : j < \beta \rangle$ a second $\lesst$-chain of towers, then 
	\begin{enumerate}
		\item $\calc ^\wedge \calt' = \langle \calt^j : j < \alpha+1 \rangle$ where $\calt^\alpha = \calt'$
		\item $\calc ^\wedge \calc' = \langle \calt^j : j < \alpha + \beta \rangle$ where $\calt^{\alpha + j} = \bar{\calt}^j$
	\end{enumerate}
\end{notation}

\begin{definition}\label{restrict-chains-def}
    If $\calc = \langle \calt^j : j < \alpha \rangle$ is a $\lesst$-chain of towers where $\calt^j$ is indexed by $I^j$ for all $j < \alpha$ and $I_0$ is an interval in $\bigcup_{j<\alpha} I^j$, then the \emph{restriction of $\calc$ to $I_0$} is $\calc \upharpoonright_* I_0 = \langle \calt^j \upharpoonright (I^j \cap I_0) : j < \alpha \rangle$.
\end{definition}

\begin{remark}
    `Restriction' is a little vague (you might confuse it with finding a subsequence of the $\calt^j$'s), but the notation is precise: the * signifies that we are restricting the towers in the sequence, not the sequence itself.
\end{remark}

\begin{remark}
    In the setup of Definition \ref{restrict-chains-def}, $\calc \upharpoonright_* I_0$ is a $\lesst$-chain of towers, as $I^j \cap I_0$ is an interval in $I^j$ (the $\lesst$ conditions are not harmed by removing the upper or lower sections of the towers).
\end{remark}

If we have a $\lesst$-chain of towers of length cofinality at least $\kappa$, we can append the union to the end:

\begin{definition}\label{tower-union-def}
    Suppose $\delta< \lambda^+$ is a limit ordinal where $\cof(\delta) \geq \kappa$, $\langle \calt^j : j < \delta \rangle$ is a $\lesst$-chain of towers, where $\calt^j = \langle M_i^j : j \in I^j \rangle^\wedge \langle a_i : i \in (I^j)^{-2}\rangle$. Let
    \begin{enumerate}
        \item $I^\delta = \bigcup_{j < \delta}I^j$
        \item $M_i^\delta = \bigcup \{M_i^j : j<\delta, i \in I^j\}$
    \end{enumerate} 
    Define the \emph{union of $\langle \calt^j : j < \delta \rangle$} by $\bigcup_{j<\delta} \calt^j = \langle M_i^\delta : i \in I^\delta \rangle ^\wedge \langle a_i : i \in (I^\delta)^{-2} \rangle$.
\end{definition}

\begin{lemma}\label{tower-union-lemma}
    Suppose $\delta< \lambda^+$ is a limit ordinal with $\cof(\delta) \geq \kappa$, $\langle \calt^j : j < \delta \rangle$ is a $\lesst$-chain of towers, where $\calt^j = \langle M_i^j : j \in I^j \rangle^\wedge \langle a_i : i \in (I^j)^{-2}\rangle$. If $I^\delta = \bigcup_{j<\delta} I^j$ is a well ordering, then $\calt^\delta = \bigcup_{j<\delta} \calt^j$ is a tower, and $\langle \calt^j : j \leq \delta \rangle$ is a $\lesst$-chain of towers.
\end{lemma}

\begin{proof}
    Let $\calt^\delta = \langle M_i^\delta : i \in I^\delta \rangle ^\wedge \langle a_i : i \in (I^\delta)^{-2} \rangle$. Given $i < \delta$, by taking $j_0 = \min\{j < \delta : i \in I^j\}$, we see that since $M_i^\delta = \bigcup_{j \in [j_0,\delta)} M_i^j$, $M_i^\delta$ is a $(\lambda, \geq \kappa)$-limit model, hence compatible with $\dnf$. 
    
    If $r, s \in I^\delta$, then since $M_r^j \lek M_s^j$ for each $j < \delta$ with $r, s \in I^j$, we have $M_r^\delta \lek M_s^\delta$. So $\langle M_i^\delta : i \in I^\delta \rangle$ is $\lek$-increasing. For the $\dnf$-non-forking of the $a_i$, let $i \in (I^\delta)^{-2}$. Take $j_0 < \delta$ such that $i \in (I^{j_0})^{-2}$. As such, for $j \in [j_0, \delta]$, $i+_{I^j}1$ and $i+_{I^j}2$ are all the same, so the notation $i+1, i+2$ is unambiguous. Now for $j \geq j_0$, $\gtp(a_i/M_{i+1}^j, M_{i+2}^j)$ $\dnf$-does not fork over $M_i^{j_0}$. By $(\geq \kappa)$-universal continuity of $\dnf$ (Lemma \ref{uniform-cty-star-gives-high-univ-cty}), we have that $\gtp(a_i/M_{i+1}^\delta, M_{i+2}^\delta)$ $\dnf$-does not fork over $M_i^{j_0}$, and hence $\dnf$-does not fork over $M_i^\delta$ by monotonicity. Therefore $\calt^\delta$ is a tower.

    Now let $j < \delta$. We will show $\calt^j \lesst \calt^\delta$. (1), (2), (3), and (5) of Definition \ref{tower_ordering_def} are immediate. For (4), note in the argument above we showed that for all $j_0 < \delta$ where $i \in (I_0)^{-2}$, $\gtp(a_i/M_{i+1}^\delta, M_{i+2}^\delta)$ $\dnf$-does not fork over $M_i^{j_0}$, as required.
\end{proof}

We have shown our ordering behaves nicely with universal towers and long unions of towers. As such, we will focus on chains of such towers, which we call \emph{brilliant}.

\begin{definition}\label{brilliant-tower-def}
    A $\lesst$-chain of towers $\langle\calt^j : j < \alpha\rangle$ is \emph{brilliant} if $\bigcup_{j < \alpha} I^j$ is a well ordering, and for all $j < \alpha$, one of the following holds:
        \begin{enumerate}
            \item $\calt^j$ is a universal tower
            \item $\cof(j) \geq \kappa$ and $\calt^j = \bigcup_{k < j} \calt^k$
        \end{enumerate}
    
\end{definition}

\begin{lemma}
    If $\calc = \langle\calt^j : j < \alpha\rangle$ is a brilliant chain of towers where $\calt^j$ is indexed by $I^j$, and $I_0$ is an interval in $\bigcup_{j<\alpha} I^j$, then $\calc \upharpoonright_* I_0$ is brilliant.
\end{lemma}

\begin{proof}
    The conditions (1) and (2) of Definition \ref{brilliant-tower-def} are preserved when upper or lower models in the tower are removed.
\end{proof}

Note that all strongly $(\lambda, \delta)$-limit towers are universal towers.

\subsection{Extending towers and brilliant chains}\label{extending-towers-and-brilliant-chains-subsection}

Our goal now is to show that brilliant chains of towers may be extended with a universal tower (and remain brilliant). We will also show how to extend towers and brilliant chains with additional nice features that will be useful in Subsection \ref{reduced-subsection}, where we study \emph{reduced chains} of towers.

Many of the proofs have a similar structure, but we present them separately since combining the results is both difficult and notationally messy. Where possible, we describe how a previous construction can be altered to get the desired new result. That said, for a reader who finds themselves fatigued by these constructions, the key results to take note of going forwards are
\begin{itemize}
	\item Corollary \ref{universal-extensions-of-brilliant-chains-exist}, which says that any brilliant chain can be extended by a strongly $(\lambda, \geq \kappa)$-limit tower
	\item Proposition \ref{brilliant-extension-from-initial-segment-exist}, which says that given the bottom part of a universal tower extending the lower levels of a chain, we can complete the extending tower (up to an isomorphism)
	\item Corollary \ref{any-brill-chain-exts-exist-with-b}, which says that we can extend a brilliant chain with a universal tower while realising a non-forking type at the bottom of the new tower.
\end{itemize}

Since brilliant chains come in several different shapes, the proof of Corollary \ref{universal-extensions-of-brilliant-chains-exist} will need two different constructions. First, we deal with the case that the chain ends with a universal tower. In this case, by Lemma \ref{transitivity-on-universal}, we only need to know how to extend a single universal tower (the final one in the chain).

\begin{proposition}\label{univ-tower-exts-exist}
    Suppose $\calt = \langle M_i:i \in I \rangle ^\wedge \langle a_i : i \in I^{-2} \rangle$ is a \emph{universal} tower, and $\gamma < \lambda^+$ is limit with $\cof(\gamma) \geq \kappa$. Then there exists $\calt'$ indexed by $I$ such that $\calt \lesst \calt'$ and $\calt'$ is strongly $(\lambda, \gamma)$-limit.
    
    Moreover, if $\langle \calt^j : j < \alpha + 1 \rangle$ is a brilliant chain of towers, and $\calt^\alpha$ is universal indexed by $I$, then there exists a universal tower $\calt'$ indexed by $I$ such that $\langle \calt^j : j < \alpha + 1 \rangle ^\wedge \calt'$ is brilliant.
\end{proposition}

\begin{proof}
    Say $i_0$ is the minimal element of $I$. We construct a $\lek$-increasing sequence $\langle N_i : i \in I \rangle$ of $(\lambda, \gamma)$-limit models and a $\subseteq$-increasing sequence of $\K$-embeddings $\langle f_i : i \in I \rangle$ such that 
    \begin{enumerate}
        \item $f_i:M_i \rightarrow N_i$ for each $i \in I$
        \item $M_{i_0} \lek N_{i_0}$ and $f_{i_0}$ is the identity embedding
        \item $f_i[M_i] \lek^u N_i$ for each $i \in I$ (in fact, $N_i$ will be a $(\lambda, \gamma)$-limit model over $f_i[M_i]$)
        \item $N_i$ is $(\lambda, \gamma)$-limit over $\bigcup_{r<i} N_r$ for all $i \in I \setminus \{i_0\}$
        \item $\gtp(f_{i+2}(a_i)/N_{i+1}, N_{i+2})$ $\dnf$-does not fork over $f_i[M_i]$.
    \end{enumerate}

    \textbf{This is possible:} We construct by recursion.

    For $i = i_0$, let $N_{i_0}$ be any $(\lambda, \gamma)$-limit model over $M_{i_0}$, and $f_{i_0}$ the identity on $M_{i_0}$.

    For $i$ limit, take any $f_i:M_i \rightarrow \hat{N}_i$ extending $\bigcup_{r < i}f_r : \bigcup_{r<i} M_r \rightarrow \bigcup_{r<i} N_r$, and take $N_i$ any $(\lambda, \gamma)$-limit over $\hat{N}_i$.

    For $i = r + 1$ where $r \in I$ is initial or limit, take any $f_i : M_i \rightarrow \hat{N}_i$ extending $f_r : M_r \rightarrow N_r$. Let $N_i$ be any $(\lambda, \gamma)$-limit over $\hat{N}_i$.

    For $i = r + 2$ for some $r \in I$, we have more work to do. We may assume we have already defined $N_{r+1}$ and $f_{r+1}:M_{r+1} \rightarrow N_{r+1}$. Take some extension $\hat{f}_{r+2}:M_{r+2} \rightarrow \hat{N}_{r+2}$ where $N_{r+1} \lek \hat{N}_{r+2}$. Note $\gtp(a_r/M_{r+1}, M_{r+2})$ $\dnf$-does not fork over $M_r$ as $\calt$ is a tower, so by invariance and monotonicity, $\gtp(\hat{f}_{r+2}(a_r)/f_{r+1}[M_{r+1}], \hat{N}_{r+2})$ $\dnf$-does not fork over $f_{r+1}[M_r]$. Using that $f_{r+1}[M_r] \lek^u f_{r+1}[M_{r+1}] \lek N_{r+1}$, by weak extension, there exists $q \in \gS(N_{r+1})$ where $\gtp(\hat{f}_{i+2}(a_r)/f_{r+1}[M_{r+1}], \hat{N}_{r+1}) \subseteq q$ and $q$ $\dnf$-does not fork over $f_{r+1}[M_r]$. Say $q = \gtp(d_r/N_{r+1}, N^*_{r+2})$.

    By type equality, there is $\bar{N}_{r+2}$ where $N^*_{r+2} \lek \bar{N}_{r+2}$ and $g_{r+2}:\hat{N}_{r+2} \rightarrow \bar{N}_{r+2}$ fixing $f_{r+1}[M_{r+1}]$ such that $g_{r+2}(\hat{f}_{r+2}(a_r)) = d_r$. Now take $N_{r+2}$ a $(\lambda, \gamma)$-limit over $\bar{N}_{r+2}$, and $f_{r+2} = g_{r+2} \circ \hat{f}_{r+2}$. Then $\gtp(f_{r+2}(a_r)/N_{r+2}, N_{r+2}) = q$ $\dnf$-does not fork over $f_{r}[M_r]$, so $f_{r+2}$ and $N_{r+2}$ are as desired.

    \begin{figure}[!ht]
\centering
\begin{circuitikz}
\tikzstyle{every node}=[font=\normalsize]
\node [font=\normalsize] at (5,8.75) {$M_r$};
\node [font=\normalsize] at (5,10.75) {$M_{r+1}$};
\node [font=\normalsize] at (8,8.75) {$f_r[M_r]$};
\node [font=\normalsize] at (8,10.75) {$f_{r+1}[M_{r+1}]$};
\draw [->, >=Stealth] (5,9) -- (5,10.5);
\draw [->, >=Stealth] (8,9) -- (8,10.5);
\draw [->, >=Stealth] (5.5,10.75) -- (7,10.75);
\draw [->, >=Stealth] (5.5,8.75) -- (7,8.75);
\node [font=\normalsize] at (5,12.75) {$M_{r+2}$};
\draw [->, >=Stealth] (5.5,12.75) -- (7,12.75);
\node [font=\normalsize] at (8,12.75) {$\hat{f}_{r+2} [M_{r+2}]$};
\draw [->, >=Stealth] (5,11) -- (5,12.5);
\draw [->, >=Stealth] (8,11) -- (8,12.5);
\draw [->, >=Stealth] (9,12.75) -- (10.5,12.75);
\node [font=\normalsize] at (11,12.75) {$\hat{N}_{r+2}$};
\node [font=\normalsize] at (11,10.75) {$N_{r+1}$};
\draw [->, >=Stealth] (9,10.75) -- (10.5,10.75);
\node [font=\normalsize] at (11,8.75) {$N_r$};
\draw [->, >=Stealth] (11,9) -- (11,10.5);
\draw [->, >=Stealth] (9,8.75) -- (10.5,8.75);
\draw [->, >=Stealth] (11.5,12.75) -- (13,12.75);
\node [font=\normalsize] at (13.5,12.75) {$\bar{N}_{r+2}$};
\draw [->, >=Stealth] (13.5,11) -- (13.5,12.5);
\node [font=\normalsize] at (13.5,10.75) {$N_{r+2}^*$};
\draw [->, >=Stealth] (11.5,10.75) -- (13,10.75);
\draw [->, >=Stealth] (13.75,13) -- (14.75,14);
\node [font=\normalsize] at (15,14.25) {$N_{r+2}$};
\node [font=\normalsize] at (14.75,13.25) {$(\lambda, \gamma)$};
\node [font=\normalsize] at (6.25,13) {$\hat{f}_{r+2}$};
\node [font=\normalsize] at (6.25,11) {$f_{r+1}$};
\node [font=\normalsize] at (6.25,9) {$f_r$};
\node [font=\normalsize] at (12.25,13) {$g_{r+2}$};
\draw [->, >=Stealth] (5,13) .. controls (5,14.25) and (9.75,14.25) .. (14.5,14.25) ;
\node [font=\normalsize] at (9.75,14.5) {$f_{r+2}$};
\end{circuitikz}

\caption{The successor case of Proposition \ref{univ-tower-exts-exist}}
\label{universal_extensions_successor_case}
\end{figure}

    \textbf{This is enough:} Take $f = \bigcup_{i \in I} f_i :\bigcup_{i \in I} M_i \rightarrow \bigcup_{i \in I} N_i$, and extend it to an isomorphism $h: M' \rightarrow \bigcup_{i \in I} N_i$. Let $M_i' = h^{-1}[N_i]$ for all $i \in I$, and take $\calt' = \langle M_i' : i \in I \rangle ^\wedge \langle a_i : i \in I^{-2} \rangle$. $\calt'$ is a tower since $I$ is a well ordering, $\langle N_i : i \in I \rangle$ is $\lek^u$-increasing, and by (3) and (5) of the construction with base monotonicity. Additionally, $\calt'$ is strongly $(\lambda, \kappa)$-limit by (4) of the construction.

    Now we check that $\calt \lesst \calt'$. Clauses (1), (3), and (5) of Definition \ref{tower_ordering_def} follow from the definition of $\calt'$, and clause (2) of Definition \ref{tower_ordering_def} follows from (3) of the construction. Clause (4) of Definition \ref{tower_ordering_def} follows from (5) in the construction and invariance of $\dnf$. So $\calt'$ is as desired.
    
    For the moreover part of the statement, take any universal tower $\calt'$ indexed by $I$ extending $\calt^\alpha$ by the previous work. Using Lemma \ref{transitivity-on-universal} and that $\calt^\alpha$ is universal, $\langle \calt^j : j < \alpha + 1 \rangle ^\wedge \calt'$ is ordered by $\lesst$, so the chain is brilliant.
\end{proof}

Now we deal with the case that the brilliant chain has limit length, or ends with a union.

\begin{proposition}\label{limit-tower-exts-exist}
    Suppose $\gamma, \delta < \lambda^+$ are limit ordinals where $\cof(\gamma) \geq \kappa$, and that $\langle \calt^j : j < \delta \rangle$ is a $\lesst$-increasing sequence of universal towers. Say  $\calt^j = \langle M_i^j : i \in I^j \rangle ^\wedge \langle a_i : i \in (I^j)^{-2} \rangle$ for $j< \delta$, and suppose $\bigcup_{j < \delta} I^j = I = I^\delta$ is a well ordering. Then there exists a strongly $(\lambda, \gamma)$-limit tower $\calt'$ indexed by $I$ such that $\calt^j \lesst \calt'$ for all $j < \delta$. 
    
    Moreover, if $\cof(\delta) \geq \kappa$, then $\calt'$ can be taken such that $\bigcup_{j < \delta} \calt^j \lesst \calt'$.
\end{proposition}

\begin{proof}
    By (5) of the tower ordering definition, we may be ambiguous with $i+1, i+2$ between the various $I^j$. Let $M_i^\delta = \bigcup_{j < \delta} M_i^j$ (note these may not be $(\lambda, \geq \kappa)$ limits if $\cof(\delta) < \kappa$).
    
    Similarly to before, we construct a $\lek$-increasing sequence $\langle N_i : i \in I \rangle$ of $(\lambda, \gamma)$-limit models and a $\subseteq$-increasing sequence of $\K$-embeddings $\langle f_i : i \in I \rangle$ such that
    \begin{enumerate}
        \item $f_i:M_i^\delta \rightarrow N_i$ for each $i \in I$
        \item $M_{i_0}^\delta \lek N_0$ and $f_{i_0}$ is the identity embedding
        \item $f_i[M_i^\delta] \lek^u N_i$ for each $i \in I$ (in fact, $N_i$ will be a $(\lambda, \gamma)$-limit model over $f_i[M_i]$)
        \item $N_i$ is a $(\lambda, \gamma)$-limit model over $\bigcup_{r<i} N_r$ for all $i \in I \setminus \{i_0\}$
        \item For $j < \delta$ and $i \in (I^j)^{-2}$, $\gtp(f_{i+2}(a_i)/N_{i+1}, N_{i+2})$ $\dnf$-does not fork over $f_i[M_i^j]$.
    \end{enumerate}

    \textbf{This is possible:} Again we construct by recursion. The $i = i_0$, $i$ limit, and $i = \alpha + 1$ for $\alpha$ either 0 or limit cases are all identical, replacing $M_i$ with $M_i^\delta$. 

    For $i = r + 2$ for some $r \in I$, again we have more work to do. We may assume we have already defined $N_{r+1}$ and $f_{r+1}:M_{r+1}^\delta \rightarrow N_{r+1}$. Take the minimal $j < \delta$ such that $r \in (I^j)^{-2}$. Take $\hat{f}_{r+2}:M_{r+2}^\delta \rightarrow \hat{N}_{r+2}$ extending $f_{r+1}$. For all $k \in (j, \delta)$, note $\gtp(a_r/M_{r+1}^k, M_{r+2}^k)$ $\dnf$-does not fork over $M_r^j$ as $\calt^j \lesst \calt^k$, so by invariance and monotonicity, $p_k = \gtp(\hat{f}_{r+2}(a_r)/f_{r+1}[M_{r+1}^k], N_{r+1})$ $\dnf$-does not fork over $f_{r+1}[M_r^j]$ for all such $k$. Using that $f_{r+1}[M_r^j] \lek^u f_{r+1}[M_r^k] \lek f_{r+1}[M_{r+1}^k] \lek N_{r+1}$, by weak existence, there exists $q_k \in \gS(N_{r+1})$ where $p_k \subseteq q_k$ and $q_k$ $\dnf$-does not fork over $f_{r+1}[M_r^j]$. By weak uniqueness (since these all extend $p_{j+1}$), all the $q_k$ are equal, so say $q_k = q \in \gS(N_{r+1})$. Let $q = \gtp(d_r/N_{r+1}, N^*_{r+2})$.

    Note that $q\upharpoonright f_{r+1}[M_{r+1}^k] = p_k = \gtp(\hat{f}_{r+2}(a_r)/f_{r+1}[M_{r+1}^k], N_{r+1})$, so by $\Kkappalims$-universal continuity* in $\K$, $q \upharpoonright f_{r+1}[M_{r+1}^\delta] = \gtp(\hat{f}_{r+2}(a_r)/f_{r+1}[M_{r+1}^\delta], N_{r+1})$. By type equality, there is $\bar{N}_{r+2}$ where $N^*_{r+2} \lek \bar{N}_{r+2}$ and $g_{r+2}:\hat{N}_{r+2} \rightarrow \bar{N}_{r+2}$ such that $g_{r+2}(f_{r+2}(a_r)) = d_r$. Now take $N_{r+2}$ a $(\lambda, \gamma)$-limit over $\bar{N}_{r+2}$, and $f_{r+2} = g_{r+2} \circ \hat{f}_{r+2}$. $f_{r+2}$ and $N_{r+2}$ are as desired (see Figure \ref{limit_extensions_successor_case} for a picture of the construction).

    \begin{figure}[!ht]
\centering

\begin{circuitikz}
\tikzstyle{every node}=[font=\normalsize]
\node [font=\normalsize] at (2.75,7.5) {$M_r^j$};
\node [font=\normalsize] at (2.75,9.5) {$M_{r+1}^j$};
\node [font=\normalsize] at (4.5,8.75) {$f_r[M_r^j]$};
\node [font=\normalsize] at (4.5,10.75) {$f_{r+1}[M_{r+1}^j]$};
\draw [->, >=Stealth] (2.75,7.75) -- (2.75,9.25);
\draw [->, >=Stealth] (4.5,9) -- (4.5,10.5);
\node [font=\normalsize] at (2.75,11.5) {$M_{r+2}^j$};
\draw [->, >=Stealth] (3,11.75) -- (4.25,12.5);
\node [font=\normalsize] at (4.5,12.75) {$\hat{f}_{r+2} [M_{r+2}^j]$};
\draw [->, >=Stealth] (2.75,9.75) -- (2.75,11.25);
\draw [->, >=Stealth] (4.5,11) -- (4.5,12.5);
\draw [->, >=Stealth] (9,12.75) -- (10.5,12.75);
\node [font=\normalsize] at (11,12.75) {$\hat{N}_{r+2}$};
\node [font=\normalsize] at (11,10.75) {$N_{r+1}$};
\draw [->, >=Stealth] (9,10.75) -- (10.5,10.75);
\node [font=\normalsize] at (11,8.75) {$N_r$};
\draw [->, >=Stealth] (11,9) -- (11,10.5);
\draw [->, >=Stealth] (9,8.75) -- (10.5,8.75);
\draw [->, >=Stealth] (11.5,12.75) -- (13,12.75);
\node [font=\normalsize] at (13.5,12.75) {$\bar{N}_{r+2}$};
\draw [->, >=Stealth] (13.5,11) -- (13.5,12.5);
\node [font=\normalsize] at (13.5,10.75) {$N_{r+2}^*$};
\draw [->, >=Stealth] (11.5,10.75) -- (13,10.75);
\draw [->, >=Stealth] (13.75,13) -- (14.75,14);
\node [font=\normalsize] at (15,14.25) {$N_{r+2}$};
\node [font=\normalsize] at (14.75,13.25) {$(\lambda, \gamma)$};
\node [font=\normalsize] at (3,12.25) {$\hat{f}_{r+2}$};
\node [font=\normalsize] at (3.15,10.25) {$f_{r+1}$};
\node [font=\normalsize] at (3.25,8.25) {$f_r$};
\node [font=\normalsize] at (12.25,13) {$g_{r+2}$};
\draw [->, >=Stealth] (5.5,12) .. controls (5.75,14.25) and (10,14.25) .. (14.5,14.25) ;
\node [font=\normalsize] at (7.5,14.1) {$f_{r+2}$};
\node [font=\normalsize] at (8,12.75) {$\hat{f}_{r+2} [M_{r+2}^\delta]$};
\draw [->, >=Stealth] (8,11) -- (8,12.5);
\node [font=\normalsize] at (8,10.75) {$f_{r+1}[M_{r+1}^\delta]$};
\draw [->, >=Stealth] (8,9) -- (8,10.5);
\node [font=\normalsize] at (8,8.75) {$f_r[M_r^\delta]$};
\draw [->, >=Stealth] (5.5,12.75) -- (7,12.75);
\draw [->, >=Stealth] (5.5,8.75) -- (7,8.75);
\draw [->, >=Stealth] (5.5,10.75) -- (7,10.75);
\node [font=\normalsize] at (5.5,11.5) {$M_{r+2}^\delta$};
\draw [->, >=Stealth] (5.5,9.75) -- (5.5,11.25);
\node [font=\normalsize] at (5.5,9.5) {$M_{r+1}^\delta$};
\draw [->, >=Stealth] (5.5,7.75) -- (5.5,9.25);
\node [font=\normalsize] at (5.5,7.5) {$M_r^\delta$};
\draw [->, >=Stealth] (3,9.75) -- (4.25,10.5);
\draw [->, >=Stealth] (3,7.75) -- (4.25,8.5);
\draw [->, >=Stealth] (3.25,7.5) -- (5,7.5);
\draw [->, >=Stealth] (3.25,9.5) -- (5,9.5);
\draw [->, >=Stealth] (3.25,11.5) -- (5,11.5);
\draw [->, >=Stealth] (5.75,11.75) -- (7,12.5);
\draw [->, >=Stealth] (5.75,9.75) -- (7,10.5);
\draw [->, >=Stealth] (5.75,7.75) -- (7,8.5);
\node [font=\normalsize] at (6.5,11.75) {$\hat{f}_{r+2}$};
\node [font=\normalsize] at (6.5,9.75) {$f_{r+1}$};
\node [font=\normalsize] at (6.5,7.85) {$f_r$};
\end{circuitikz}

\caption{The successor case of Proposition \ref{limit-tower-exts-exist}}
\label{limit_extensions_successor_case}
\end{figure}

    \textbf{This is enough:} The argument is essentially the same from here as Proposition \ref{univ-tower-exts-exist}, replacing $M_i$ with $M_i^\delta$. Take $f = \bigcup_{i \in I} f_i :\bigcup_{i \in I} M_i^\delta \rightarrow \bigcup_{i \in I} N_i$, and extend it to an isomorphism $h: M' \rightarrow \bigcup_{i \in I} N_i$. Let $M_i' = h^{-1}[N_i]$ for all $i \in I$, and take $\calt' = \langle M_i' : i \in I \rangle ^\wedge \langle a_i : i \in I^{-2} \rangle$. $\calt'$ is a tower since $I$ is a well ordering, $\langle N_i : i \in I \rangle$ is increasing, and by (3) and (5) of the construction with base monotonicity. Additionally, $\calt'$ is strongly $(\lambda, \gamma)$-limit by (4) of the construction. $\calt^j \lesst \calt'$ for all $j < \delta$ from our construction, by essentially the same reasoning as in Proposition \ref{univ-tower-exts-exist}.

    For the moreover part, we check the clauses of Definition \ref{tower_ordering_def}. (1), (3), and (5) of Definition \ref{tower_ordering_def} are clear from the construction. (2) of Definition \ref{tower_ordering_def} follows from (3) of the construction. 
    
    It remains to show (4) of Definition \ref{tower_ordering_def} holds. For $i \in (I^\delta)^{-2}$, take $j<\delta$ large enough that $i\in (I^{j})^{-2}$. We have that $\gtp(a_i/M_{i+1}', M_{i+2}')$ $\dnf$-does not fork over $M_i^j$ as $\calt^j \lesst \calt'$, so $\gtp(a_i/M_{i+1}', M_{i+2}')$ $\dnf$-does not fork over $M_i^\delta$ by base monotonicity. This completes the proof.
\end{proof}

Finally we prove that all brilliant chains of towers can be extended by a universal tower.

\begin{corollary}\label{universal-extensions-of-brilliant-chains-exist}
    Suppose $\gamma < \lambda^+$ is limit, and $\langle\calt^j : j < \alpha\rangle$ is a brilliant chain of towers where $\alpha \geq 1$. Then there exists a strongly $(\lambda, \gamma)$-limit tower $\calt^{\alpha}$ such that $\langle \calt^j : j < \alpha + 1 \rangle$ is brilliant.
\end{corollary}

\begin{proof}
    First, suppose $\alpha = \beta+1$ for some $\beta < \lambda^+$. Then either $\calt^\beta$ falls into condition (1) or condition (2) of Definition \ref{brilliant-tower-def}. 
    
    In case (1), $\calt^\beta$ is universal. Apply Proposition \ref{univ-tower-exts-exist} to get $\calt^\alpha$ strongly $(\lambda, \gamma)$-limit where $\calt^\beta \lesst \calt^\alpha$. Since for all $j < \beta$ $\calt^j \lesst \calt^\beta \lesst \calt^\alpha$, we have $\calt^j \lesst \calt^\alpha$ by Lemma \ref{transitivity-on-universal}. As $\calt^\alpha$ is universal, $\langle \calt^j : j < \alpha + 1 \rangle$ is brilliant.

    In case (2), since (1) of Definition \ref{brilliant-tower-def} implies that  $\calt^{j+1}$ is universal for all $j < \beta$, there is a cofinal sequence $\langle j_t : t < \mu\rangle$ of $\beta$ such that $\mu \geq \kappa$ is regular and $\calt^{j_t}$ is universal for all $t < \mu$. By the `moreover' version of Proposition \ref{limit-tower-exts-exist}, there is a strongly $(\lambda, \gamma)$-tower $\calt^\alpha$ where $\calt^j \lesst \calt^\alpha$ for all $j \in \{j_t : t < \mu\} \cup \{\beta\}$. For all $j < \beta$, there is $t < \mu$ where $j < j_t$, so $\calt^j \lesst \calt^{j_t} \lesst \calt^\alpha$. Since $\calt^{j_t}$ is universal, $\calt^j < \calt^\alpha$ by Lemma \ref{transitivity-on-universal}. As $\calt^\alpha$ is universal, $\langle \calt^j : j < \alpha + 1 \rangle$ is brilliant.

    Otherwise, $\alpha$ is a limit. The argument is similar to the (2) case from before. Again since $\calt^{j+1}$ is universal for all $j < \alpha$, there is a cofinal sequence $\langle j_t : t < \mu\rangle$ of $\beta$ such that $\mu$ is regular and $\calt^{j_t}$ is universal for all $t < \mu$ (in this case, it is possible that $\mu < \kappa$). Again by Proposition \ref{limit-tower-exts-exist}, there is a strongly $(\lambda, \gamma)$-limit tower $\calt^\alpha$ where $\calt^{j_t} \lesst \calt^\alpha$ for all $t < \mu$. For $j < \alpha$, there is $t < \mu$ where $j < j_t$, so $\calt^j \lesst \calt^{j_t} \lesst \calt^\alpha$. By Lemma \ref{transitivity-on-universal}, we have that $\calt^j \lesst \calt^\alpha$ for all $j < \alpha$. As $\calt^\alpha$ is universal, $\langle \calt^j : j < \alpha \rangle$ is brilliant. This concludes the proof.
\end{proof}

The following result shows that we can modify these arguments to complete partial extensions of towers `up to isomorphism'.

\begin{proposition}\label{brilliant-extension-from-initial-segment-exist}
    Suppose $\langle \calt^j : j < \alpha + 1 \rangle$ is a brilliant sequence of towers. Assume $\calt^\alpha = \langle M_i:i < I \rangle ^\wedge \langle a_i : i \in I^{-2} \rangle$ is a universal tower, that $I_0 \subseteq I$ is an initial segment, and $\left( \langle \calt^j :j < \alpha + 1 \rangle \upharpoonright^* I_0 \right) ^\wedge \calt^*$ is brilliant where $\calt^* = \langle M_i^*:i \in I_0 \rangle ^\wedge \langle a_i : i \in (I_0)^{-2} \rangle$ is a universal tower. Then there is $\calt'= \langle M_i':i \in I \rangle ^\wedge \langle a_i : i \in I^{-2} \rangle$ a universal tower such that $\langle \calt^j : j < \alpha + 1 \rangle^\wedge \calt'$ is brilliant and an isomorphism $g:\bigcup_{i \in I_0} M_i' \underset{\bigcup_{i \in I_0}M_i}{\cong} \bigcup_{i \in I_0} M_i^*$ such that $g[M_i'] = M_i^*$ for each $i \in I_0$.
\end{proposition}

\begin{proof}
	Rather than giving a new proof from the ground up, we describe how to modify the proof of Corollary \ref{universal-extensions-of-brilliant-chains-exist} to get the result. Essentially the method is the same, but we only ever construct new towers from the minimal element of $I \setminus I_0$ onwards.
	
    For notational convenience, note that if $I_0 = \varnothing$, this holds by an application of Corollary \ref{universal-extensions-of-brilliant-chains-exist}. So assume $I_0 \neq \varnothing$.
    
    First note the analogous statements of Proposition \ref{univ-tower-exts-exist} and Proposition \ref{limit-tower-exts-exist} (that is, we suppose also that we have universal $\calt^*$ $\lesst$-extending the existing towers restricted to $I_0$, and we want our new tower $\calt'$ to satisfy $\calt' \upharpoonright I_0 = \calt^*$) hold by the same proofs, with the following modifications: in either proof, replace conditions (2) and (4) with 
    \begin{enumerate}
        \item[(2*)] For all $i \in I_0$, $N_i = M_i^*$ and $f_i$ is the inclusion $M_i \rightarrow M_i^*$.
        \item[(4*)] $N_i \lek^u N_{i+1}$ for all $i \in I^-$.
    \end{enumerate}
    where $M_i = M_i^\delta$ in the case of Proposition \ref{limit-tower-exts-exist}.
    
    (2*) determines the construction for $i \in I_0$, and from there both successor cases and the limit case are the same. Note that $\calt\upharpoonright I_0 \lesst \calt^*$ and the assumption that $\calt^*$ is universal ensure that the hypotheses all hold for $i \in I_0$. As before, define $f = \bigcup_{i < \alpha} f_i : \bigcup_{i \in I}M_i \rightarrow \bigcup_{i \in I}N_i$, and take $h : M' \rightarrow \bigcup_{i \in I} N_i$. Let $M_i' = h^{-1}[N_i]$, and take $\calt'= \langle M_i':i \in I \rangle ^\wedge \langle a_i : i \in I^{-2} \rangle$. By the same reasoning as before, $\calt'$ is an appropriate extension - in particular $\calt'$ is universal by condition (4*). And by (2*), $M_i' = h^{-1}[M_i^*]$ for all $i < \beta$, so $\calt' \upharpoonright \beta = \calt^*$ as desired. Take $g = h^{-1} \upharpoonright \bigcup_{i \in I_0} M_i'$, and we see that $\calt'$ and $g$ are as desired.
    
    So we can find universal extensions $\calt'$ for universal towers with $\calt^*$ as an initial segment, and similarly for universal chains. From here, the method of Corollary \ref{universal-extensions-of-brilliant-chains-exist}, using these modified versions of Proposition \ref{univ-tower-exts-exist} and Proposition \ref{limit-tower-exts-exist} in place of the originals, finishes the proof.
\end{proof}

And finally, another modified version. Note this is the only time we use $(\lambda, \theta)$-weak non-forking amalgamation:

\begin{proposition}\label{univ-tower-exts-exist-with-b}
    Suppose $\calt = \langle M_i : i <I \rangle ^\wedge \langle a_i : i \in I^{-2} \rangle$ is a strongly \emph{$(\lambda, \theta)$}-limit tower and $\gamma < \lambda^+$ is limit where $\cof(\gamma) \geq \kappa$. Let $i_0$ be minimal in $I$, and suppose $M_* \lek^u M_{i_0}$ is a $(\lambda, \geq \kappa)$-limit model, and $p \in \gS(M_{i_0})$ $\dnf$-does not fork over $M_*$. Then there exists a strongly $(\lambda, \gamma)$-limit tower $\calt' =  \langle M_i' : i <I \rangle ^\wedge \langle a_i : i \in I^{-2} \rangle$ such that 
    \begin{enumerate}
        \item $\calt \lesst \calt'$
        \item $b \in M_{i_0}'$
        \item $p = \gtp(b/M_{i_0}, M_{i_0}')$
        \item $\gtp(b/M_i, \bigcup_{i<\alpha} M_i')$ $\dnf$-does not fork over $M_*$ for all $i \in I$.
    \end{enumerate}
\end{proposition}

\begin{proof}
	By relabelling if necessary, we can assume $I = (\alpha, <)$ for some ordinal $\alpha < \lambda^+$, and $i_0 = 0$.
	
    Similar to before, we construct a $\lek$-increasing sequence $\langle N_i : i < \alpha \rangle$ of $(\lambda, \gamma)$-limit models, a $\subseteq$-increasing sequence of $\K$-embeddings $\langle f_i : i < \alpha \rangle$, and $c \in N_0$, such that 
    \begin{enumerate}
        \item $f_i:M_i \rightarrow N_i$ for each $i < \alpha$
        \item $M_0 \lek N_0$ and $f_{0}:M_0 \rightarrow N_0$ is the identity embedding
        \item $f_i[M_i] \lek^u N_i$ for each $i < \alpha$ (in fact, $N_i$ will be a $(\lambda, \gamma)$-limit model over $f_i[M_i]$)
        \item $N_i$ is a $(\lambda, \gamma)$-limit model over $\bigcup_{r<i} N_r$ for all $i \in [1, \alpha)$
        \item $\gtp(f_{i+2}(a_i)/f_{i+2}[M_{i+1}], N_{i+2})$ $\dnf$-does not fork over $f_i[M_i]$ for each $i < \alpha^{-2}$.
        \item $\gtp(c/f_i[M_i], N_i)$ $\dnf$-does not fork over $M_*$ for each $i < \alpha$.
    \end{enumerate}

    \textbf{This is possible:} We construct by recursion. 

    For $i = 0$, let $p = \gtp(c/M_0, N_0)$. Increasing $N_0$ if necessary, we can assume $N_0$ is a $(\lambda, \gamma)$-limit over $M_0$. Take $f_0:M_0 \rightarrow N_0$ to be the identity embedding.


    When $i$ is a limit, we have $N_r$ and $f_r$ for $r < i$. Let $M_i^0 = \bigcup_{r<i} M_r$, $N_i^0 = \bigcup_{r<i} N_r$, Let $f_i^0 = \bigcup_{r<i} f_r : M^0_i \rightarrow N_i^0$. By $\lambda$-AP, there exist $\hat{N}_i \in \K_\lambda$ where $N_i^0 \lek \hat{N}_i$ and $\hat{f}_i : M_i \rightarrow \hat{N}_i$ extending $f_i^0$.

    For each $r < i$, as $\gtp(c/f_i^0[M_r], N_r)$ $\dnf$-does not fork over $M_*$, using weak extension and $M_* \lek^u M_r$, there exists $q_r \in \gS(\hat{f}_i[M_i])$ an extension of $\gtp(c/f_r[M_r], N_r)$ which $\dnf$-does not fork over $M_*$. By weak uniqueness, since $q_r \upharpoonright f_0[M_0] = \gtp(c/M_0, N_0)$ $\dnf$-does not fork over $M_*$, these $q_r$ are all equal, so there is some $q \in \gS(\hat{f}_i[M_i])$ such that $q_r = q$ for all $r < i$. Since $q \upharpoonright f_r[M_r] = \gtp(c/f_r[M_r], N_r)$ which $\dnf$-does not fork over $M_*$ for all $r < i$, $q \upharpoonright f^0_i[M^0_i] = \gtp(c/f^0_i[M^0_i], N_i^0)$ by universal continuity* in $\K$. Say $q = \gtp(d/\hat{f}_i[M_i], N^*_i)$. As $\gtp(c/f_i^0[M_i^0], N_i^0) = \gtp(d/f^0_i[M_i^0], N^*_i)$, there exist $\bar{N}_i \in \K_\lambda$ with $N_i^0 \lek \bar{N}_i$ and $g_i : N_i^* \rightarrow \bar{N}_i$ fixing $\hat{f}_i[M_i^0]$ such that $g_i(d) = c$. Take $N_i$ a $(\lambda, \gamma)$-limit model over $\bar{N}_i$. Then set $f_i = g_i \circ \hat{f}_i : M_i \rightarrow N_i$.

    \begin{figure}[!ht]
\centering

\begin{circuitikz}
\tikzstyle{every node}=[font=\normalsize]
\node [font=\normalsize] at (4.25,5.5) {$M^*$};
\node [font=\normalsize] at (5.5,7) {$M_i^0$};
\node [font=\normalsize] at (5.5,8.75) {$M_i$};
\node [font=\normalsize] at (7.25,7) {$f_i^0[M_i^0]$};
\node [font=\normalsize] at (9.75,7) {$N_i^0$};
\node [font=\normalsize] at (9.75,8.75) {$\hat{N}^i$};
\node [font=\normalsize] at (7.5,8.75) {$\hat{f}_i[M_i]$};
\node [font=\normalsize] at (6.25,9.25) {$\hat{f}_i$};
\node [font=\normalsize] at (7.5,10.5) {$N_i^*$};
\node [font=\normalsize] at (12,10.5) {$\bar{N}_i$};
\draw [->, >=Stealth] (4.25,5.75) -- (5.25,6.75);
\draw [->, >=Stealth] (5.5,7.25) -- (5.5,8.5);
\draw [->, >=Stealth] (5.75,7) -- (6.75,7);
\draw [->, >=Stealth] (8.25,7) -- (9.25,7);
\draw [->, >=Stealth] (8.25,8.75) -- (9.25,8.75);
\draw [->, >=Stealth] (5.75,8.75) -- (6.75,8.75);
\draw [->, >=Stealth] (7.75,10.5) -- (11.75,10.5);
\draw [->, >=Stealth] (10,7) .. controls (12,7) and (12,8.75) .. (12,10.25) ;
\draw [->, >=Stealth] (12.5,10.75) -- (13.5,11.75);
\draw [->, >=Stealth] (7.5,7.25) -- (7.5,8.5);
\draw [->, >=Stealth] (9.75,7.25) -- (9.75,8.5);
\draw [->, >=Stealth] (7.5,9) -- (7.5,10.25);
\node [font=\normalsize] at (9.75,10.75) {$g_i$};
\node [font=\normalsize] at (13.75,12) {$N_i$};
\node [font=\normalsize] at (13.5,11) {$(\lambda, \gamma)$};
\draw [->, >=Stealth] (5.5,9) .. controls (5.75,11.5) and (7,12) .. (13.5,12);
\node [font=\normalsize] at (8,12) {$f_i$};
\end{circuitikz}

\caption{The limit case of Proposition \ref{univ-tower-exts-exist-with-b}}
\label{extending_with_b_limit_case}
\end{figure}

    When $i = \alpha + 1$ for $\alpha$ either $0$ or limit, do the same construction, but replace $M_i^0$, $N_i^0$, and $f_i^0$ by $M_\alpha$, $N_\alpha$, and $f_\alpha$ respectively.

    For $i = r + 2$ for some $r \in \alpha^{-2}$, we are given $N_{r+1}$ and $f_{r+1}$. Take some isomorphism $\hat{f}_{r+2}:M_{r+2} \rightarrow \hat{M}_{r+2}$ extending $f_{r+1}:M_{r+1} \rightarrow f_{r+1}[M_{r+1}]$. We have that $\gtp(\hat{f}_{r+2}(a_r)/\hat{f}_{r+2}[M_{r+1}], \hat{M}_{r+2})$ $\dnf$-does not fork over $f_r[M_r]$ by invariance, and $\gtp(c/\hat{f}_{r+2}[M_{r+1}], N_{r+1})$ $\dnf$-does not fork over $f_r[M_r]$ by base monotonicity. Since $\calt$ is a strongly $(\lambda, \theta)$-limit tower, $f_{r+1}[M_{r+1}]$ is a $(\lambda, \theta)$-limit model over $f_r[M_r]$. So by $(\lambda, \theta)$-weak non-forking amalgamation of $\dnf$, there exists $\bar{N}_{r+2} \in \K_\lambda$ and a map $g_{r+2}:\hat{M}_{r+2} \rightarrow \bar{N}_{r+2}$ fixing $\hat{f}_{r+2}[M_{r+1}]$ such that $\gtp(g_{r+2}(\hat{f}_{r+2}(a_r))/N_{r+1}, \bar{N}_{r+2})$ and $\gtp(c/g_{r+2}[f_{r+1}[M_{r+1}]], \bar{N}_{r+2})$ $\dnf$-do not fork over $f_r[M_r]$. Take $N_{r+2}$ any $(\lambda, \kappa)$-limit model over $\bar{N}_{r+2}$, and let $f_{r+2} = g_{r+2} \circ \hat{f}_{r+2}$. So far we have established conditions (1)-(5) of the construction.

    \begin{figure}[!ht]
\centering

\begin{circuitikz}
\tikzstyle{every node}=[font=\normalsize]
\node [font=\normalsize] at (2.75,6.75) {$M_*$};
\draw [->, >=Stealth] (3,7) -- (3.75,7.75);
\node [font=\normalsize] at (4,8) {$M_r$};
\draw [->, >=Stealth] (4.5,8) -- (6,8);
\node [font=\normalsize, color=blue] at (6.75,8) {$f_r[M_r]$};
\draw [->, >=Stealth] (4,8.25) -- (4,9.75);
\draw [->, >=Stealth, color=blue] (6.75,8.25) -- (6.75,9.75);
\draw [->, >=Stealth] (4,10.25) -- (4,11.75);
\node [font=\normalsize] at (4,10) {$M_{r+1}$};
\draw [->, >=Stealth] (4.5,10) -- (5.75,10);
\node [font=\normalsize, color=blue] at (6.75,10) {$f_{r+1}[M_{r+1}]$};
\draw [->, >=Stealth, color=blue] (6.75,10.25) -- (6.75,11.75);
\node [font=\normalsize] at (4,12) {$M_{r+2}$};
\draw [->, >=Stealth] (4.5,12) -- (6,12);
\node [font=\normalsize, color=blue] at (6.75,12) {$\hat{M}_{r+2}$};
\draw [->, >=Stealth, color=red, dashed] (7.5,12) -- (9,12);
\draw [->, >=Stealth, color=blue] (7.75,10) -- (9,10);
\draw [->, >=Stealth] (7.5,8) -- (9,8);
\node [font=\normalsize, color=red] at (9.5,12) {$\bar{N}_{r+2}$};
\node [font=\normalsize, color=blue] at (9.5,10) {$N_{r+1}$};
\node [font=\normalsize] at (9.5,8) {$N_r$};
\draw [->, >=Stealth, color=red, dashed] (9.5,10.25) -- (9.5,11.75);
\draw [->, >=Stealth] (9.5,8.25) -- (9.5,9.75);
\draw [->, >=Stealth] (9.75,12.25) -- (10.5,13);
\node [font=\normalsize] at (10.75,13.25) {$N_{r+2}$};
\node [font=\normalsize] at (10.75,12.5) {$(\lambda, \gamma)$};
\node [font=\normalsize, color=red] at (8.25,12.25) {$g_{r+2}$};
\node [font=\normalsize] at (5.15,12.35) {$\hat{f}_{r+2}$};
\draw [->, >=Stealth] (4,12.25) .. controls (4,13.25) and (7.25,13.25) .. (10.25,13.25) ;
\node [font=\normalsize] at (7.25,13.5) {$f_{r+2}$};
\node [font=\normalsize] at (5.15,10.25) {$f_{r+1}$};
\node [font=\normalsize] at (5.15,8.25) {$f_r$};
\end{circuitikz}

\caption{The $r + 2$ case of Proposition \ref{univ-tower-exts-exist-with-b}. $(\lambda, \theta)$-weak non-forking amalgamation on the blue components results in the dashed red components.}
\label{extending_with_b_successor_case}
\end{figure}

    Note that $\gtp(c/f_{r+2}[M_{r+1}], \bar{N}_{r+2})$ $\dnf$-does not fork over $f_{r+2}[M_r]$, and $\gtp(c/f_{r+2}[M_{r+1}], N_{r+2})$ $\dnf$-does not fork over $M_*$ by the induction hypothesis. We have that $M_* \lek f_{r+2}[M_r] \lek^u f_{r+2}[M_{r+1}] \lek f_{r+2}[M_{r+2}]$, so by weak transitivity $\gtp(c/f_{r+2}[M_{r+1}], \bar{N}_{r+2})$ $\dnf$-does not fork over $M_*$. Hence we satisfy condition (6) of the construction also. This completes the construction.

    \textbf{This is enough:} Similar to Proposition \ref{univ-tower-exts-exist}, take $f = \bigcup_{i \in I} f_i :\bigcup_{i \in I} M_i \rightarrow \bigcup_{i \in I} N_i$, and extend it to an isomorphism $h: M' \rightarrow \bigcup_{i \in I} N_i$. Let $M_i' = h^{-1}[N_i]$ for all $i < \delta$, and take $\calt' = \langle M_i' : i \in I \rangle ^\wedge \langle a_i : i \in I^{-2} \rangle$. Take $b = h^{-1}(c)$.

    Exactly as in Proposition \ref{univ-tower-exts-exist}, $\calt'$ is a $(\lambda, \gamma)$-limit tower and $\calt \lesst \calt'$. Because $b \in M_0'$, $c \in N_0$, and $f$ fixes $M_0$, we have $\gtp(b/M_0, M_0') = f^{-1}(\gtp(c/M_0, N_0)) = f^{-1}(p) = p$. Finally $\gtp(b/M_i, M_i')$ $\dnf$-does not fork over $M_*$ by (6) of the construction, as desired.
    
\end{proof}

The following simple corollary will be easier to use when dealing with brilliant chains.

\begin{corollary}\label{any-brill-chain-exts-exist-with-b}
    Suppose $\langle\calt^j : j < \alpha+1 \rangle$ is brilliant where $\calt^\alpha = \langle M_i : i \in I \rangle ^\wedge \langle a_i : i \in I^{-2} \rangle$, $\gamma < \lambda^+$ is limit with $\cof(\gamma) \geq \kappa$, and $i_0 \in I$ is minimal. Suppose $M_* \lek^u M_{i_0}$ is a $(\lambda, \geq \kappa)$-limit model, and $p \in \gS(M_{i_0})$ $\dnf$-does not fork over $M_*$, then there exists strongly $(\lambda, \gamma)$-limit tower $\calt' = \langle M_i' : i <\beta \rangle ^\wedge \langle a_i : i \in I^{-2} \rangle$ such that 
    \begin{enumerate}
        \item $\langle\calt^j : j < \alpha+1 \rangle^\wedge \calt'$ is brilliant
        \item $b \in M_{i_0}'$
        \item $p = \gtp(b/M_{i_0}, M_{i_0}')$
        \item $\gtp(b/M_i, \bigcup_{i<\alpha} M_i')$ $\dnf$-does not fork over $M_*$ for all $i < \beta$.
    \end{enumerate}
\end{corollary}

\begin{proof}
    Take any strongly $(\lambda, \theta)$-limit tower $\calt^* = \langle M_i^* : i \in I \rangle ^\wedge \langle a_i : i \in I^{-2} \rangle$ where $\langle\calt^j : j < \alpha+1 \rangle ^\wedge \calt^*$ is brilliant by Corollary \ref{universal-extensions-of-brilliant-chains-exist}. Let $p_{i_0} \in \gS(M_{i_0}^*)$ be the $\dnf$-non-forking extension of $p$ over $M_*$, which exists by weak extension. Then by Proposition \ref{univ-tower-exts-exist-with-b} applied to $p_0$ and $\calt^*$, there is $\calt' = \langle M_i' : i \in I \rangle ^\wedge \langle a_i : i \in I^{-2} \rangle$ where $\calt^* \lesst \calt'$, $b \in M_{i_0}'$, $p = \gtp(b/M_{i_0}, M_{i_0}')$, and $\gtp(b/M_i, \bigcup_{i\in I} M_i')$ $\dnf$-does not fork over $M_*$ for all $i \in I$. Since $\langle\calt^j : j < \alpha+1 \rangle ^\wedge \calt^*$ is brilliant, $\calt^* \lesst \calt'$, and $\calt^*$ is universal, $\langle\calt^j : j < \alpha+1 \rangle ^\wedge \calt'$ is brilliant by Lemma \ref{transitivity-on-universal}. $\calt'$ satisfies the other relevant conditions (2), (3), (4) of the statement by the corresponding conditions from Proposition \ref{univ-tower-exts-exist-with-b}.
\end{proof}

\subsection{Full towers}\label{full-tower-subsection}

We now have all the tools needed to construct a sequence of towers $\langle \calt^j : j \leq \alpha \rangle$ as in Figure \ref{construction_overview} such that $\cof(\alpha) = \delta_1$ and where the top row witnesses that $M^\alpha_{\delta_2}$ is a $(\lambda, \delta_1)$ limit model - given such $\alpha$, get $\calt^0$ by Lemma \ref{towers-exist}, use Corollary \ref{universal-extensions-of-brilliant-chains-exist} to get $\calt^j$ for $j < \alpha$, and take $\calt^\alpha$ to be the union of the previous towers. We now begin developing the machinery to guarantee that the final tower we construct witnesses that its largest model is a $(\lambda, \delta_2)$-limit model - that is, that the final tower in Figure \ref{construction_overview} is universally increasing and continuous at $\delta_2$.

A first attempt at making the final tower universal might be to ensure that a cofinal sequence of the towers $\langle \calt^j : j < \alpha \rangle$ we construct are universal. Unfortunately, a union of universal towers will not necessarily be universal. Because of this, we will define a notion that will make our towers (often enough) universal, while being preserved under tower unions.

One complication this adds is that we will need to add new elements between the indexes of our towers to find full extensions. This is why we choose to index our towers by general well ordered sets rather than just ordinals.

\begin{definition}
    Suppose $\calt = \langle M_i : i \in I \rangle^\wedge \langle M_i : i \in I^{-2} \rangle$ is a tower, and $I_0 \subseteq I$. We say $\calt$ is \emph{$I_0$-full} if for every $i \in (I_0)^-$ and every $p \in \gS(M_i)$, there exists $k \in [i, i +_{I_0} 1)_I$ such that $\gtp(a_k/M_{k+_I 1}, M_{k +_I 2})$ $\dnf$-does not fork over some $M' \lek^u M_i$ and extends $p$.
\end{definition}

\begin{remark}
	Implicitly we are assuming $k \in I^{-2}$ since $\gtp(a_k/M_{k+_I 1}, M_{k+_I 2})$ is mentioned.
\end{remark}

That is, $\calt$ is $I_0$-full if it realises non-forking extensions of all types over $M_i$ as some $a_k$, before level $i+_{i_0}1$. By realising all types over $M_i$ frequently enough up our tower, this allows us to ensure there is a universally increasing sequence inside it: given $i \in I_0$, since $M_{i+_{I_0}1}$ realises all types over $M_i$, $M_{i +_{I_0} \lambda}$ is universal over $M_i$ by Fact \ref{limit-models-exist} (provided $i+_{I_0} \lambda$ exists). So while technically a full tower is not necessarily universal, by choosing $I$ and $I_0$ correctly, we can ensure $\calt$ contains a cofinal sequence of universally increasing models, which will be enough.

As alluded to above, unions of $I_0$-full towers are also $I_0$-full. This is why we choose to realise the types as non-forking extensions by the $a_k$'s.

\begin{lemma}\label{long-unions-of-full-towers-are-full}
    If $\langle \calt^j : j \leq \delta\rangle$ is a $\lesst$-increasing sequence of universal $I_0$-full towers continuous at $\delta$ and $\cof(\delta) \geq \kappa$, then $\bigcup_{i < \delta} \calt^j$ is $I_0$-full.
\end{lemma}

\begin{proof}
    For $j \leq \delta$, let $\calt^j = \langle M_i^j : i \in I^j \rangle ^\wedge \langle a_i : i \in (I^j)^{-2} \rangle$. Note $I_0 \subseteq I^j$ for all $j \leq \delta$. We will use $+$ to denote addition in $I^j$ for $j \leq \delta$ (which is unambiguous because of tower ordering where applied), and distinguish $+_{I_0}$.

    Suppose $i \in (I_0)^-$, and $p \in \gS(M_i^\delta)$. By $(\geq \kappa)$-local character, as $\langle M_i^j : j < \delta \rangle$ is $\lek^u$-increasing, there exists $j < \delta$ such that $p$ $\dnf$-does not fork over $M_i^j$.
    
    Since $\calt^j$ is $I_0$-full, there is $k \in [i, i+_{I_0}1)_{I^j}$ such that $\gtp(a_k/M^j_{k+1}, M^j_{k+2})$ $\dnf$-does not fork over $M' \lek^u M_i^j$ and extends $p \upharpoonright M_i^j$.

    Now, let $q = \gtp(a_k/M_{k+1}^\delta, M_{k+2}^\delta)$. $q$ $\dnf$-does not fork over $M_k^j$ as $\calt^j \lesst \calt^\delta$. And we have that $q \upharpoonright M_{k+1}^j = \gtp(a_k/M_{k+1}^j, M_{k+2}^j)$ $\dnf$-does not fork over $M'$, and $M' \lek M_k^j \lek^u M_{k+1}^j \lek M_{k+1}^\delta$ (using that $\calt^j$ is universal for the universal extension). By weak transitivity, $q$ $\dnf$-does not fork over $M'$. Note that $p \upharpoonright M_i^j = \gtp(a_k/M_i^j, M_{k+2}^\delta) = q \upharpoonright M_i^j$, and both $\dnf$-do not fork over $M' \lek^u M^j_i$. So by weak uniqueness, $q \upharpoonright M_i^\delta = p$. So $q = \gtp(a_i/M_{k+1}^\delta, M_{k+2}^\delta)$ extends $p$ and $\dnf$-does not fork over $M' \lek^u M_i^\delta$, as desired. Therefore $\calt^\delta$ is $I_0$-full.

\end{proof}

\begin{notation}
    Whenever $(I, <_I), (J, <_J)$ are linearly ordered sets, use $(I \times J, <_{I \times J})$ to denote the lexicographic ordering.
\end{notation}

We now show that all towers of appropriate indexes have full extensions. The formality of the following proof may obscure the main idea behind it, so we'll give some intuition. The gist of the argument is that (after extending to a strongly $(\lambda, \kappa)$-limit tower, or just strongly universal), we push a new universal sequence of models between $M_{(i, \alpha)}$ and $M_{(i+1, 0)}$ that realise all the non-forking extensions of types over $M_{(i, 0)}$ needed to make it into a full tower. This can be regarded as adding $\lambda$-many new levels in to the tower between $i$ and $i+_{I_0}1$ for all $i \in I_0^-$.

\begin{lemma}\label{full-extensions-exist}
    Suppose $\alpha$ and $\gamma$ are limit ordinals with $\alpha < \gamma < \lambda^+$ and $\cof(\gamma) = \lambda$. Suppose $\calt = \langle N_{(i, k)} : (i, k) \in I \times \alpha \rangle ^\wedge \langle a_{(i, k)} : (i, k) \in I \times \alpha \rangle$ is universal. Then there exists a universal and $I \times \{0\}$-full tower $\calt' = \langle M_{(i, k)} : (i, k) \in I \times \gamma \rangle ^\wedge \langle a_{(i, k)} : (i, k) \in I \times \gamma \rangle$ where $\calt \lesst \calt'$.
\end{lemma}

\begin{proof}
    By Corollary \ref{univ-tower-exts-exist}, there exists a  strongly $(\lambda, \kappa)$-limit tower $\calt^* = \langle M_{(i, k)} : (i, k) \in I \times \alpha \rangle ^\wedge \langle a_{(i, k)} : (i, k) \in I \times \alpha \rangle$ where $\calt \lesst \calt^*$. For each $i \in I$, we define $M_{(i, k)}$ by induction on $k \in [\alpha, \lambda)$.

    First suppose $i$ is non-final in $I$, that is, $i \in I^-$. Let $\langle p_k : k \in [\alpha, \gamma)\rangle$ be an enumeration of $\gS(M_{(i, 0)})$, possibly with repeats. This is possible as $\cof(\gamma) = \lambda$ and $\K$ is stable in $\lambda$\footnote{note that if our independence relation were defined only on non-algebraic types, we could run into trouble here without adding a NMM assumption - see Subsection \ref{subsection-non-alg-jep-nmm}}. As $M_{(i, 0)}$ is a $(\lambda, \geq\kappa)$-limit, by $(\geq \kappa)$-local character there exists $M_{(i,k)}^0 \lek^u M_{(i, 0)}$ such that $p_k$ $\dnf$-does not fork over $M_{(i, k)}^0$ for $k \in [\alpha, \gamma)$. Before we define $M_{(i, k)}$ and $a_{(i, k)}$, define $M_{(i, k)}'$ and $a_{(i, k)}'$ for $k < \gamma$ as follows:

    For $k < \alpha$, $M_{(i, k)}' = M_{(i, k)}$ and $a_{(i, k)}' = a_{(i, k)}$.

    For $k\in[\alpha, \gamma)$ limit, take any $M_{(i, k)}'$ a $(\lambda, \kappa)$-limit over $\bigcup_{r < k} M_{(i, r)}'$. Also take $M_{(i, k+1)}'$ any $(\lambda, \kappa)$-limit over $M_{(i, k)}'$ ($+$ here is ordinal addition).

    For $k \in [\alpha, \gamma)$ successor where we have defined $a_r', M_{(i, s)}'$ for $r < k$, $s \leq k+1$, let $M_{(i, k+2)}'$ be any $(\lambda, \kappa)$-limit over $M_{(i, k+1)}'$. By weak extension, there is $q_k \in \gS(M_{(i, k+1)}')$ which $\dnf$-does not fork over $M_{(i, k)}^0$ and extends $p_k$. Let $a_k' \in M_{(i, k+2)}'$ be any realisation of $q_k$ in $M_{(i, k+2)}'$ (one exists as $M_{(i, k+1)}' \lek^u M_{(i, k+2)}'$). This completes the construction of the $M_{(i, k)}'$ and $a_{(i, k)}'$.

    Let $M_{(i, \gamma)}'= \bigcup_{k < \gamma} M_{(i, k)}'$. Since $\calt^*$ is strongly $(\lambda, \kappa)$-limit, $M_{(i+1, 0)}$ is universal over $\bigcup_{k < \alpha} M_{(i, k)}$, so there exists $g_i : M_{(i, \gamma)}' \rightarrow M_{(i+1, 0)}$ fixing $\bigcup_{k < \alpha} M_{(i, k)}$. Take $M_{(i, k)} = g_i(M_{(i, k)}')$ and $a_i = g_i(a_{(i, k)}')$ for $i \in [\alpha, \gamma)$ (note these equalities hold for $k<\alpha$).
    
    If $I$ has a final element $i_*$, construct $M'_{(i_*, k)}$ exactly as before for $k < \gamma$, and set $M_{(i_*, k)} = M'_{(i_*, k)}$ for all $k < \gamma$.

    Define $\calt' = \langle M_{(i, k)} : (i, k) \in I \times \gamma \rangle ^\wedge \langle a_{(i, k)} : (i, k) \in I \times \gamma \rangle$. From our construction, $\calt' \upharpoonright I \times \alpha = \calt^*$, so as $\calt \lesst \calt^*$, we have $\calt \lesst \calt'$. $\langle M_{(i, \alpha)} : (i, k) \in I \times \gamma \rangle$ is $\lek^u$-increasing, so $\calt'$ is universal. And $\calt'$ is $I \times \{0\}$-full as we realise $p_k$ as $\gtp(a_k/M_{(i, k+1)}, M_{(i, k+2)})$ which $\dnf$-does not fork over $M_{(i, k)}^0 \lek^u M_{(i, 0)}$.
\end{proof}

\subsection{Reduced chains of towers}\label{reduced-subsection}

We now tackle the problem of ensuring our final tower in Figure \ref{construction_overview} is continuous at high cofinalities. Similar to what we did with full towers, we define a notion of \emph{reduced} chains of towers, whose final towers are always continuous at high cofinalities, and are preserved by unions.

\begin{definition}
    Let $\calc = \langle \calt^j : j < \alpha + 1 \rangle$ be a brilliant chain of towers where $\calt^j$ is indexed by $I^j$ for all $j < \alpha + 1$, and $\calt^\alpha = \langle M_i^\alpha : i \in I^\alpha \rangle ^\wedge \langle a_i : i \in (I^\alpha)^{-2} \rangle$. We say $\calc$ is \emph{reduced} if and only if for every universal tower $\calt^{\alpha+1} = \langle M_i^{\alpha+1} : i \in I \rangle ^\wedge \langle a_i : i \in I^{-2} \rangle$ such that $\langle \calt^j : j < \alpha + 2 \rangle$ is brilliant, and every $r < s \in I$, $M_s^\alpha \cap M_r^{\alpha+1} = M_r^\alpha$.
\end{definition}

\begin{lemma}\label{reduced-extensions-exist}
    If $\calc = \langle \calt^j : j < \alpha + 1 \rangle$ is a brilliant chain of towers where $\calt^\alpha$ is universal and indexed by $I$, then there exists a non-empty brilliant chain of towers $\calc'$ with all towers indexed by $I$ such that $\calc ^\wedge \calc'$ is a reduced and brilliant chain of towers.
\end{lemma}

\begin{proof}
    Suppose no such $\calc'$ exists for contradiction. Then construct by recursion a sequence of towers $\langle \bar{\calt}^j : j < \lambda^+ \rangle$ with $\bar{\calt}^j = \langle M_i^j : i \in I \rangle ^\wedge \langle a_i : i \in I^{-2} \rangle$ where for all $j < \lambda^+$,
    \begin{enumerate}
        \item $\calc^\wedge\langle \bar{\calt}^k : k < j \rangle$ is brilliant
        \item if $j < \lambda^+$ and $\cof(j) < \kappa$ (including successor $j$), then $\bar{\calt}^j$ is universal
        \item if $\cof(j) \geq \kappa$, then $\bar{\calt}^j = \bigcup_{k < j} \bar{\calt}^k$
        \item there exist $r_j < s_j \in I$ such that $M_{r_j}^{j+1} \cap M_{s_j}^j \neq M_{r_j}^j$.
    \end{enumerate}

    \textbf{This is possible:} We proceed by induction on $j < \lambda^+$. Take $\bar{\calt}^0$ universal with $\calc ^\wedge \bar{\calt}^0$ brilliant by Proposition \ref{universal-extensions-of-brilliant-chains-exist}. 
    
    If $\cof(j) \geq \kappa$, let $\bar{\calt}^j = \bigcup_{k < j} \bar{\calt}^k$. This fits clause (2) of Definition \ref{brilliant-tower-def}, so $\calc^\wedge\langle \bar{\calt}^k : k < j + 1 \rangle$ is brilliant. 

    If $j$ is limit and $\cof(j) < \kappa$, then take any $\bar{\calt}^j$ universal with $\calc ^\wedge \langle\bar{\calt}^k : k < j + 1 \rangle$ brilliant by Proposition \ref{universal-extensions-of-brilliant-chains-exist}.
    
    For $j+1$, $\calc^\wedge\langle \bar{\calt}^k : k < j+1 \rangle$ is brilliant, so by our assumption not reduced. Hence there exists a universal $\bar{\calt}^{j+1}$ where $\calc^\wedge\langle \bar{\calt}^k : k < j+2 \rangle$ is brilliant and some $r_j < s_j \in I$ such that $M_{r_j}^{j+1} \cap M_{s_j}^j \neq M_{r_j}^j$. This completes the construction.

    \textbf{This is enough:} 
    For $i \in I$ and $j < \lambda^+$, let 
    \[
        N_i^j = \begin{cases}
            M_i^j & \text{if $j$ is not a limit}\\
            \bigcup_{k<j} M_i^k & \text{if $j$ is a limit}
    \end{cases}
    \]

    This is just to turn $\langle M_i^j : j < \lambda^+\rangle $ into continuous sequences for each $i \in I$. Note that $M_i^j = N_i^j$ when $\cof(j) \geq \kappa$ by our construction.
    
    For $i \in I$, let $N_i^{\lambda^+} = \bigcup_{j < \lambda^+} N_i^j$ and for $j < \lambda^+$ $N_I^j = \bigcup_{i \in I} N_i^j$. For $i \in I$, define $C_i = \{j < \lambda^+ : N_i^{\lambda^+} \cap N_I^j = N_i^j \}$. It is straightforward to show $C_i$ is closed in $\lambda^+$. $C_i$ is also unbounded in $\lambda^+$; note that if $j < \lambda^+$, there is $j' \in (j, \lambda^+)$ such that $N_i^{\lambda^+} \cap N_I^j \subseteq N_i^{j'}$. Using this, we can for any $j_0 < \lambda^+$ find $j_n$ for $n \in [1, \omega)$ such that $\langle j_n : n < \omega \rangle$ is increasing and $N_i^{\lambda^+} \cap N_I^{j_n} \subseteq N_i^{j_{n+1}}$. Then $\sup_{n \in \omega} j_n \in C_i$.

    As $\operatorname{otp}(I) < \lambda^+$,  $|I| < \lambda^+$, so taking $C = \bigcap\{C_i : i \in I \}$, $C$ is closed and unbounded in $\lambda^+$ also. Since the set of all $\{j < \lambda^+ : \cof(j) = \kappa\}$ is stationary in $\lambda^+$, there is $j \in C$ where $\cof(j) = \kappa$. Then for all $i \in I$, $j \in C_i$, so $N_i^{\lambda^+} \cap N_I^j = N_i^j$. But then in particular, $N_{r_j}^{j+1} \cap N_{s_j}^j = N_{r_j}^j$. Since $\cof(j) \geq \kappa$ and $j+1$ is a successor, $N_{r_j}^{j+1} = M_{r_j}^{j+1}$, $N_{s_j}^j = M_{s_j}^j$, and $M_{r_j}^j = N_{r_j}^j$, so $M_{r_j}^{j+1} \cap M_{s_j}^j = M_{r_j}^j$, contradicting condition (4) of our construction.
\end{proof}

\begin{remark}
	In \cite{bema} and \cite{bovan}, notions of reduced \emph{towers} are used, rather than reduced chains. While their presentation is simpler, it creates issues in our context. More precisely, in those cases the tower ordering is transitive, so to construct the sequence of towers $\langle \bar{\calt}^j : j < \lambda^+ \rangle$, they only need $\bar{\calt}^{j+1}$ to extend $\bar{\calt}^j$. For us though, if $j$ is limit, it may be possible to find universal $\bar{\calt}^{j+1}$ where $\bar{\calt}^j \lesst \bar{\calt}^{j+1}$ but where we do not have $\bar{\calt}^k \lesst \bar{\calt}^{j+1}$, for some $k < j$ (or similarly for $\calt^k$ rather than $\bar{\calt}^k$). Thus we need to work with chains of towers, rather than individual towers.
\end{remark}

\begin{lemma}\label{unions-of-reduced-towers-are-reduced}
    Suppose $\delta < \lambda^+$ is limit where $\cof(\delta) \geq \kappa$ and $\langle \calt^j : j < \delta+1 \rangle$ is a brilliant chain of towers, continuous at $\delta$. Suppose there is a cofinal sequence $\langle j_k : k < \mu \rangle$ of $\delta$ where $\langle \calt^{j} : j < j_k \rangle$ is reduced for all $k < \mu$. Then $\langle \calt^j : j < \delta + 1 \rangle$ is reduced.
\end{lemma}

\begin{proof}
    Let $\calt^j = \langle M_i^j : i \in I^j \rangle ^\wedge \langle a_i : i \in (I^j)^{-2} \rangle$ for all $j \leq \delta$. Suppose $\calt' = \langle M_i':i \in I^\delta \rangle ^\wedge \langle a_i : i \in (I^\delta)^{-2} \rangle$ is universal where $\calt^j \lesst \calt'$ for all $j \leq \delta$. Suppose $r < s \in I^\delta$. Take $k_0 < \mu$ such that $r, s \in I^{j_{k_0}}$. Then
    \begin{align*}
        M^\delta_s \cap M_r' &= \Big( \bigcup_{k \in [k_0, \mu)} M^{j_k}_s \Big) \cap M_r'&\\
        &= \bigcup_{k \in [k_0, \mu)} (M^{j_k}_s \cap M_r')&\\
        &= \bigcup_{k \in [k_0, \mu)} M^{j_k}_r \qquad \qquad &\text{(as $\langle \calt^{j} : j \leq j_k \rangle$ is reduced, $\calt'$ is universal,}\\
        &&  \text{ and $\langle \calt^{j} : j \leq j_k \rangle ^\wedge \calt'$ is brilliant)}\\
        &= M^\delta_r&
    \end{align*}
    So $\calt^\delta$ is reduced as desired.
\end{proof}

Over the following series of results, we will show that the final tower of a reduced chain of towers is always continuous at cofinalities $\geq \kappa$.

\begin{lemma}\label{extensions-with-non-forking-element-at-the-bottom}
    Suppose $\delta$ is a limit ordinal, $\langle\calt^j : j < \alpha+1 \rangle$ is a brilliant chain of towers with $\calt^\alpha = \langle M_i:i < \delta + 1 \rangle ^\wedge \langle a_i : i < \delta \rangle$, $M_* \lek^u M_0$ is a $(\lambda, \geq \kappa)$-limit model, and $b \in M_\delta$ such that $\gtp(b/M_i, M_\delta)$ $\dnf$-does not fork over $M_*$ for all $i < \delta$. Then there exists a universal tower $\calt' = \langle M_i:i < \delta + 1 \rangle ^\wedge \langle a_i : i < \delta \rangle$ such that $\langle\calt^j : j < \alpha+1 \rangle^\wedge \calt'$ is brilliant and $b \in M_0'$.
\end{lemma}

\begin{proof}
    By Corollary \ref{any-brill-chain-exts-exist-with-b} applied to $\langle\calt^j : j < \alpha+1 \rangle \upharpoonright \delta$, $b$, and $M^*$, there exists $\calt^* = \langle M_i^*:i < \delta \rangle ^\wedge \langle a_i : i < \delta \rangle$ and $b^* \in M_0^*$ such that $\gtp(b^*/M_0, M_0^*) = \gtp(b/M_0, M_\delta)$ and $\gtp(b^*/M_i, \bigcup_{i<\delta}M_i^*)$ $\dnf$-does not fork over $M_*$ for all $i < \delta$. For $i < \delta$, since both $\gtp(b/M_i, M_\delta)$ and $\gtp(b^*/M_i, \bigcup_{i<\delta}M_i^*)$ are extensions of $\gtp(b/M_0, M_0')$, and both $\dnf$-do not fork over $M_* \lek^u M_0 \lek \bigcup_{i < \delta} M_i$, by weak uniqueness, $\gtp(b/M_i, M_\delta) = \gtp(b^*/M_i, \bigcup_{i<\delta}M_i^*)$. Finally we get that $\gtp(b/\bigcup_{i<\delta}M_i, M_\delta) = \gtp(b^*/\bigcup_{i < \delta}M_i, \bigcup_{i<\delta}M_i^*)$ by $\Kkappalims$-universal continuity* in $\K$.

    Take $M_\delta^*$ a $(\lambda, \kappa)$-limit model over $\bigcup_{i<\delta}M^*_i$. By type equality, there exists $f:M_\delta \rightarrow M_\delta^*$ fixing $\bigcup_{i < \delta} M_i$ such that $f(b) = b^*$. Then take an isomorphism $g:M_\delta' \rightarrow M_\delta^*$ extending $f$, and let $M_i' = f^{-1}[M_i^*]$ for each $i < \delta$. Set $\calt' = \langle M_i:i < \delta + 1 \rangle ^\wedge \langle a_i : i < \delta \rangle$. We have $\calt^j \upharpoonright \delta \lesst \calt'$ by construction, and since we know $M_\delta \lek^u M_\delta'$, $\calt^j \lesst \calt'$ for all $j\leq \alpha$ (only conditions (1) and (2) of Definition \ref{tower_ordering_def} are relevant for the limit level $\delta$). As $\calt'$ is a universal tower, $\langle\calt^j : j < \alpha+1 \rangle^\wedge \calt'$ is brilliant as desired.
\end{proof}

The following says that the $\lesst$-chain of initial segments of towers from a reduced chain is reduced. Recall the notation $\calc \upharpoonright_* I_0$ from Definition \ref{restrict-chains-def}, the chain of all restrictions of towers in $\calc$ to their index's intersection with $I_0$.

\begin{lemma}\label{initial-segments-of-reduced-towers-are-reduced}
    If $\langle \calt^j : j < \alpha + 1 \rangle$  is reduced and brilliant where $\calt^\alpha = \langle M_i : i \in I \rangle^\wedge \langle a_i : i < I^{-2} \rangle$, and $I_0$ is an initial segment of $I$, then $\langle \calt^j : j < \alpha + 1 \rangle \upharpoonright_* I_0$ is reduced.
\end{lemma}

\begin{proof}
    Suppose $\Big(\langle \calt^j : j < \alpha + 1 \rangle \upharpoonright_* I_0 \Big)^\wedge \calt^*$ is brilliant where $\calt^* = \langle M_i^* : i \in I_0 \rangle^\wedge \langle a_i : i < (I_0)^{-2} \rangle$ is universal. By Proposition \ref{brilliant-extension-from-initial-segment-exist}, there exists $\calt'= \langle M_i':i \in I \rangle ^\wedge \langle a_i : i \in I^{-2} \rangle$ a universal tower such that $\langle \calt^j : j < \alpha + 1 \rangle ^\wedge \calt'$ is brilliant and a $\K$-embedding $g:\bigcup_{i < \beta} M_i' \underset{\bigcup_{i < \beta}M_i}{\rightarrow} \bigcup_{i < \beta} M_i^*$ such that $g[M_i'] = M_i^*$. As $\langle \calt^j : j < \alpha + 1 \rangle$ is reduced, for all $r < s < \alpha$, $M_s \cap M_r' = M_r$. In particular, this holds for $r < s < \beta$, which implies $g[M_s \cap M_r'] = g[M_r]$, which simplifies to $M_s \cap M_r^* = M_r$ as $g$ fixes $M_r, M_s$ and $g[M_r'] = M^*_r$. So $\langle \calt^j : j \leq \alpha + 1 \rangle \upharpoonright_* I_0$ is reduced, as desired.
\end{proof}

The following says that the $\lesst$-chain of end segments of a reduced chain is almost reduced - that is, they satisfy the reduced condition for all models in the final tower except possibly the bottom one. 

\begin{lemma}\label{reduced-end-segments-are-almost-reduced}
    Suppose $\langle \calt^j : j < \alpha + 1 \rangle$ is brilliant and reduced where $\calt^\alpha = \langle M_i : i \in I \rangle^\wedge \langle a_i : i \in I^{-2} \rangle$, and that $i_0 \in I$. Let $I_0 = \{i \in I : i_0 \leq_I i\}$ be the end segment from $i_0$. Suppose that $\calt' = \langle M_i' : i \in I_0 \rangle^\wedge \langle a_i : i \in (I_0)^{-2} \rangle$ is universal such that $(\langle \calt^j : j < \alpha + 1 \rangle \upharpoonright_* I_0) ^\wedge \calt'$ is brilliant. Then for any $r, s \in I_0 \setminus \{i_0\}$ where $r <_I s$, $M_r' \cap M_s = M_r$.
\end{lemma}

\begin{proof} 
    By relabelling the towers if necessary, assume $I = \beta$ and $i_0 = \xi$ for notational convenience (so $I_0 = [\xi, \beta)$). If $\xi + 1 < \beta$, the problem becomes trivial, so assume $\gamma + 1 < \beta$.

    Take some universal $\calt^* = \langle M_i^* : i < \beta + 2 \rangle^\wedge \langle a_i : i < \beta \rangle$ such that $\langle \calt^j : j < \alpha + 1 \rangle \upharpoonright_* (\xi + 2) \lesst \calt^*$ by Corollary \ref{universal-extensions-of-brilliant-chains-exist}.

    As $\calt^\alpha \upharpoonright [\xi, \beta) \lesst \calt^*$, $M_{\xi + 1} \lek^u M_{\xi+1}'$. As $M_{\xi + 1} \lek M_{\xi + 1}^*$, there exists a $\K$-embedding $f : M_{\xi + 1}^* \underset{M_{\xi + 1}}{\rightarrow} M_{\xi + 1}'$. For $i \leq \xi$, let $M_i'' = f[M_i^*]$, and for $i \in [\xi + 1, \beta)$, $M_i'' = M_i'$. Let $\calt'' = \langle M_i'' : i \in \beta \rangle^\wedge \langle a_i : i < \beta^{-2} \rangle$. 
    
    Say $\calt^j = \langle M_i^j : i \in I^j \rangle^\wedge \langle a_i : i \in (I^j)^{-2}\rangle$ for $j \leq \alpha$ (so $I^j \subseteq \beta$ for all $j < \alpha$). Note that $\gtp(a_i/M_{i+1}'', M_{i+2}'')$ $\dnf$-does not fork over $M_i^j$ for all $j \leq \alpha$ and $i \in (I^j)^{-2}$: firstly, if $i+1 \leq \xi$, we have $\gtp(a_i, M_{i+1}'', M_{i+2}'') = \gtp(f(a_i)/f[M_{i+1}^*], f[M_{i+2}^*]) = f(\gtp(a_i/M_{i+1}, M_{i+2}^*))$ $\dnf$-does not fork over $M_i^j = f[M_i^j]$ from $\calt^j\upharpoonright ((\xi + 2)\cap I^j) \lesst \calt^*$ and by invariance (when $i+1 = \xi$, the first equality follows from monotonicity as $f[M_{i+2}^*] = f[M_{\xi + 1}^*] \lek M_{\xi + 1}'' = M_{i+2}''$, and if $i+1 < \xi$, it follows from $f[M_{i+2}^*] = M_{i+2}''$ by definition). Secondly, if $i \geq \xi +1$, we have $\gtp(a_i/M_{i+1}'', M_{i+2}'') = \gtp(a_i/M_{i+1}', M_{i+2}')$ $\dnf$-does not fork over $M_i^j$ from $\calt \upharpoonright [\xi, \beta) \lesst\calt'$.
    
    In particular, note $\gtp(a_i/M_{i+1}'', M_{i+2}'')$ $\dnf$-does not fork over $M_i''$ for all $i < \alpha$ by monotonicity, so $\calt''$ is a tower. As $\calt^*$ and $\calt'$ are universal, $\calt''$ is also. And as $\gtp(a_i/M_{i+1}'', M_{i+2}'')$ $\dnf$-does not fork over $M_i^j$ for all $j < \alpha$ and $i \in (I^j)^{-2}$, $\calt^j \lesst \calt''$. Since $\langle \calt^j : j < \alpha + 1 \rangle$ is reduced, $\langle \calt^j : j < \alpha + 1 \rangle ^\wedge \calt''$ is brilliant, and $\calt''$ is universal, for all $r < s < \alpha$ we have that $M_r'' \cap M_s = M_r$. In particular, for $r < s$ both in $[\beta +1, \alpha)$, since $M_r'' = M_r'$, we have $M_r' \cap M_s = M_r$ as desired.
\end{proof}

\begin{proposition}\label{reduced-towers-are continuous-at-high-cofinality}
    If $\langle \calt^j : j < \alpha + 1 \rangle$ is brilliant and reduced where $\calt^\alpha = \langle M_i : i \in I\rangle^\wedge \langle M_i : i \in I^{-2} \rangle$ and $\delta \in I$ with $\cof_I(\delta) \geq \kappa$, then $\calt^\alpha$ is continuous at $\delta$.
\end{proposition}

\begin{proof}
    Suppose not for contradiction. By relabelling, it is enough to prove this in the case that $I = \beta$ is an ordinal. Let $\delta$ be the minimal ordinal with $\cof(\delta) \geq \kappa$ such that there exists $\beta > \delta$ and a tower $\calt = \langle M_i : i \in \beta \rangle^\wedge \langle M_i : i < \beta^{-2} \rangle$ discontinuous at $\delta \in \beta$. Fix the minimal such $\beta$ and a corresponding $\calt$ as above.

    Since $\calt$ is discontinuous at $\delta$, there is $b \in M_\delta \setminus \bigcup_{i < \delta} M_i$. Note that $\langle \calt^j : j \leq \alpha + 1 \rangle \upharpoonright_* (\delta + 1)$ has final tower discontinuous at $\delta$ also, and is reduced by Lemma \ref{initial-segments-of-reduced-towers-are-reduced}. Since $\beta$ is minimal, $\beta = \delta + 1$.

    Since $\langle \calt^j : j < \alpha + 1 \rangle$ is brilliant, by Corollary \ref{universal-extensions-of-brilliant-chains-exist} there is a universal tower $\calt' = \langle M_i' : i < \beta \rangle^\wedge \langle a_i : i < \beta^{-2} \rangle$ indexed by $\delta + 1$ where $\langle \calt^j : j < \alpha + 1\rangle ^\wedge \calt'$ is brilliant. Since $\cof(\delta) \geq \kappa$, by $(\geq \kappa)$-local character, there is $\gamma < \delta$ such that $\gtp(b/\bigcup_{i < \delta} M_i', M_\delta')$ $\dnf$-does not fork over $M_\gamma'$. Replacing $\gamma$ with $\gamma+1$ if needed, we may assume there exists a $(\lambda, \geq \kappa)$-limit model $M_*$ where $M_* \lek^u M_\gamma'$ and $\gtp(b/\bigcup_{i < \delta} M_i', M_\delta')$ $\dnf$-does not fork over $M_*$. By Lemma \ref{extensions-with-non-forking-element-at-the-bottom} applied to $\left((\langle \calt^j : j < \alpha + 1 \rangle^\wedge \calt') \upharpoonright^* [\gamma, \delta]\right)$, $b$, and $M_*$, there is a universal tower $\calt'' = \langle M_i'' : i \in [\gamma, \delta] \rangle^\wedge \langle a_i : i \in [\gamma, \delta) \rangle$ such that $\left((\langle \calt^j : j < \alpha + 1 \rangle^\wedge \calt') \upharpoonright^* [\gamma, \delta]\right)  ^\wedge \calt''$ is brilliant and $b \in M_\gamma''$.

    So $\left(\langle \calt^j : j < \alpha + 1 \rangle \upharpoonright^* [\gamma, \delta]\right) ^\wedge \calt''$ is brilliant also. Because $\langle \calt^j : j < \alpha + 1 \rangle$ is reduced, and $\calt''$ is universal, for all $r < s$ in $[\gamma+1, \alpha)$ we have $M_r'' \cap M_s = M_r$ by Lemma \ref{reduced-end-segments-are-almost-reduced}. Taking $r = \gamma + 1$ and $s = \delta$, we have $M_{\gamma + 1}'' \cap M_\delta = M_{\gamma + 1}$. But $b \in (M_{\gamma + 1}'' \cap M_\delta) \setminus M_{\gamma + 1}$, a contradiction.
\end{proof}

\subsection{The main theorem}

We restate and prove the main theorem. The strategy is to interweave sequences of full and reduced towers using the machinery we have built up. The union will then be reduced and full, and as discussed previously, by indexing our tower correctly, the top model of the final tower will be both a $(\lambda, 
\delta_1)$ limit model and a $(\lambda, \delta_2)$ limit model over the bottom model of the first tower.

\maintheorem*

\begin{proof}
    Firstly, by Remark \ref{assume-K-prime-is-Kkappalims}, we may assume $\K' = \Kkappalims$. By Fact \ref{cofinality-isomorphisms}, it is enough to show that there is a model which is both $(\lambda, \delta_1)$-limit over $M$ and $(\lambda, \delta_2)$-limit over $M$. By taking an appropriate extension, without loss of generality $M$ is a $(\lambda, \kappa)$-limit model. 

    Define inductively $<$-increasing continuous sequences of limit ordinals $\langle \alpha_k : k \leq \delta_1\rangle$ in $\lambda^+$, and  $\langle \beta_k : k \leq \delta_1\rangle$ in $\lambda^+$, and a brilliant chain towers $\langle \calt^j : j \leq \beta_{\delta_1}\rangle$, such that:
    \begin{enumerate}
        \item $\beta_0 = 0$, $\beta_{k+1}$ is successor for all $k < \delta_1$
        \item For all $k <\delta_1$ and all $j \in [\beta_k, \beta_{k+1})$, $\calt^j = \langle M^j_{(a, b, c)} : (a, b, c) \in (\delta_2 + 1) \times \lambda \times \alpha_k \rangle ^\wedge \langle a_{(a, b, c)} : (a, b, c) \in (\delta_2 + 1) \times \lambda \times \alpha_k \rangle$
        \item $M_{(0, 0, 0)} = M$
        \item For all $k < \delta_1$, $\langle \calt^j : j < \beta_{k+1}\rangle$ is reduced
        \item For all $k < \delta_1$, $\calt^{\beta_{k+1}}$ is universal and $(\delta_2 + 1) \times \lambda \times \{0\}$-full
        \item $\calt^{\beta_{\delta_1}} = \bigcup_{k < \delta_1} \calt^{\beta_k}$.
    \end{enumerate}
    \textbf{This is possible:} We construct by recursion on $k \leq \delta_1$ $\alpha_k, \beta_k$, and $\langle \calt^j : j < \beta_k \rangle$ with the structure as given in clause (2) of the construction. 
    
    Let $\alpha_0<\lambda^+$ be any limit, $\beta_0 = 0$. Let $\calt^0$ any $(\lambda, \kappa)$-tower indexed by $(\delta_2 + 1) \times \lambda \times \alpha_1$ as in clause (2) of the construction where $M_{(0, 0, 0)} = M$ (this exists by Lemma \ref{towers-exist}). 

    For the successor case, suppose we have $\alpha_k, \beta_k, \langle \calt^j : j \leq \beta_k \rangle$ for $k <\delta_1$. By Lemma \ref{reduced-extensions-exist}, there is some $\beta_{k+1} > 0$ and towers $\calt^j$ indexed by $(\delta_2 + 1) \times \lambda \times \alpha_k$ for $j \in (\beta_k, \beta_{k+1})$ such that $\langle \calt^j : j < \beta_{k+1}\rangle$ is reduced. Take $\alpha_{k+1} = \alpha_k + \lambda$. By Lemma \ref{full-extensions-exist}, there is $\calt^{\beta_{k+1}}$ indexed by $(\delta_2 + 1) \times \lambda \times \alpha_{k+1}$ such that $\langle \calt^j : j \leq \beta_{k+1}\rangle$ is brilliant and $\calt^{\beta_{k+1}}$ is $(\delta_2 + 1) \times \lambda \times \{0\}$-full.
    
    At limits $k < \delta_1$, take $\alpha_k = \bigcup_{j < k} \alpha_j$, $\beta_k = \bigcup_{j < k} \beta_j$ and $\calt^{\beta_k}$ such that $\langle \calt^j : j \leq \beta_k \rangle$ is brilliant by Corollary \ref{universal-extensions-of-brilliant-chains-exist}. 
    
    For $k = \delta_1$, take $\alpha_{\delta_1} = \bigcup_{j < \delta_1} \alpha_j$, $\beta_{\delta_1} = \bigcup_{j < \delta_1} \beta_j$, and $\calt^{\beta_{\delta_1}} = \bigcup_{j<\beta_{\delta_1}} \calt^j$. This completes the construction.

    \textbf{This is enough:} We will show that the model $M^{\beta_{\delta_1}}_{(\delta_2, 0, 0)}$ is both a $(\lambda, \delta_1)$-limit model over $M$ and a $(\lambda, \delta_2)$-limit model over $M$.

    First note that $\langle M^{\beta_j}_{(\delta_2, 0, 0)} : j \leq \delta_1 \rangle$ is a $\lek^u$ increasing sequence continuous at $\delta_1$, as the sequence $\langle \calt^{\beta_j} : j \leq \delta_1 \rangle$ is $\lesst$-increasing and continuous at $\delta_1$. So $M^{\beta_{\delta_1}}_{(\delta_2, 0, 0)}$ is a $(\lambda, \delta_1)$-limit model over $M^0_{(\delta_2, 0, 0)}$, and therefore a $(\lambda, \delta_1)$-limit model over $M$ as $M \lek M^0_{(\delta_2, 0, 0)}$.

    Next, we show that $\langle M^{\beta_{\delta_1}}_{(s, 0, 0)} : s \leq \delta_2 \rangle$ is also $\lek^u$-increasing and continuous at $(\delta_2, 0, 0)$. Note that as $\langle \calt^k : k < \beta_j\rangle$ is reduced for all $j<\delta_1$ and $\langle \calt^k : k \leq \beta_{\delta_1}\rangle$ is continuous at $\beta_{\delta_1}$, $\langle \calt^k : k \leq \beta_{\delta_1}\rangle$ is reduced by Lemma \ref{unions-of-reduced-towers-are-reduced}. Therefore, since $\cof_{(\delta_2 + 1) \times \lambda \times \alpha_{\delta_1}}(\delta_2, 0, 0) = \cof(\delta_2) \geq \kappa$, by Proposition \ref{reduced-towers-are continuous-at-high-cofinality}, $\calt^{\beta_{\delta_1}}$ is continuous at $(\delta_2, 0, 0)$. On the other hand, $\cof(\delta_1) \geq \kappa$ and $\calt^{\beta_{\delta_1}} = \bigcup_{k < \delta_1}\calt^{\beta_{k + 1}}$, so by Lemma \ref{long-unions-of-full-towers-are-full}, $\calt^{\delta_1}$ is $(\delta_2 + 1) \times \lambda \times \{0\}$-full. In particular, for all $a< \delta_2$ and $b < \lambda$, $M^{\beta_{\delta_1}}_{(a, b+1, 0)}$ realises all types over $M^{\beta_{\delta_1}}_{(a, b, 0)}$. So by Fact \ref{limit-models-exist}, for all $a < \delta_2$, $M^{\beta_{\delta_1}}_{(a+1, 0, 0)}$ is universal over $M^{\beta_{\delta_1}}_{(a, 0, 0)}$. Hence $M^{\beta_{\delta_1}}_{(\delta_2, 0, 0)} = \bigcup_{a < \delta_2} M^{\beta_{\delta_1}}_{(a, 0, 0)} $ is a $(\lambda, \delta_2)$-limit model over $M^{\beta_{\delta_1}}_{(0, 0, 0)}$, and therefore a $(\lambda, \delta_2)$-limit model over $M$ as $M \lek M^{\beta_{\delta_1}}_{(0, 0, 0)}$. So $M^{\beta_{\delta_1}}_{(\delta_2, 0, 0)}$ is as desired.

    The `moreover' statement now follows from Fact \ref{cofinality-isomorphisms}.
\end{proof}

\section{Applications}\label{applications-section}

In this section, we note a couple of easy applications of the main theorem. These are all known results, though in the case of Corollary \ref{bovan-main-corollary} particularly the proof was previously quite different.

Firstly, Theorem \ref{main-theorem} immediately gives us the main result from \cite{bovan}, as promised.

\begin{corollary}[{\cite[Theorem 1.2]{bovan}}]\label{bovan-main-corollary}
    Let $\K$ be an AEC stable in $\lambda \geq \LS(\K)$ with amalgamation, joint embedding, and no maximal models in $\K_\lambda$. Assume $\lambda$-non-splitting satisfies universal continuity,  $(\geq \kappa)$-local character for some $\kappa < \lambda^+$, and $(\lambda, \theta)$-symmetry for some regular $\theta \in [\kappa, \lambda^+)$.

    Then for any $M, N_1, N_2 \in \K_\lambda$ where $N_l$ is a $(\lambda, \geq\kappa)$-limit model over $M$ for $l = 1, 2$, $N_1 \underset{M}{\cong} N_2$. Moreover, for any $N_1, N_2 \in \K_\lambda$ both $(\lambda, \geq\kappa)$-limit models, $N_1 \cong N_2$.
\end{corollary}

\begin{proof}
    $\lambda$-non-splitting satisfies invariance, monotonicity, base monotonicity, weak uniqueness, and weak extension by Fact \ref{non-splitting-properties}. We have assumed universal continuity and $(\geq \kappa)$-local character, and $(\lambda, \theta)$-weak non-forking amalgamation follows from $(\lambda, \theta)$-symmetry and Lemma \ref{symmetry_implies_nfap}. So $\K$ with $\lambda$-non-splitting satisfies Hypothesis \ref{main_hypothesis}, and the statement now follows from Theorem \ref{main-theorem}.
\end{proof}

We may also reobtain the main theorem of \cite[\textsection 3]{bema} for free here. Recall the definition of non-forking amalgamation (Definition \ref{def-dnf-properties}(\ref{def-nfap})).

\begin{lemma}\label{strong-properties-give-weak-lemma}
	Let $\K$ be an AEC stable in $\lambda \geq \LS(\K)$ with amalgamation, joint embedding, and no maximal models in $\K_\lambda$. Assume $\dnf$ is an independence relation on an AC $\K'$ where $\Kkappalims \subseteq \K' \subseteq \K_\lambda$ satisfying uniqueness, extension, and non-forking amalgamation.
	
	Then $\dnf$ satisfies weak uniqueness and weak extension. Moreover, for all $\theta \in [\kappa, \lambda^+)$, $\dnf$ satisfies $(\lambda, \theta)$-weak non-forking amalgamation.
\end{lemma}

\begin{proof}
	Since uniqueness and extension are clearly stronger forms of weak uniqueness and weak extension respectively, it remains only to show that $\dnf$ satisfies weak $(\lambda, \theta)$-non-forking amalgamation for all regular $\theta \in [\kappa, \lambda^+)$.
	
	Take $\theta \in [\kappa, \lambda^+)$ regular. Suppose $M_0, M, M_1, M_2 \in \K_\lambda$ where $M$ is a $(\lambda, \theta)$-limit model over $M_0$, and $M \lek M_l$ and $a_l \in M_l$ for $l = 1, 2$ such that $\gtp(a_l /M, M_l)$ $\dnf$-does not fork over $M_0$. By monotonicity, $\gtp(a_l /M, M_l)$ $\dnf$-does not fork over $M$. So by non-forking amalgamation, there exist $N \in \K_\lambda$ and $f_l : M_l \rightarrow N$ fixing $M$ for $l = 1, 2$ such that $\gtp(f_l(a_l)/f_{3-l}[M_{3-l}], N)$ $\dnf$-does not fork over $M$ for $l = 1, 2$. Note that by invariance $\gtp(f_l(a_l)/M, N)$ $\dnf$-does not fork over $M_0$. Since transitivity follows from uniqueness and extension (see for example \cite[Fact 2.16]{bema}), we may apply transitivity to get that $\gtp(f_l(a_l)/f_{3-l}[M_{3-l}], N)$ $\dnf$-does not fork over $M_0$ for $l =1, 2$, as desired.
\end{proof}

\begin{corollary}\label{bema-weaker-corollary}
    Let $\K$ be an AEC stable in $\lambda \geq \LS(\K)$ with amalgamation, joint embedding, and no maximal models in $\K_\lambda$. Assume $\dnf$ is an independence relation on an AC $\K'$ where $\Kkappalims \subseteq \K' \subseteq \K_\lambda$ satisfying uniqueness, extension, $\Kkappalims$-universal continuity* in $\K$, $(\geq \kappa)$-local character for some $\kappa < \lambda^+$, and non-forking amalgamation.

    Then for any $M, N_1, N_2 \in \K_\lambda$ where $N_l$ is a $(\lambda, \geq\kappa)$-limit model over $M$ for $l = 1, 2$, $N_1 \underset{M}{\cong} N_2$. Moreover, for any $N_1, N_2 \in \K_\lambda$ both $(\lambda, \geq\kappa)$-limit models, $N_1 \cong N_2$.
\end{corollary}

\begin{proof}   
    By Lemma \ref{strong-properties-give-weak-lemma}, $\dnf$ satisfies weak uniqueness, weak extension, and $(\lambda, \kappa)$-weak non-forking amalgamation. Thus we may apply Theorem \ref{main-theorem}, which gives the desired result.
\end{proof}

\begin{remark}
	The version of non-forking amalgamation in \cite{bema} may appear stronger than Definition \ref{def-dnf-properties}(\ref{def-nfap}) (it does not assume $\gtp(a_l/M, M_l)$ $\dnf$-does not fork over $M$ for $l = 1, 2$), but they are equivalent under \cite[Hypothesis 3.7]{bema} when $\dnf$ is defined only on $\Kkappalims$ as in that case existence holds \cite[Lemma 3.10]{bema}.
\end{remark}

\begin{remark}
    This gives us all the `high limits are isomorphic' applications from \cite{bema}: first order stable theories (\cite[Lemma 6.11]{bema}); independence relations in the sense of \cite{lrv} (\cite[Lemma 6.1, Theorem 5.1]{bema}); tame AECs with monster models and symmetry (\cite[$\mathsection$ 3.4]{bema}, \cite[Corollary 3.46]{bema}); modules of rings with pure embeddings (\cite[$\mathsection$ 6.2]{bema}); and torsion Abelian groups (\cite[Lemma 6.14]{bema}). Note we have only proved the `high limits are isomorphic' part of these results - for those that fully categorise the spectrum of limit models, \cite[$\mathsection$ 4]{bema} shows how to prove `low limits are non-isomorphic' in those cases as well.
\end{remark}

\printbibliography

\end{document}